\documentclass[a4paper, 11pt, reqno]{amsart}
\usepackage[utf8]{inputenc}
\usepackage[usenames,dvipsnames]{color}
\usepackage{amssymb, amsmath, amsfonts, amsthm, latexsym, graphics, graphicx,ytableau, mathtools, enumerate, booktabs, pifont}
\usepackage[linesnumbered,ruled,vlined]{algorithm2e}
    \DontPrintSemicolon
    \SetAlgoSkip{medskip}

\usepackage[hidelinks]{hyperref}

\makeatletter
\def\blfootnote{\xdef\@thefnmark{}\@footnotetext}
\makeatother

\newtheorem{theorem}{Theorem}[section]
\newtheorem{corollary}[theorem]{Corollary}
\newtheorem{lemma}[theorem]{Lemma}
\newtheorem{proposition}[theorem]{Proposition}

\theoremstyle{definition}
\newtheorem{definition}[theorem]{Definition}
\newtheorem{example}[theorem]{Example}
\theoremstyle{remark}
\newtheorem{remark}[theorem]{Remark}

\newcommand{\ol}{\overline}

\DeclarePairedDelimiter{\parens}{\lparen}{\rparen}

\DeclarePairedDelimiterX{\pres}[2]{\langle}{\rangle}{#1\,\delimsize\vert\,\mathopen{}#2}

\newcommand*{\cont}[1]{\mathrm{cont}\parens{#1}}
\newcommand*{\supp}[1]{\mathrm{supp}\parens{#1}}
\newcommand*{\std}[1]{\mathrm{Std}\parens{#1}}
\newcommand*{\rstd}[1]{\mathrm{rStd}\parens{#1}}
\newcommand*{\reading}[1]{\mathrm{read}\parens{#1}}

\newcommand{\N}{\mathbb{N}}
\newcommand{\cR}{\mathcal{R}}
\newcommand{\cL}{\mathcal{L}}
\newcommand{\cB}{\mathcal{B}}
\newcommand{\X}{\mathcal{X}}
\newcommand{\uord}{\mathbf{u}}
\newcommand{\vord}{\mathbf{v}}
\newcommand{\word}{\mathbf{w}}

\newcommand{\rord}{\mathbf{r}}
\newcommand{\sord}{\mathbf{s}}

\newcommand*{\plac}{{\mathsf{plac}}}

\newcommand*{\hypo}{{\mathsf{hypo}}}

\newcommand*{\sylv}{{\mathsf{sylv}}}

\newcommand*{\sylvh}{{\mathsf{sylv}^{\#}}}

\newcommand*{\baxt}{{\mathsf{baxt}}}

\newcommand*{\rst}{{\mathsf{rSt}}}

\newcommand*{\lst}{{\mathsf{lSt}}}

\newcommand*{\mst}{{\mathsf{mSt}}}

\newcommand*{\jst}{{\mathsf{jSt}}}

\newcommand*{\rtg}{{\mathsf{rTg}}}

\newcommand*{\ltg}{{\mathsf{lTg}}}

\newcommand*{\mtg}{{\mathsf{mTg}}}

\newcommand*{\jtg}{{\mathsf{jTg}}}

\newcommand*{\plr}{\mathrm{P}^\rightarrow}
\newcommand*{\prl}{\mathrm{P}^\leftarrow}
\newcommand*{\psymb}{\mathrm{P}}

\newcommand*{\qlr}{\mathrm{Q}^\rightarrow}
\newcommand*{\qrl}{\mathrm{Q}^\leftarrow}
\newcommand*{\qsymb}{\mathrm{Q}}

\newcommand*{\prst}[1]{\prl_{\rst}\parens{#1}}
\newcommand*{\plst}[1]{\plr_{\lst}\parens{#1}}
\newcommand*{\pmst}[1]{\psymb_{\mst}\parens{#1}}
\newcommand*{\pltg}[1]{\plr_{\ltg}\parens{#1}}
\newcommand*{\prtg}[1]{\prl_{\rtg}\parens{#1}}
\newcommand*{\pmtg}[1]{\psymb_{\mtg}\parens{#1}}
\newcommand*{\psylv}[1]{\prl_{\sylv}\parens{#1}}

\newcommand*{\qrst}[1]{\qrl_{\rst}\parens{#1}}
\newcommand*{\qlst}[1]{\qlr_{\lst}\parens{#1}}
\newcommand*{\qmst}[1]{\qsymb_{\mst}\parens{#1}}
\newcommand*{\qltg}[1]{\qlr_{\ltg}\parens{#1}}
\newcommand*{\qrtg}[1]{\qrl_{\rtg}\parens{#1}}
\newcommand*{\qmtg}[1]{\qsymb_{\mtg}\parens{#1}}

\newcommand*{\shrst}[1]{\mathrm{Sh}_{\rst}\parens{#1}}
\newcommand*{\shlst}[1]{\mathrm{Sh}_{\lst}\parens{#1}}
\newcommand*{\shmst}[1]{\mathrm{Sh}_{\mst}\parens{#1}}

\newcommand*{\shrtg}[1]{\mathrm{Sh}_{\rtg}\parens{#1}}

\newcommand*{\shsylv}[1]{\mathrm{Sh}_{\sylv}\parens{#1}}

\newcommand*{\shbaxt}[1]{\mathrm{Sh}_{\baxt}\parens{#1}}

\newcommand*{\vhypo}{\mathbb{V}(\hypo)}

\newcommand*{\vbaxt}{\mathbb{V}(\baxt)}
\newcommand*{\vrst}{\mathbb{V}(\rst)}
\newcommand*{\vlst}{\mathbb{V}(\lst)}
\newcommand*{\vmst}{\mathbb{V}(\mst)}
\newcommand*{\vjst}{\mathbb{V}(\jst)}

\usepackage{tikz}
\usetikzlibrary{babel,cd,arrows.meta,calc,decorations.pathmorphing,decorations.text,matrix,positioning,shapes.geometric,shapes.misc,shapes.symbols}

\usetikzlibrary{matrix}

\tikzset{
	pretableaumatrix/.style={
		ampersand replacement=\&,
		matrix of math nodes,
		outer sep=1mm,
		inner sep=0mm,
		anchor=center,
		row sep={between borders,-\pgflinewidth},
		column sep={between borders,-\pgflinewidth},
		dottedentry/.style={densely dotted},
		spaceentry/.style={draw=none,execute at begin node=\null},
	},
	pretableaunode/.style={
		font=\small,
		draw=gray,
		sharp corners,
		rectangle,
		anchor=base,
		text height=3.75mm,
		text depth=1.25mm,
		minimum height=5mm,
		minimum width=5mm,
		inner sep=0mm,
		outer sep=0mm,
	},
	tableaumatrix/.style={
		pretableaumatrix,
		every node/.append style={
			pretableaunode,
		},
	},
	medtableaumatrix/.style={
		pretableaumatrix,
		every node/.append style={
			pretableaunode,
			font=\footnotesize,
			text height=2.75mm,
			text depth=.75mm,
			minimum height=3.5mm,
			minimum width=3.5mm
		},
	},
	smalltableaumatrix/.style={
		pretableaumatrix,
		every node/.append style={
			pretableaunode,
			font=\scriptsize,
			text height=1.85mm,
			text depth=.15mm,
			minimum height=2.5mm,
			minimum width=2.5mm,
		},
	},
	tinytableaumatrix/.style={
		pretableaumatrix,
		every node/.append style={
			pretableaunode,
			font=\tiny,
			text height=1.25mm,
			text depth=.15mm,
			minimum height=1.75mm,
			minimum width=1.75mm
		},
	},
	tableau/.style={
		baseline=-1.25mm,
		every matrix/.style={tableaumatrix},
	},
	medtableau/.style={
		baseline=-1.25mm,
		every matrix/.style={medtableaumatrix},
	},
	smalltableau/.style={
		baseline=-1.25mm,
		every matrix/.style={smalltableaumatrix},
	},
	preshapetableaumatrix/.style={
		pretableaumatrix,
		execute at end cell={\strut},
		every node/.append style={
			draw=black,
			anchor=base,
			inner sep=0mm,
			outer sep=0mm,
		},
		shadedentry/.style={fill=gray},
		darkshadedentry/.style={fill=darkgray},
	},
	medshapetableaumatrix/.style={
		preshapetableaumatrix,
		every node/.append style={
			text height=2.75mm,
			text depth=.75mm,
			minimum height=3.5mm,
			minimum width=3.5mm
		},
	},
	shapetableaumatrix/.style={
		ampersand replacement=\&,
		matrix of math nodes,
		outer sep=0mm,
		inner sep=0mm,
		anchor=base,
		row sep={between borders,-\pgflinewidth},
		column sep={between borders,-\pgflinewidth},
		execute at begin cell={\strut},
		every node/.append style={draw,anchor=base,text height=1mm,text depth=.5mm,minimum size=1.5mm,inner sep=0mm,outer sep=0mm},
	},
	shapetableau/.style={
		every matrix/.style={shapetableaumatrix},
	},
	topalign/.style={
		every matrix/.append style={name=maintableau,anchor=maintableau-1-1.base},
		baseline,
	},
}

\newcommand*\tableau[2][]{\tikz[tableau,#1]\matrix{#2};}

\usetikzlibrary{shapes.geometric} 
\usetikzlibrary{shapes.misc}  

\tikzset{
	bst/.style={
		standard/.style={
			font=\small,
			draw=gray,
			rounded rectangle,
			minimum width=4.5mm,
			minimum height=4.5mm,
			inner xsep=0mm,
			inner ysep=1mm,
			outer sep=0mm,
			line width=.5pt,
		},
		empty/.style={
			minimum width=3mm,
			minimum height=3mm,
		},
		triangle/.style={
			isosceles triangle,
			isosceles triangle apex angle=60,
			shape border rotate=90,
			rounded corners=2mm,
			minimum width=8mm,
			inner xsep=0mm,
			inner ysep=.5mm
		},
		blank/.style={
			draw=none,
		},
		nodecount/.style={
			blank,
			font=\scriptsize,
		},
		every node/.style={standard},
		every child/.style={draw=black,line width=.6pt},
		level distance=10mm,
		level 1/.style={sibling distance=60mm},
		level 2/.style={sibling distance=30mm},
		level 3/.style={sibling distance=15mm},
	},
	medbst/.style={
		bst,
		level distance=10mm,
		level 1/.style={sibling distance=15mm},
		level 2/.style={sibling distance=15mm},
		level 3/.style={sibling distance=15mm},
	},
	smallbst/.style={
		bst,
		level distance=8mm,
		level 1/.style={sibling distance=10mm},
		level 2/.style={sibling distance=10mm},
		level 3/.style={sibling distance=10mm},
	},
	tinybst/.style={
		bst,
		level distance=5mm,
		level 1/.style={sibling distance=8mm},
		level 2/.style={sibling distance=8mm},
		level 3/.style={sibling distance=8mm},
		every node/.append style={
			font=\footnotesize,
		},
		triangle/.append style={
			rounded corners=1mm,
			minimum width=7mm,
			inner xsep=-.5mm,
		},
	},
	microbst/.style={
		bst,
		standard/.append style={
			font=\scriptsize,
			minimum width=3mm,
			minimum height=3mm,
			inner ysep=1.5mm,
		},
		level distance=5mm,
		level 1/.style={sibling distance=12mm},
		level 2/.style={sibling distance=12mm},
		level 3/.style={sibling distance=12mm},
	},
	nanobst/.style={
		bst,
		standard/.append style={
			font=\tiny,
			minimum width=2mm,
			minimum height=2mm,
			inner ysep=.25mm,
		},
		level distance=2mm,
		level 1/.style={sibling distance=4mm},
		level 2/.style={sibling distance=4mm},
		level 3/.style={sibling distance=4mm},
	},
}
\title[Meet and join plactic-like monoids]{Plactic-like monoids arising from meets and joins of stalactic and taiga congruences}

\date{September 18, 2023}

\author[Thomas Aird]{Thomas Aird}
\address[Thomas Aird]{Department of Mathematics, The University of Manchester, Alan Turing Building, Oxford Rd, Manchester, M13 9PL, UK.}
\email{thomas.aird@manchester.ac.uk}
\thanks{The first author's work was supported by the London Mathematical Society, the Heilbronn Institute for Mathematical Research, and NOVA University Lisbon.}

\author[Duarte Ribeiro]{Duarte Ribeiro}
\address[Duarte Ribeiro]{Center for Mathematics and Applications (NOVA Math), FCT NOVA, 2829-516 Caparica, Portugal.}
\email{dc.ribeiro@campus.fct.unl.pt}
\thanks{The second author's work was supported by National Funds through the FCT -- Funda\c{c}\~{a}o para a Ci\^{e}ncia e a Tecnologia, I.P., under the scope of the projects PTDC/MAT-PUR/31174/2017, UIDB/04621/2020 and UIDP/04621/2020 (Center for Computational and Stochastic Mathematics), and UIDB/00297/2020 and UIDP/00297/2020 (Center for Mathematics and Applications).}

\begin{document}
	
\begin{abstract}
	We study the four plactic-like monoids that arise by taking the meets and joins of stalactic and taiga congruences. We obtain the combinatorial objects associated with the meet monoids, establishing Robinson--Schensted-like correspondences and giving extraction and iterative insertion algorithms for these objects. We then obtain results on the sizes of classes of words equal in plactic-like monoids, show that some of these monoids are syntactic, and characterise their equational theories.
\end{abstract}


\maketitle


\section{Introduction} \label{section:introduction}

Plactic-like monoids are quotients of the free monoid sharing the interesting property that their elements are uniquely identified with combinatorial objects, thus giving these monoids a fundamental role in algebraic combinatorics. The namesake of this family is the plactic monoid $\plac$, also known as the monoid of Young tableaux \cite{lothaire_2002}. Defined by Lascoux and Schützenberger \cite{LS1978}, it has been widely applied in several areas of mathematics such as representation theory \cite{green2006polynomial}, symmetric functions \cite{macdonald_symmetric} and crystal bases \cite{bump_crystalbases}. In the last decade, the equational theories of the plactic monoids have received great attention \cite{aird2023semigroup,ckkmo_placticidentity,izhakian_cloaktic,kubat_identities}. In particular, Johnson and Kambites \cite{johnson_kambites_tropical_plactic} gave faithful tropical representations of finite-rank plactic monoids, thus showing that they all satisfy non-trivial identities. 

Other plactic-like monoids arise in the context of combinatorial Hopf algebras whose bases are indexed by combinatorial objects. In particular, the hypoplactic monoid $\hypo$ \cite{novelli_hypoplactic}, the sylvester monoid $\sylv$ \cite{hivert_sylvester}, the Baxter monoid $\baxt$ \cite{giraudo_baxter}, the (left) stalactic monoid $\lst$ \cite{hnt_stalactic} and the (right) taiga monoid $\rtg$ \cite{priez_binary_trees} have recently been studied with regards to their connections with crystal graphs \cite{cm_sylv_crystal,cm_hypo_crystal,cain_gray_malheiro_rewriting}, and their equational theories \cite{cain_johnson_kambites_malheiro_representations_2022,cm_identities,cain_malheiro_ribeiro_sylvester_baxter_2023,cain_malheiro_ribeiro_hypoplactic_2022,han2021preprint}.

The Baxter monoid was introduced by Giraudo \cite{giraudo_baxter} as the analogue of the plactic monoid for Hopf algebras indexed by Baxter permutations. Its corresponding congruence was shown to be the meet of the sylvester congruence and its dual. Thus, its combinatorial objects are pairs of twin binary search trees, that is, pairs of combinatorial objects of the sylvester monoid and its dual that satisfy certain properties. From this connection, a Robinson--Schensted-like correspondence is established, along with several combinatorial properties.

On the other hand, the (right) taiga monoid was defined by Priez \cite{priez_binary_trees} as the quotient of the free monoid by the join of the sylvester and (right) stalactic congruences. Its associated combinatorial objects are binary search trees with multiplicities, and a Robinson--Schensted-like correspondence and `hook-length'-like formula are given.

In this paper, we introduce four plactic-like monoids by factoring the free monoid over an ordered alphabet by the meets and joins of stalactic and taiga congruences, in the vein of Giraudo and Priez. It is organised as follows: In Section~\ref{section:preliminaries}, we first give the necessary background on words. Then, for completeness, we present a detailed description of the combinatorial objects of the stalactic and taiga monoids, giving formal proofs of Robinson--Schensted-like correspondences. We end this section with left and right-oriented definitions of the stalactic and taiga monoids. We define the meet and join plactic-like monoids in Section~\ref{section:meets_and_joins_of_stalactic_and_taiga_congruences}, for which we give presentations, characterise when words are equal in them, and study their compatibility properties.

Section~\ref{section:Robinson_Schensted_like_correspondences} focuses on the combinatorial objects associated with the meet-stalactic and meet-taiga monoids. We define $\psymb$-symbols by taking pairs of stalactic tableaux and binary search trees with multiplicities, respectively, which satisfy certain properties. We similarly define $\qsymb$-symbols and establish Robinson--Schensted-like correspondences. These symbols are analogues of semistandard Young tableaux and recording tableaux, respectively, in the case of the plactic monoid. We then give algorithms to extract words from $\psymb$-symbols. Finally, we give iterative versions of the insertion algorithms, allowing us to easily compute the concatenation of words under these congruences.

In Section~\ref{section:Counting_in_the_stalactic_and_taiga_monoids}, we provide explicit formulas for the sizes of congruence classes in stalactic and taiga monoids, with the exception of the meet-stalactic and meet-taiga monoids. For these, we show that computing the size of congruence classes is equivalent to counting the number of linear extensions in specific posets, for which we give a recursive formula. Then, we provide a `hook-length'-like formula to compute the number of stalactic tableaux which form a pair of twin stalactic tableaux with a given one, and give bounds for the corresponding number for binary search trees with multiplicities.

We study the syntacticity of plactic-like monoids in Section~\ref{section:Syntacticity_of_plactic_like_monoids}, with respect to the shape functions of combinatorial objects. Finally, in Section~\ref{section:Equational_theories_of_meet_and_join_stalactic_and_taiga_monoids}, we give a characterisation of the identities satisfied by the meet and join plactic-like monoids, which can be verified in polynomial time. We then produce a finite basis for each of these monoids and compute their axiomatic ranks.


\section{Preliminaries} \label{section:preliminaries}

We first recall the necessary definitions and notations. For the necessary background on semigroups and monoids, see \cite{howie1995fundamentals}; for presentations, see \cite{higgins1992techniques}. Let $\N = \{1,2,\dots\}$ denote the set of natural numbers, without zero. For $a,b \in \N$, we denote by $[a]$ the set $\{1 < \cdots < a\}$, and by $[[a,b]]$ the set $\{ k \in \N \colon \min{(a,b)} \leq k \leq \max{(a,b)} \}$, or simply $[a,b]$ when $a<b$. Permutations are written in one-line notation, so they can be viewed as words of length $n$ over $[n]$, with no repetitions. We denote the inverse of a permutation $\sigma$ of $[n]$ by $\sigma^{-1}$.


\subsection{Words} \label{subsection:words}

For any non-empty set $\X$, we denote by $\X^*$ the free monoid generated by $\X$, that is, the set of all words over $\X$ under concatenation. We denote the empty word by $\varepsilon$, and refer to $\X$ as an \emph{alphabet}, and to its elements as \emph{letters}. For a word $\word \in \X^*$, we denote its length by $|\word|$ and, for each $x \in \X$, we denote the number of occurrences of $x$ in $\word$ by $|\word|_x$. If $|\word|_x = 1$, we say $x$ is a \emph{simple letter} of $\word$, and we denote the restriction of $\word$ to its simple letters by $\ol{\word}$. For words $\uord,\vord \in \X^*$ we say that $\uord$ is a \emph{factor} of $\vord$ if there exist $\vord_1,\vord_{2} \in \X^*$ such that $\vord = \vord_1 \uord \vord_2$. 

The subset of $\X$ of letters that occur in $\word$ is called the \emph{support} of $\word$, denoted by $\supp{\word}$, and the function from $\X$ to $\N_0$ given by $x \mapsto |\word|_x$ is called the \emph{content} of $\word$, denoted by $\cont{\word}$. Clearly, two words that share the same content also share the same support.

We denote by $\overrightarrow{\rho_\word}$ (resp. $\overleftarrow{\rho_\word}$) the bijection from $\supp{\word}$ to $[|\supp{\word}|]$ mapping each letter $x \in \supp{\word}$ to its position in the ordering of $\supp{\word}$ according to its first (resp. last) occurrence, when reading $\word$ from left-to-right. One can extend this bijection to an isomorphism from $\supp{\word}^*$ to $[|\supp{\word}|]^*$.

\begin{remark}
    The definition of $\overleftarrow{\rho_\word}$ might be slightly counterintuitive since we order by last occurrences from left-to-right instead of first occurrences from right-to-left, but this definition greatly simplifies the notation and arguments used throughout this paper. 
\end{remark}

For a word $\word$ over a totally ordered alphabet $\X$, the \emph{standardised word} of $\word$, denoted by $\std{\word}$, is the unique permutation of size $|\word|$ given by reading $\word$ from left-to-right, attaching the subscript $i$ to the $i$-th occurrence of a letter, then replacing each letter by its rank with regards to the lexicographic order
\[
    a_i < b_j \Leftrightarrow (a < b) \vee ((a = b) \wedge (i < j)).
\]
For $\supp{\word} = \{ a_1 < \cdots < a_k\}$, the \emph{packed} word of $\word$ is the unique word obtained by replacing each $a_i$ in $\word$ with $i$. The \emph{restriction} of $\word$ to a subset $\mathcal{S}$ of $\X$, denoted by $\word[S]$, is the word obtained from $\word$ by removing any letter not in $\mathcal{S}$. The \emph{reverse} of $\word$ is the word whose $i$-th letter is the $(|\word|+1-i)$-th letter of $\word$. For an alphabet $\X$ with $n$ letters, the \emph{Schützenberger involution} is the unique antiautomorphism of $\X^*$ induced by the map sending the $i$-th letter of $\X$ to its $(n+1-i)$-th letter.

We now define a variation of the standardised word: For a word $\word$ over a totally ordered alphabet $\X$, the \emph{reverse-standardised word} of $\word$, denoted by $\rstd{\word}$, is the unique permutation of size $|\word|$ given by reading $\word$ from left-to-right, attaching the subscript $i$ to the $i$-th occurrence of a letter, then replacing each letter by its rank with regards to the lexicographic order
\[
    a_i < b_j \Leftrightarrow (a < b) \vee ((a = b) \wedge (i > j)).
\]
Notice that, if two letters in a word are different, they are ordered in the same way as when computing the standardised word, but if they are the same, the order is reversed. This will be useful to distinguish strictly decreasing sequences.

Let $\X$ be an alphabet and let $\equiv$ be a congruence on $\X^*$. We say the factor monoid $X^*/{\equiv}$ is compatible with:
\begin{itemize}
    \item \emph{standardisation} when $\uord$ and $\vord$ are congruent if and only if they share the same content and their standardised words are congruent, for any $\uord, \vord \in \X^*$;
    \item \emph{reverse-standardisation} when $\uord$ and $\vord$ are congruent if and only if they share the same content and their reverse-standardised words are congruent, for any $\uord, \vord \in \X^*$;
    \item \emph{packing} when $\uord$ and $\vord$ are congruent if and only if they share the same content and their packed words are congruent, for any $\uord, \vord \in \X^*$;
    \item \emph{restriction to alphabet subsets} when $\uord$ and $\vord$ are congruent only if their restrictions to $\mathcal{S}$ are congruent, for any $\uord, \vord \in \X^*$ and any subset $\mathcal{S}$ of $\X$;
    \item \emph{restriction to alphabet intervals} when $\uord$ and $\vord$ are congruent only if their restrictions to $\mathcal{I}$ are congruent, for any $\uord, \vord \in \X^*$ and any interval $\mathcal{I}$ of $\X$;
    \item \emph{word reversals} when $\uord$ and $\vord$ are congruent if and only if their reverses are congruent, for any $\uord, \vord \in \X^*$;
    \item the \emph{Schützenberger involution} when $|\X| \in \N$, and $\uord$ and $\vord$ are congruent if and only if their images under the Schützenberger involution are congruent, for any $\uord, \vord \in \X^*$.
\end{itemize}


\subsection{Combinatorial objects and insertion algorithms} \label{subsection:combinatorial_objects_and_insertion_algorithms}

We now give the necessary background on the combinatorial objects needed to define the stalactic and taiga monoids. For completeness, we give a full proof of the Robinson--Schensted-like correspondences for these objects, which we will require for the following sections. For background on right strict, left strict and pairs of twin binary search trees, see \cite{giraudo_baxter}. We denote any empty combinatorial object by $\perp$. We introduce unlabelled and labelled combinatorial objects, and we refer to the underlying unlabelled objects of labelled objects as \emph{shapes}.


\subsubsection{Composition diagrams, stalactic tableaux and patience-sorting tableaux} \label{subsubsection:composition_diagrams_stalactic_tableaux_and_patience_sorting_tableaux}

A \emph{composition diagram} is a top-aligned finite collection of cells, with no order imposed on the length of the columns. A \emph{stalactic tableau} is an $\N$-labelled composition diagram where all cells in a column share the same label, and no two columns share the same label. An example of a stalactic tableau is the following:
\begin{equation*}
    \label{example:stalactic_tableau}
    \tikz[tableau]\matrix{
		2 \& 1 \& 3 \& 5 \& 4\\
		2 \& 1 \&   \& 5 \&  \\
          \& 1 \&   \&   \&  \\ 
	};
\end{equation*}
We say that a column in a composition diagram or stalactic tableau is \emph{simple} if it only has one cell. A remark about an abuse of language: we refer to a column whose cells are labelled by some letter $a$ by `the column labelled $a$', if the context is clear.
The subset of $\N$ of letters that label columns of a stalactic tableau $T$ is called the \emph{support} of $T$, denoted by $\supp{T}$. We define $\rho_T \colon \supp{T} \to [|\supp{T}|]$ as the function which gives the position of a letter in the top row of $T$, when reading it from left-to-right. If the context is clear, we denote by $\rho_T$ the top row of $T$. 

Algorithm \hyperref[alg:RightStalacticLeftInsertion]{\textsf{rStLI}} allows one to insert a letter into a stalactic tableau, either at the bottom of a previously existing column, or by creating a new column to the left of the tableau:

\begin{algorithm}[H] \label{alg:RightStalacticLeftInsertion}
        \DontPrintSemicolon
        \KwIn{A stalactic tableau $T$, a letter $a \in \N$.}
        \KwOut{A stalactic tableau $a \rightarrow T$.}
        \BlankLine
        \eIf{$a$ does not label any cell in $T$}{add a new cell labelled by $a$ to the left of the top row of $T$;}{add a new cell labelled by $a$ to the bottom of the unique column with cells labelled by $a$;}
        \Return the resulting stalactic tableau $a \rightarrow T$.
        \caption{\textit{Right-Stalactic Left-Insertion} (\textsf{rStLI}).}
\end{algorithm}
We can define a \textit{Left-Stalactic Right-Insertion} (\hyperref[alg:RightStalacticLeftInsertion]{\textsf{lStRI}}) algorithm by replacing ``right'' with ``left'' and ``left'' with ``right'', respectively, in the previous algorithm. 

Using algorithm \hyperref[alg:RightStalacticLeftInsertion]{\textsf{rStLI}} (resp. \hyperref[alg:RightStalacticLeftInsertion]{\textsf{lStRI}}), one can compute a unique stalactic tableau from a word $\word \in \N^*$: Starting from the empty tableau, read $\word$ from right-to-left (resp. left-to-right) and insert its letters one-by-one into the tableau. The resulting tableau is denoted by $\prst{\word}$ (resp. $\plst{\word}$).

\begin{example}
	\label{example:insertion_alg_prst}
    Computing $\prst{212511354}$:
    \begin{gather*}  
		\tableau{
			4 \\
		}
		\quad\xleftarrow{5}\quad 
		\tableau{
			5 \& 4 \\
		}
		\quad\xleftarrow{3}\quad 
		\tableau{
			3 \& 5 \& 4 \\
		}
		\quad\xleftarrow{1}\quad 
		\tableau{
			1 \& 3 \& 5 \& 4 \\
		}
		\quad\xleftarrow{1}\\[10pt]
		\xleftarrow{1}\quad
		\tableau{
			1 \& 3 \& 5 \& 4 \\
		      1 \\
		}
		\quad\xleftarrow{5}\quad
		\tableau{
			1 \& 3 \& 5 \& 4 \\
		      1 \&   \& 5 \\
		}
		\quad\xleftarrow{2}\quad
		\tableau{
			2 \& 1 \& 3 \& 5 \& 4 \\
		       \& 1 \&   \& 5 \\
		}
		\quad\xleftarrow{1}\\[10pt]
		\xleftarrow{1}\quad
		\tableau{
            2 \& 1 \& 3 \& 5 \& 4 \\
		       \& 1 \&   \& 5 \\
              \& 1 \\
		}
        \quad\xleftarrow{2}\quad
		\tableau{
            2 \& 1 \& 3 \& 5 \& 4 \\
		      2 \& 1 \&   \& 5 \\
              \& 1 \\
		}
	\end{gather*}
\end{example}

The \emph{column reading} of a stalactic tableau $T$ is the word $\reading{T}$ given by reading the labels of each column from top to bottom, and from leftmost column to rightmost column. As such, the column reading of a stalactic tableau is a product of words which are powers of letters, where the $i$-th word corresponds to the $i$-th column of the tableau. Notice that reading columns from bottom to top gives the same output. It is also clear that applying algorithm \hyperref[alg:RightStalacticLeftInsertion]{\textsf{rStLI}} or \hyperref[alg:RightStalacticLeftInsertion]{\textsf{lStRI}} to a column reading of a tableau gives back the same tableau. Thus, we have the following:

\begin{lemma} \label{lemma:column_reading_stalactic_tableau}
    For $\word \in \N^*$, the column reading of $\plst{\word}$ (resp. $\prst{\word}$) is the product of powers of the letters of $\word$, where the exponents are the number of occurrences of the letters and the order of the letters is given by $\overrightarrow{\rho_\word}$ (resp. $\overleftarrow{\rho_\word}$).
\end{lemma}

A \emph{(standard) increasing} (resp. \emph{decreasing}) \emph{patience-sorting tableau} is a composition diagram, labelled by a permutation of $[n]$, where $n$ is the number of cells, such that the top row is strictly increasing from left-to-right and the columns are strictly increasing (resp. strictly decreasing) from top-to-bottom. Examples of increasing and decreasing patience-sorting tableaux, respectively, are the following:
\begin{equation*}
    \label{example:patience_sorting_tableaux}
    \tikz[tableau]\matrix{
		1 \& 2 \& 4 \& 7 \& 9\\
		3 \& 5 \& 8 \&   \&  \\
          \& 6 \&   \&   \&  \\ 
	};
    \quad \text{and} \quad
    \tikz[tableau]\matrix{
		3 \& 6 \& 7 \& 8 \& 9\\
		1 \& 5 \&   \& 4 \&  \\
          \& 2 \&   \&   \&  \\ 
	};
\end{equation*}

\begin{remark}
    The given definitions are variations of the usual definition of patience-sorting tableaux (or piles). Notice that the tableaux defined here are top-aligned. The following insertion algorithm will also differ from the usual patience-sorting algorithm. 
\end{remark}

To obtain an increasing patience-sorting tableau from a permutation, we use Algorithm \hyperref[alg:IncreasingPatienceSortingRightInsertion]{\textsf{iPsRI}} iteratively:

\begin{algorithm}[H] \label{alg:IncreasingPatienceSortingRightInsertion}
        \DontPrintSemicolon
        \KwIn{An increasing patience-sorting tableau $T$, a letter $a \in \N$.}
        \KwOut{An increasing patience-sorting tableau $T \leftarrow a$.}
        \BlankLine
        \eIf{$a$ is greater than the label of every cell in the top row of $T$}{add a new cell labelled by $a$ to the right of the top row of $T$;}{add a new cell to the bottom of the leftmost column whose label is greater than $a$, move each label in the column to the cell strictly below it and label the topmost cell with $a$;}
        \Return the resulting increasing patience-sorting tableau $T \leftarrow a$.
        \caption{\textit{Increasing Patience-Sorting Right-Insertion} (\textsf{iPsRI}).}
\end{algorithm}

We can define a \textit{Decreasing Patience-Sorting Left-Insertion} (\hyperref[alg:IncreasingPatienceSortingRightInsertion]{\textsf{dPsLI}}) algorithm by replacing ``right'' with ``left'', ``left'' with ``right'', and ``greater'' with ``less'', respectively, in the previous algorithm. 

Using algorithm \hyperref[alg:IncreasingPatienceSortingRightInsertion]{\textsf{iPsRI}} (resp. \hyperref[alg:IncreasingPatienceSortingRightInsertion]{\textsf{dPsLI}}), one can compute a unique increasing (resp. decreasing) patience-sorting tableau from a word $\word \in \N^*$: Starting from the empty tableau, read $\rstd{\overrightarrow{\rho_{\word}}(\word)}^{-1}$ (resp. $\rstd{\overleftarrow{\rho_{\word}}(\word)}^{-1}$) from left-to-right (resp. right-to-left) and insert its letters one-by-one into the tableau. The resulting tableau is denoted by $\qlst{\word}$ (resp. $\qrst{\word}$). In other words, to compute $\qlst{\word}$ (resp. $\qrst{\word}$), we first replace each letter in $\word$ by its image under $\overrightarrow{\rho_{\word}}$ (resp. $\overleftarrow{\rho_{\word}}$), reverse-standardise it, and then compute the inverse of the resulting permutation. 

\begin{lemma} \label{lemma:rstd_rho_inverse_positions_of_word}
    Let $\word \in \N^*$. Then, $\rstd{\overrightarrow{\rho_{\word}}(\word)}^{-1}$ (resp. $\rstd{\overleftarrow{\rho_{\word}}(\word)}^{-1}$) is a product of strictly decreasing words, with the $i$-th word giving the positions of the letter $\overrightarrow{\rho_{\word}}^{-1}(i)$ (resp. $\overleftarrow{\rho_{\word}}^{-1}(i)$) occurring in $\word$, in decreasing order. 
\end{lemma}

\begin{proof}
    The result follows from the fact that reverse-standardisation orders the letters of a word by their natural order first, and then by the reverse of the order of their occurrences. 
\end{proof}

Notice that the last letter of each strictly decreasing word is less than the first letters of words to its right, since the former corresponds to the first occurrence of a letter in $\word$ and the latter corresponds to the last occurrences of letters that only appear after it in $\word$.

\begin{example}
	\label{example:insertion_alg_qlst}
    Computing $\qrst{212511354}$: Let $\word = 212511354$. Then, 
    \[
        \rstd{\overleftarrow{\rho_{\word}}(\word)}^{-1} = \rstd{121422345}^{-1} = 251843679^{-1} = 316527849.
    \]
    \begin{gather*}  
		\tableau{
			9 \\
		}
		\quad\xleftarrow{4}\quad 
		\tableau{
			4 \& 9 \\
		}
		\quad\xleftarrow{8}\quad 
		\tableau{
			8 \& 9 \\
            4   \\
		}
		\quad\xleftarrow{7}\quad 
		\tableau{
			7 \& 8 \& 9 \\
              \& 4   \\
		}
		\quad\xleftarrow{2}\\[10pt]
		\xleftarrow{2}\quad
		\tableau{
			2 \& 7 \& 8 \& 9 \\
              \&   \& 4   \\
		}
		\quad\xleftarrow{5}\quad
		\tableau{
			5 \& 7 \& 8 \& 9 \\
            2 \&   \& 4   \\
		}
		\quad\xleftarrow{6}\quad
		\tableau{
			6 \& 7 \& 8 \& 9 \\
            5 \&   \& 4   \\
            2 \\
		}
		\quad\xleftarrow{1}\\[10pt]
		\xleftarrow{1}\quad
		\tableau{
            1 \& 6 \& 7 \& 8 \& 9 \\
              \& 5 \&   \& 4   \\
              \& 2 \\
		}
        \quad\xleftarrow{3}\quad
		\tableau{
            3 \& 6 \& 7 \& 8 \& 9 \\
            1 \& 5 \&   \& 4   \\
              \& 2 \\
		}
	\end{gather*}
\end{example}

The \emph{column reading} of an increasing (resp. decreasing) patience-sorting tableau $T$ is the word $\reading{T}$ given by reading the labels of each column from bottom to top (resp. from top to bottom), and from leftmost column to rightmost column. As such, the column reading of a patience-sorting tableau will be a product of strictly decreasing words, with the $i$-th word corresponding to the $i$-th column of the tableau. Furthermore, applying algorithm \hyperref[alg:IncreasingPatienceSortingRightInsertion]{\textsf{iPsRI}} and \hyperref[alg:IncreasingPatienceSortingRightInsertion]{\textsf{dPsLI}} to a column reading of a tableau gives back the same tableau. Thus, we can conclude that:

\begin{lemma} \label{lemma:column_reading_patience_sorting_tableau}
For any $\word \in \N^*$, we have 
\[
    \reading{\qlst{\word}} = \rstd{\overrightarrow{\rho_{\word}}(\word)}^{-1} \quad \text{and} \quad \reading{\qrst{\word}} = \rstd{\overleftarrow{\rho_{\word}}(\word)}^{-1}.
\]
\end{lemma}

As such, we have that the stalactic $\psymb$ and $\qsymb$-symbols have the same structure, in the following sense:

\begin{lemma} \label{lemma:pxst_qxst_same_shape}
    Let $\word \in \N^*$. Then, $\plst{\word}$ (resp. $\prst{\word}$) and $\qlst{\word}$ (resp. $\qrst{\word}$) have the same shape. Furthermore, the letters of $\word$ given by the $i$-th column of $\plst{\word}$ (resp. $\prst{\word}$) have their position given by the letters of the $i$-th column of $\qlst{\word}$ (resp. $\qrst{\word}$).
\end{lemma}
\begin{proof}
    The result follows from Lemmas~\ref{lemma:column_reading_stalactic_tableau}, \ref{lemma:rstd_rho_inverse_positions_of_word}, and \ref{lemma:column_reading_patience_sorting_tableau}.
\end{proof}

Now, we show the analogue of the Robinson--Schensted correspondence for pairs of stalactic and patience-sorting tableaux. We show the result for the left-insertion algorithm case:

\begin{theorem} \label{theorem:robinson_left_stalactic}
    The map $\word \mapsto (\plst{\word}, \qlst{\word})$ is a bijection between the elements of $\N^*$ and the set formed by the pairs $(T,S)$ where
    \begin{enumerate}[(i)]
        \item $T$ is a stalactic tableau.
        \item $S$ is an increasing patience-sorting tableau.
        \item $T$ and $S$ have the same shape.
    \end{enumerate}
\end{theorem}
\begin{proof}
    Let $\word \in \N^*$. We can see that the map is well-defined by the definition of $\plst{\word}$ and $\qlst{\word}$ and Lemma~\ref{lemma:pxst_qxst_same_shape}. Furthermore, it is possible to reconstruct $\word$ from $(\plst{\word}, \qlst{\word})$ by taking the column readings of $\plst{\word}$, which gives $\cont{\word}$ and $\overrightarrow{\rho_{\word}}$ by Lemma~\ref{lemma:column_reading_stalactic_tableau}, and $\qlst{\word}$, which, by Lemmas~\ref{lemma:rstd_rho_inverse_positions_of_word} and \ref{lemma:column_reading_patience_sorting_tableau}, allows us to reconstruct $\word$ from the information obtained from $\plst{\word}$. In other words, for $\uord = \reading{\plst{\word}}$ and $\sigma = \reading{\qlst{\word}}$, we have that the $\sigma(i)$-th letter of $\word$ is given by the $i$-th letter of $\uord$, for each $1 \leq i \leq |\word|$. As such, the map is injective.
    
    For surjectivity, let $(T,S)$ satisfy conditions \textit{(i)--(iii)}. Let $\uord = \reading{T}$ and $\sigma = \reading{S}$, then define $\word \in \N^*$ such that the $\sigma(i)$-th letter of $\word$ is given by the $i$-th letter of $\uord$, for each $1 \leq i \leq |\uord|$. 

    It is clear that $\cont{\word} = \cont{\uord}$. On the other hand, since $T$ and $S$ have the same shape, it follows from the observations made on column readings and the definition of $\word$ that the letters of $\word$ given by the $i$-th column of $T$ have their position given by the letters of the $i$-th column of $S$. In particular, the position of the first occurrence of a letter is given by the topmost letter in its column. Since $S$ is an increasing patience-sorting tableau, the topmost letter in a column is less than every letter in any column to the right. As such, we have $\overrightarrow{\rho_\word} = \overrightarrow{\rho_\uord} = \rho_T$ and therefore $\plst{\word} = T$.
    
    On the other hand, $\rstd{\overrightarrow{\rho_{\word}}(\word)}^{-1}$ is the unique product of strictly decreasing words, where the $i$-th word gives the positions of the letter $\overrightarrow{\rho_{\word}}^{-1}(i)$ in $\word$, which, by the previous paragraph, corresponds to the $i$-th column of $T$. As such, by Lemma~\ref{lemma:column_reading_patience_sorting_tableau}, we have that $\reading{\qlst{\word}} = \sigma$, and hence, since $\plst{\word}$ and $S$ have the same shape, $\qlst{\word} = S$. Thus, the map is surjective.
\end{proof}

By symmetrical reasoning, we obtain the result for the right-insertion algorithm case:

\begin{theorem} \label{theorem:robinson_right_stalactic}
    The map $\word \mapsto (\prst{\word}, \qrst{\word})$ is a bijection between the elements of $\N^*$ and the set formed by the pairs $(T,S)$ where
    \begin{enumerate}[(i)]
        \item $T$ is a stalactic tableau.
        \item $S$ is a decreasing patience-sorting tableau.
        \item $T$ and $S$ have the same shape.
    \end{enumerate}
\end{theorem}


\subsubsection{Binary trees with multiplicities and binary search trees with multiplicities} \label{subsubsection:binary_trees_with_multiplicities_and_binary_search_trees_with_multiplicities}

A \emph{binary tree with multiplicities} (BTM) is a rooted planar binary tree where each node is labelled by a positive integer, called a \emph{multiplicity}. A \emph{binary search tree with multiplicities} (BSTM) is a BTM where each node is further labelled by a letter of $\N$, such that the letter of the left (resp. right) child of any node is strictly less than (resp. strictly greater than) the letter of said node. In other words, if we remove the multiplicities of the nodes in a BSTM, we obtain a binary search tree. For ease of notation, we write multiplicities as superscripts. An example of a BSTM is the following:

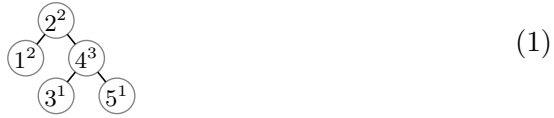
\begin{equation}
    \label{example:bstm}
    \begin{tikzpicture}[tinybst,baseline=-4mm]
        \node {$2^2$}
        child { node {$1^2$} }
        child { node {$4^3$}
            child { node {$3^1$} }
            child { node {$5^1$} }
        };
    \end{tikzpicture}
\end{equation}

To be clear on terminology, by \emph{label} of a node of a BSTM, we mean the pair consisting of the letter and multiplicity that labels the node. If we are just referring to the letter, we simply say the \emph{letter} of a node. We say a node in a BSTM is \emph{simple} if it has multiplicity $1$. 

Algorithm \hyperref[alg:TaigaLeafInsertion]{\textsf{TgLI}} allows one to insert a letter into a BSTM, either by a new leaf labelled by it or by increasing the multiplicity of a node labelled by it, and still obtain a BSTM:

\begin{algorithm}[H] \label{alg:TaigaLeafInsertion}
        \DontPrintSemicolon
        \KwIn{A BSTM $T$, a letter $a \in \N$.}
        \KwOut{A BSTM $T \uparrow a$.}
        \BlankLine
        \eIf{$T$ is empty}{add a node labelled $a$ with multiplicity $1$.}{
            let $b$ be the letter of the root node of $T$;\\
            \If{$a < b$}{recursively insert $a$ into the left subtree of the root node;}
            \ElseIf{$a > b$}{recursively insert $a$ into the right subtree of the root node;}
            \Else{increment by 1 the multiplicity of the root node;}
        }
        \Return the resulting tree $T \uparrow a$.
        \caption{\textit{Taiga Leaf Insertion} (\textsf{TgLI}).}
\end{algorithm}

Using Algorithm \hyperref[alg:TaigaLeafInsertion]{\textsf{TgLI}}, one can compute a unique BSTM from a word $\word \in \N^*$: Starting from the empty tree, read $\word$ from right-to-left (resp. left-to-right) and insert its letters one-by-one into the tree. The resulting tree is denoted by $\prtg{\word}$ (resp. $\pltg{\word}$).

\begin{example}
	\label{example:insertion_alg_prtg}
	Computing $\prtg{451423412}$:
	\begin{gather*}  
		\begin{tikzpicture}[tinybst, baseline=-1mm]
			\node {$2^1$};
		\end{tikzpicture}
		\quad\xleftarrow{1}\quad
		\begin{tikzpicture}[tinybst, baseline=-4mm]
			\node {$2^1$}
			child { node {$1^1$} }
            child[missing];
		\end{tikzpicture}
		\quad\xleftarrow{4}\quad
		\begin{tikzpicture}[tinybst, baseline=-4mm]
		    \node {$2^1$}
            child { node {$1^1$} }
            child { node {$4^1$} };
		\end{tikzpicture}
		\quad\xleftarrow{3}\quad
		\begin{tikzpicture}[tinybst, baseline=-6mm]
		    \node {$2^1$}
            child { node {$1^1$} }
            child { node {$4^1$}
                child { node {$3^1$} }
                child [missing]
            };
		\end{tikzpicture}
		\quad\xleftarrow{2}\quad		
		\begin{tikzpicture}[tinybst, baseline=-6mm]
		    \node {$2^2$}
            child { node {$1^1$} }
            child { node {$4^1$}
                child { node {$3^1$} }
                child [missing]
            };
		\end{tikzpicture}
		\quad\xleftarrow{4}\\[10pt]
		\xleftarrow{4}\quad
		\begin{tikzpicture}[tinybst, baseline=-6mm]
		    \node {$2^2$}
            child { node {$1^1$} }
            child { node {$4^2$}
                child { node {$3^1$} }
                child [missing]
            };
		\end{tikzpicture}
		\quad\xleftarrow{1}\quad
		\begin{tikzpicture}[tinybst, baseline=-6mm]
		    \node {$2^2$}
            child { node {$1^2$} }
            child { node {$4^2$}
                child { node {$3^1$} }
                child [missing]
            };
        \end{tikzpicture}
		\quad\xleftarrow{5}\quad
		\begin{tikzpicture}[tinybst, baseline=-6mm]
		    \node {$2^2$}
            child { node {$1^2$} }
            child { node {$4^2$}
                child { node {$3^1$} }
                child { node {$5^1$} }
            };
        \end{tikzpicture}
		\quad\xleftarrow{4}\quad
        \begin{tikzpicture}[tinybst, baseline=-6mm]
		    \node {$2^2$}
            child { node {$1^2$} }
            child { node {$4^3$}
                child { node {$3^1$} }
                child { node {$5^1$} }
            };
        \end{tikzpicture}
	\end{gather*}
\end{example}

The \emph{inorder traversal} of a labelled rooted binary tree is the sequence of nodes obtained by recursively computing the inorder traversal of the left subtree of the root node, then adding the root node to the sequence and then recursively computing the inorder traversal of the right subtree of the root node. We say a node of a tree is its $i$-th node if it is the $i$-th node in the inorder traversal of the tree. 
The \emph{inorder reading} of a BSTM $T$ is the word $\reading{T}$ given by the product of powers of letters labelling the nodes of the inorder traversal, where the exponents are given by the multiplicities. 
Notice that the inorder reading of a BSTM is a weakly increasing sequence. 

Notice that a BSTM cannot have two nodes with the same letter. As such, given a BTM and a subset of $\N$ with size equal to the number of nodes of the tree, there is only one way to label the tree with the letters of the subset and obtain a BSTM (see Definition 8 and the following remark in \cite{priez_binary_trees}).

An \emph{increasing} (resp. \emph{decreasing}) \emph{binary tree} is an $\N$-labelled rooted planar binary tree, such that the label of each node is less than (resp. greater than) the label of its children, and no two nodes have the same label. For $\word \in \N^*$ with no repeated letters, algorithm \hyperref[alg:RecursiveIncreasingBinaryTreeAlgorithm]{\textsf{RiBTA}} allows one to compute an increasing binary tree $\mathrm{incr}\parens{\word}$ obtained from $\word$ in the following way:

\begin{algorithm}[H] \label{alg:RecursiveIncreasingBinaryTreeAlgorithm}
        \DontPrintSemicolon
        \KwIn{A word $\word \in \N^*$ with no repeated letters.}
        \KwOut{An increasing binary tree $\mathrm{incr}\parens{\word}$.}
        \BlankLine
        let $\mathrm{incr}\parens{\word} = \perp$; \\
        \If{$|\word| \geq 1$}{
            let $\word = \uord a \vord$, where $a$ is the least letter of $\word$; \\
            label the root node of $\mathrm{incr}\parens{\word}$ by $a$ and recursively compute its left subtree $\mathrm{incr}\parens{\uord}$ and right subtree $\mathrm{incr}\parens{\vord}$;
        }
        \Return the resulting tree $\mathrm{incr}\parens{\word}$.
        \caption{\textit{Recursive Increasing Binary Tree Algorithm} (\textsf{RiBTA}).}
\end{algorithm}

We can define a \emph{Recursive Decreasing Binary Tree Algorithm} (\hyperref[alg:RecursiveIncreasingBinaryTreeAlgorithm]{\textsf{RdBTA}}) algorithm by replacing ``\textrm{incr}'' with ``\textrm{decr}'' and ``least'' with ``greatest'' in the previous algorithm.


We say that a binary tree whose labels are non-intersecting subsets of $\N$ is an \emph{increasing} (resp. \emph{decreasing}) \emph{binary tree over sets} (BTS) if the binary tree obtained by replacing each set with its minimum (resp. maximum) element is an increasing (resp. decreasing) binary tree. For ease of notation, we label the nodes of these trees with the elements of the subsets only. Examples of increasing and decreasing binary trees over sets are the following:
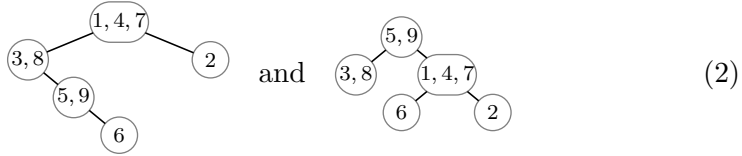
\begin{equation}
    \label{example:increasing_decreasing_binary_trees_over_sets}
    \begin{tikzpicture}[microbst,baseline=-8mm]
        \node {$1,4,7$}
        child { node {$3,8$} 
            child[missing]
            child { node {$5,9$}
                child[missing]
                child { node {$6$} }
            }
        }
        child[missing]
        child { node {$2$} };
    \end{tikzpicture}
    \quad \text{and} \quad
    \begin{tikzpicture}[microbst, baseline=-6mm]
        \node {$5,9$}
        child { node {$3,8$} }
        child { node {$1,4,7$}
            child { node {$6$} }
            child { node {$2$} }
        };
    \end{tikzpicture}
\end{equation}

Using Algorithm \hyperref[alg:RecursiveIncreasingBinaryTreeAlgorithm]{\textsf{RiBTA}}, (resp. \hyperref[alg:RecursiveIncreasingBinaryTreeAlgorithm]{\textsf{RdBTA}}) one can compute a unique increasing (resp. decreasing) BTS from a word $\word \in \N^*$:
Let $\supp{\word} = \{x_1 < \dots < x_k\}$. For each $i \in [k]$, consider the set $W_i$ of all $j \in [|\word|]$ such that $x_i$ is the $j$-th letter of $\word$, when reading $\word$ from left-to-right. Let $\vord = v_1\dots v_k \in \N^*$ be the word where $v_i = \min(W_i)$ (resp. $v_i = \max(W_i)$). Now, apply \hyperref[alg:RecursiveIncreasingBinaryTreeAlgorithm]{\textsf{RiBTA}}, (resp. \hyperref[alg:RecursiveIncreasingBinaryTreeAlgorithm]{\textsf{RdBTA}}) to $\vord$ to obtain an increasing (resp. decreasing) binary tree. Finally, for each $i \in [k]$ replace the label $v_i$ in the tree with the set $W_i$. The resulting tree is denoted $\qltg{\word}$ (resp. $\qrtg{\word}$).

\begin{example}
	\label{example:insertion_alg_qrtg}
	Computing $\qrtg{451423412}$: Let $\word = 451423412$. Then, $W_1 = \{3,8\}$, $W_2 = \{5,9\}$, $W_3 = \{6\}$, $W_4 = \{1,4,7\}$ and $W_5 = \{2\}$ and, as such, $\vord = 89672$.
	
	\begin{gather*}  
		\begin{tikzpicture}[microbst, baseline=-1mm]
			\node {$9$};
		\end{tikzpicture}
		\quad\xleftarrow{8,7}\quad
		\begin{tikzpicture}[microbst, baseline=-4mm]
		    \node {$9$}
            child { node {$8$} }
            child { node {$7$} };
		\end{tikzpicture}
		\quad\xleftarrow{6,2}\quad
		\begin{tikzpicture}[microbst, baseline=-6mm]
		    \node {$9$}
            child { node {$8$} }
            child { node {$7$}
                child { node {$6$} }
                child { node {$2$} }
            };
        \end{tikzpicture}
		\quad\leftarrow\quad
        \begin{tikzpicture}[microbst, baseline=-6mm]
		    \node {$5,9$}
            child { node {$3,8$} }
            child { node {$1,4,7$}
                child { node {$6$} }
                child { node {$2$} }
            };
        \end{tikzpicture}
	\end{gather*}
\end{example}

\begin{lemma} \label{lemma:pxtg_qxtg_same_shape}
    Let $\word \in \N^*$ and $\supp{\word} = \{x_1 < \dots < x_k\}$. Then, $\pltg{\word}$ (resp. $\prtg{\word}$) and $\qltg{\word}$ (resp. $\qrtg{\word}$) have the same underlying binary tree shape. Furthermore, for each $i \in [k]$, the letters $x_i$ have their positions in $\word$ given by the label of the $i$-th node of $\qltg{\word}$ (resp. $\qrtg{\word}$).
\end{lemma}
\begin{proof}
    We prove the case for $\qrst{\word}$ and $\prst{\word}$. The proof for the remaining case is symmetrical.

    By the definition of $\qrst{\word}$, its underlying binary tree shape is given by $\mathrm{decr}(\vord)$, where $\vord = v_1 \cdots v_k$ is such that $v_i$ is the position of the last occurrence of $x_i$ in $\word$. Furthermore, the $i$-th node of $\qrst{\word}$ is the unique node whose label is a set containing the label of the $i$-th node of $\mathrm{decr}(\vord)$.
    
    Notice that $\vord$ is a word where no letter is repeated. Thus, $\mathrm{decr}(\vord)$ has the same shape as $\mathrm{decr}(\std{\vord})$ and there is an order-preserving bijection taking the label of the $i$-th of $\mathrm{decr}(\vord)$ to that of the $i$-th node of $\mathrm{decr}(\std{\vord})$. 
    
    By \cite[Theorem~13]{hivert_sylvester}, which gives the sylvester Robinson--Schensted correspondence, we have that $\prtg{\std{\vord}^{-1}}$ has the same shape as $\mathrm{decr}(\std{\vord})$, as Algorithm \hyperref[alg:TaigaLeafInsertion]{\textsf{TgLI}} outputs the same result as the algorithm given in \cite[Definition~7]{hivert_sylvester} for right strict binary search trees, if the input is a permutation and one views a BSTM with only simple nodes as a right strict binary search tree.
    Moreover, the label of $i$-th node of $\mathrm{decr}(\std{\vord})$ gives the position of $i$ in $\std{\vord}^{-1}$. 
    
    Notice that the binary tree shape of $\prtg{\word}$ is determined by the last occurrences of letters in $\word$, and that $\std{\vord}^{-1}$ is the word obtained from $\word$ by only keeping the last occurrences of each letter and replacing each $x_i$ with $i$. As the map $x_i \mapsto i$ is an order-preserving bijection, we conclude that $\prtg{\std{\vord}^{-1}}$ has the same shape as $\prtg{\word}$.
    
    In conclusion, the $i$-th node of $\prtg{\word}$ corresponds to the $i$-th node of $\qrtg{\word}$, thus, by the definition of $\qrtg{\word}$, the positions of $x_i$ in $\word$ are given by the label of the $i$-th node of $\qrtg{\word}$.
\end{proof}

Now, we show the analogue of the Robinson--Schensted correspondence for pairs of BSTMs and BTSs. We show the result for the right-insertion algorithm case:

\begin{theorem} \label{theorem:robinson_right_taiga}
    The map $\word \mapsto (\prtg{\word}, \qrtg{\word})$ is a bijection between the elements of $\N^*$ and the set formed by the pairs $(T,S)$ where
    \begin{enumerate}[(i)]
        \item $T$ is a BSTM.
        \item $S$ is a decreasing BTS such that the union of the sets labelling $S$ is the interval $[m]$, where $m$ is the sum of the multiplicities of $T$.
        \item $T$ and $S$ have the same underlying binary tree shape.
        \item the multiplicity of the $i$-th node of $T$ is the cardinality of the set labelling the $i$-th node of $S$.
    \end{enumerate}
\end{theorem}
\begin{proof}
    Let $\word \in \N^*$ and $\supp{\word} = \{x_1,\dots,x_k\}$. We can see that the map is well-defined by the definition of $\prtg{\word}$ and $\qrtg{\word}$ and Lemma~\ref{lemma:pxtg_qxtg_same_shape}. Furthermore, it is possible to reconstruct $\word$ from $(\prtg{\word}, \qrtg{\word})$ as by Lemma~\ref{lemma:pxtg_qxtg_same_shape} the $i$-th node of $\qrtg{\word}$ gives the positions of $x_i$ in $\word$. As such, the map is injective.
    
    For surjectivity, let $(T,S)$ satisfy conditions \textit{(i)--(iv)}. Let $\supp{\reading{T}} = \{x_1,\dots,x_k\}$. Now, define $\word \in \N^*$ as the rearrangement of $\reading{T}$ where the positions of the letter $x_i$ are given by the set labelling the $i$-th node of $S$. By conditions \textit{(ii)} and \textit{(iv)}, $\word$ is well-defined.
    
    We aim to show $\prtg{\word} = T$ and $\qrtg{\word} = S$. By definition, we have $\cont{\word} = \cont{\reading{T}}$. Since $S$ is a decreasing BTS and has the same underlying binary tree shape as $T$, then, by Algorithm \hyperref[alg:TaigaLeafInsertion]{\textsf{TgLI}}, the node labelled $x_i$ is an ancestor of the node labelled $x_j$ in $\prtg{\word}$ if and only if the same occurs in $T$. Thus, $\prtg{\word} = T$. On the other hand, by the definition of $\word$ and Lemma~\ref{lemma:pxtg_qxtg_same_shape}, we have that $\qrtg{\word} = S$. Hence, the map is surjective.
\end{proof}

By symmetrical reasoning, we obtain the result for the left-insertion algorithm case:
\begin{theorem} \label{theorem:robinson_left_taiga}
    The map $\word \mapsto (\pltg{\word}, \qltg{\word})$ is a bijection between the elements of $\N^*$ and the set formed by the pairs $(T,S)$ where
    \begin{enumerate}[(i)]
        \item $T$ is a BSTM.
        \item $S$ is an increasing BTS such that the union of the sets labelling $S$ is the interval $[m]$, where $m$ is the sum of the multiplicities of $T$.
        \item $T$ and $S$ have the same underlying binary tree shape.
        \item the multiplicity of the $i$-th node of $T$ is the cardinality of the set labelling the $i$-th node of $S$.
    \end{enumerate}
\end{theorem}

\subsection{The stalactic and taiga monoids} \label{subsection:the_stalactic_and_taiga_monoids}

We now give the definitions and required properties of the (right and left) stalactic and taiga monoids. For background on the sylvester monoid, see \cite{hivert_sylvester}; on the Baxter monoid, see \cite{giraudo_baxter}; on the identities satisfied by these monoids, see \cite{cain_malheiro_ribeiro_sylvester_baxter_2023}. 

We define the \emph{right stalactic congruence} $\equiv_\rst$ on $\N^*$ in the following way: for $\uord,\vord \in \N^*$,
\[ \uord \equiv_\rst \vord \Leftrightarrow \prst{\uord} = \prst{\vord} \]
and is generated by the relations
\[ \cR_\rst = \{(ba \uord b,ab \uord b)\colon a,b \in \N, \uord \in \N^*\}. \]

The factor monoid $\N^*/{\equiv_\rst}$ is the infinite-rank \emph{right stalactic monoid}, denoted by $\rst$. The congruence naturally restricts to a congruence on $[n]^*$, for each $n \in \N$, and the corresponding factor monoid $[n]^*/{\equiv_\rst}$ is the right stalactic monoid of rank $n$. Notice that any right stalactic monoid of rank $n$ is (isomorphic to) a submonoid of all right stalactic monoids of rank greater than $n$.

It follows from the definition of $\rst$ that any element $[\uord]_\rst$ of $\rst$ can be uniquely identified with the stalactic tableau $\prst{\uord}$. All words in each $\equiv_\rst$-class share the same content (and therefore the same support), hence, we extend the definition of content and support to $\equiv_\rst$-classes and stalactic tableaux in a natural way. On the other hand, it follows from the defining relations of the right stalactic congruence that two words $\uord, \vord \in \N^*$ are $\equiv_\rst$-congruent if and only if $\cont{\uord} = \cont{\vord}$ and $\overleftarrow{\rho_\uord} = \overleftarrow{\rho_\vord}$. From this, it is immediate that $\rst$ is left-cancellative.

The left stalactic monoids $\lst$, right taiga monoids $\rtg$ and left taiga monoids $\ltg$ are defined in a similar manner. Recall that the sylvester and \#-sylvester congruences are generated, respectively, by the relations
\begin{align*}
    \cR_\sylv = \{(ca \uord b,ac \uord b) \colon a \leq b < c, \uord \in \N^*\} &\quad \text{and} \\ 
    \cR_\sylvh = \{(b \uord ac,b \uord ca) \colon a < b \leq c, \uord \in \N^*\}&.
\end{align*}
The left-stalactic congruence is generated by the relations
\[ \cR_\lst = \{(b \uord ab,b \uord ba)\colon a,b \in \N, \uord \in \N^*\}.\]
The right-taiga and left-taiga congruences are generated, respectively, by the relations
\[
    \cR_\rtg = \cR_\sylv \cup \cR_\rst \quad \text{and} \quad \cR_\ltg = \cR_\sylvh \cup \cR_\lst.
\]

It follows from the defining relations of the left stalactic congruence that two words $\uord, \vord \in \N^*$ are $\equiv_\lst$-congruent if and only if $\cont{\uord} = \cont{\vord}$ and $\overrightarrow{\rho_\uord} = \overrightarrow{\rho_\vord}$. From this, it is immediate that $\lst$ is right-cancellative.

The proof of \cite[Lemma~17]{cm_identities} can be easily adapted to show that $\rtg$ is left-cancellative and $\ltg$ is right-cancellative.



\section{Meets and joins of stalactic and taiga congruences} \label{section:meets_and_joins_of_stalactic_and_taiga_congruences}

In this section, we define four plactic-like monoids arising from the meets and joins of the stalactic and taiga congruences, and study their compatibility properties.


\subsection{Monoids arising from meets and joins of stalactic and taiga congruences} \label{subsection:monoids_arising_from_meets_and_joins_of_stalactic_and_taiga_congruences}

We first define the meets and joins of the stalactic and taiga congruences, respectively, and then their corresponding monoids. We then give presentations for the meet monoids and characterise words that are equal in the join monoids.


\subsubsection{The meet-stalactic and join-stalactic monoids} \label{subsubsection:the_meet_stalactic_and_join_stalactic_monoids}

We define the \emph{meet-stalactic} congruence $\equiv_{\mst}$ and the \emph{join-stalactic} congruence $\equiv_{\jst}$ as, respectively, the meet and join of the right-stalactic and left-stalactic congruences, and define the infinite-rank \emph{meet-stalactic} monoid $\mst$ and \emph{join-stalactic} monoid $\jst$ as quotients of $\N^*$ by their respective congruences. These congruences naturally restrict to congruences on $[n]^*$, for each $n \in \N$, and the corresponding factor monoids $[n]^*/{\equiv_\mst}$ and $[n]^*/{\equiv_\jst}$ are the meet-stalactic and join-stalactic monoids of rank $n$.

Notice that two words $\uord,\vord \in \N^*$ are $\equiv_\mst$-congruent if and only if they share the same content, $\overleftarrow{\rho_\uord} = \overleftarrow{\rho_\vord}$ and $\overrightarrow{\rho_\uord} = \overrightarrow{\rho_\vord}$. Furthermore, it is clear that the join-stalactic monoid is presented by $\pres{\N}{\cR_\jst}$, where $\cR_\jst := \cR_\rst \cup \cR_\lst$. On the other hand, we have the following results:

\begin{proposition} \label{prop:MeetStalacticPresentation}
    The meet-stalactic monoid $\mst$ is presented by $\pres{\N}{\cR_\mst}$, where
    \begin{align*}
        \cR_\mst :=& \{ (b \uord ba \vord b,b \uord ab \vord b)\colon a,b \in \N, \uord,\vord \in \N^* \} \\
        & \cup \{ (a \uord ab \vord b,a \uord ba \vord b)\colon a,b \in \N, \uord,\vord \in \N^* \}.
    \end{align*}
\end{proposition}

\begin{proof}
    Consider the congruence $\equiv$ generated by the relations $\cR_\mst$, as given in the statement. It is clear that, for any words $\uord,\vord \in \N^*$, if $\uord \equiv \vord$, then $\uord \equiv_\rst \vord$ and $\uord \equiv_\lst \vord$, hence $\uord \equiv_{\mst} \vord$. Suppose now that $\uord \equiv_\rst \vord$ and $\uord \equiv_\lst \vord$. Let
    \[
        \uord = \word_1 \uord' \word_2 \quad \text{and} \quad \vord = \word_1 \vord' \word_2,
    \]
    where $\uord' \neq \vord'$ and $\word_1, \word_2 \in \N^*$ are the common prefix and suffix of maximum length, respectively. Notice that the first (resp. last) letter of $\uord'$ is different from the first (resp. last) letter of $\vord'$. On the other hand, since $\uord$ and $\vord$ have the same content, then $\uord'$ and $\vord'$ have the same content as well. In particular, they have the same length, hence, we will prove by induction on the length of $\uord'$ and $\vord'$ that $\uord \equiv \vord$.
    
    The base case for the induction is when $|\uord'|=|\vord'|=2$. In this case, we have $\uord' = ab$ and $\vord' = ba$, for some $a,b \in \N$. Since $\uord \equiv_\rst \vord$, then $a$ or $b$ must occur in $\word_2$, and since $\uord \equiv_\lst \vord$, then $a$ or $b$ must occur in $\word_1$. Therefore, if the same symbols occur in both $\word_1$ and $\word_2$, we use the first defining relation of $\equiv$, and if different symbols occur, we use the second defining relation to show that $\uord \equiv \vord$.
    
    Suppose now that $|\uord'|=|\vord'|>2$. Write $\uord' = a \uord_1'$ and $\vord' = \vord_1' a \vord_2'$, for some $a \in \N$ and $\vord_1' \in \N^+$ such that $a \notin \supp{\vord_1'}$. Consider the following cases:
    \begin{itemize}
        \item $a \notin \supp{\word_1} \cup \supp{\vord_2' \word_2}$;
        \item $a \in \supp{\word_1}$ and $a \notin \supp{\vord_2' \word_2}$;
        \item $a \in \supp{\vord_2' \word_2}$ and $a \notin \supp{\word_1}$;
        \item $a \in \supp{\word_1} \cap \supp{\vord_2' \word_2}$.
    \end{itemize}

    If $a \notin \supp{\word_1}$, then the first occurrence of $a$ is after the first occurrence of every letter of $\vord_1'$, in the word $\vord$. Therefore, since $\uord \equiv_\lst \vord$, we have that $\supp{\vord_1'} \subseteq \supp{\word_1}$. On the other hand, if $a \notin \supp{\vord_2' \word_2}$, then the last occurrence of $a$ is before the last occurrence of every letter of $\vord_1'$, in the word $\uord$. Therefore, since $\uord \equiv_\rst \vord$, we have that $\supp{\vord_1'} \subseteq \supp{\vord_2' \word_2}$. As such, using the first defining relation of $\equiv$ in the first and fourth cases and the second defining relation of $\equiv$ in the second and third cases, finitely many times, we have that $\vord \equiv \word_1 a \vord'' \word_2$, for some $\vord'' \in \N^*$. Hence, by the induction hypothesis, $\uord \equiv \vord$.
    
    Thus, we have shown that $\uord \equiv_\lst \vord$ and $\uord \equiv_\rst \vord$ imply that $\uord \equiv \vord$, and as such, the congruences $\equiv$ and $\equiv_\mst$ are the same. 
\end{proof}
    
\begin{proposition} \label{prop:jst_characterisation}
    Let $\uord,\vord \in \N^*$. Then, $\uord \equiv_\jst \vord$ if and only if $\cont{\uord} = \cont{\vord}$ and $\ol{\uord} = \ol{\vord}$. 
\end{proposition}
\begin{proof}
Consider the congruence $\equiv$ given by $\uord \equiv \vord$ if and only if $\cont{\uord} = \cont{\vord}$ and $\ol{\uord} = \ol{\vord}$, for $\uord,\vord \in \N^*$. We first show that $\uord \equiv \vord$ implies $\uord \equiv_\jst \vord$. Let $\ol{\uord} = \ol{\vord} = a_1 \cdots a_k$, where $a_1, \dots, a_k$ denote the simple letters, and let $x_1, \dots, x_l$ denote the non-simple letters, with $k + l = |\supp{\uord}|$. 

Ranging $1 \leq j \leq l$, we can repeatedly use $\equiv_\rst$ to shift the leftmost $x_j$ to the beginning of both $\uord$ and $\vord$; and then repeatedly use $\equiv_\lst$ to obtain a word with the prefix $x_j^{|\uord|_{x_j}}$. Inductively, this allows us to show that
\[ \uord \equiv_\jst x_l^{|\uord|_{x_l}} \cdots x_1^{|\uord|_{x_1}} a_1 \cdots a_k \equiv_\jst \vord. \]

On the other hand, notice that $\uord \equiv_\lst \vord$ implies $\ol{\uord} = \ol{\vord}$ since the order of occurrences of simple letters can be deduced from the order of the first occurrences of all letters. Similarly, $\uord \equiv_\rst \vord$ implies $\ol{\uord} = \ol{\vord}$. Thus, since the join-stalactic congruence is the congruence join of $\equiv_\lst$ and $\equiv_\rst$, $\uord \equiv_\jst \vord$ also implies $\ol{\uord} = \ol{\vord}$. Hence, $\uord \equiv_\jst \vord$ implies $\uord \equiv \vord$.
\end{proof}


\subsubsection{The meet-taiga and join-taiga monoids} \label{subsubsection:the_meet_taiga_and_join_taiga_monoids}

As with the case of the stalactic monoids, we now consider the meet and join of the right-taiga and left-taiga congruences, which we respectively call the \emph{meet-taiga} congruence $\equiv_{\mtg}$ and the \emph{join-taiga} congruence $\equiv_{\jtg}$, and define the \emph{meet-taiga} monoid $\mtg$ and \emph{join-taiga} monoid $\jtg$ as quotients of $\N^*$ by their respective congruences. These congruences naturally restrict to congruences on $[n]^*$, for each $n \in \N$, and the corresponding factor monoids $[n]^*/{\equiv_\mtg}$ and $[n]^*/{\equiv_\jtg}$ are the meet-taiga and join-taiga monoids of rank $n$. Notice that the meet-taiga congruence contains both the Baxter congruence and the meet-stalactic congruence. 

As before with the case of the join-stalactic monoid, it is clear that the join-taiga monoid is presented by $\pres{\N}{\cR_\jtg}$, where $\cR_\jtg := \cR_\rtg \cup \cR_\ltg$. On the other hand, we have the following results:

\begin{proposition} \label{prop:MeetTaigaPresentation}
    The meet-taiga monoid $\mtg$ is presented by the relations $\cR_\mtg$, where
    \begin{align*}
        \cR_\mtg :=& \{ (b \uord ad \vord c,b \uord da \vord c)\colon a \leq d, \; b,c \in [a,d], \; \uord,\vord \in \N^* \}.
    \end{align*}
\end{proposition}

\begin{proof} 
    Consider the congruence $\equiv$ generated by the relations $\cR_\mtg$, as given in the statement. It is clear that, for any words $\uord,\vord \in \N^*$, if $\uord \equiv \vord$, then $\uord \equiv_\rtg \vord$ and $\uord \equiv_\ltg \vord$, hence $\uord \equiv_{\mtg} \vord$. Suppose now that $\uord \equiv_\rtg \vord$ and $\uord \equiv_\ltg \vord$. Let
    \[
        \uord = \word_1 \uord' \word_2 \quad \text{and} \quad \vord = \word_1 \vord' \word_2,
    \]
    where $\uord' \neq \vord'$ and $\word_1, \word_2 \in \N^*$ are of maximum length. Notice that the first (resp. last) letter of $\uord'$ is different from the first (resp. last) letter of $\vord'$. On the other hand, since $\uord$ and $\vord$ have the same content, then $\uord'$ and $\vord'$ have the same content as well. In particular, they have the same length, hence, we will prove by induction on the length of $\uord'$ that $\uord \equiv \vord$.
    
    The base case for the induction is when $|\uord'|=|\vord'|=2$. In this case, we have $\uord' = ad$ and $\vord' = da$, for some $a,d \in \N$. Assume, without loss of generality, that $a \leq d$. Since $\uord \equiv_\ltg \vord$, then there exists $b \in \N$ such that $a \leq b \leq d$ and $b$ occurs in $\word_1$. Similarly, since $\uord \equiv_\rtg \vord$, then there exists $c \in \N$ such that $a \leq c \leq d$ and $c$ occurs in $\word_2$. Therefore, we use the defining relation of $\equiv$ to show that $\uord \equiv \vord$.
    
    Suppose now that $|\uord'|=|\vord'|>2$. Write $\uord' = a \uord_1'$ and $\vord' = \vord_1' b a \vord_2'$, for some $a,b \in \N$ and $\vord_1' \in \N^*$ such that $a \notin \supp{\vord_1'}$. Notice that $b \in \supp{\uord_1'}$.

    If $\uord$ factorises as $\uord = \uord_1 b \uord_2 a \uord_3$, for some $\uord_1,\uord_2,\uord_3 \in \N^*$ such that $\supp{\uord_1} \cap [[a,b]] = \emptyset$, then $b \in \supp{\word_1}$. On the other hand, if $\uord$ does not factorise in such a way, then there exists $c \in \supp{\word_1 \vord_1'} \cap [[a,b]]$, since $\uord \equiv_\ltg \vord$. Similarly, if $\uord$ factorises as $\uord = \uord_1 b \uord_2 a \uord_3$, for some $\uord_1,\uord_2,\uord_3 \in \N^*$ such that $\supp{\uord_3} \cap [[a,b]] = \emptyset$, then there exists $a \in \supp{\uord' \word_2}$ such that $a$ is to the right of the last occurrence of $b$ in $\uord$. As such, we have that $a \in \supp{\vord_2'\word_2}$. On the other hand, if $\uord$ does not factorise in such a way, then there exists $c \in \supp{\vord_2'\word_2} \cap [[a,b]]$, since $\uord \equiv_\rtg \vord$.
    
    Thus, we can use the defining relation of $\equiv$ to show that $\vord \equiv \word_1 \vord_1' a b \vord_2' \word_2$. Notice that this reasoning can be applied finitely many times to show that $\vord \equiv \word_1 a \vord'' \word_2$, for some $\vord'' \in \N^*$. Hence, by the induction hypothesis, $\uord \equiv \vord$.
    
    Thus, we have shown that $\uord \equiv_\ltg \vord$ and $\uord \equiv_\rtg \vord$ implies that $\uord \equiv \vord$, and as such, the congruences $\equiv$ and $\equiv_\mtg$ are the same.    
\end{proof}

\begin{proposition}
\label{prop:jtg_characterisation}
    Let $\uord,\vord \in \N^*$. Let $A_1,\dots A_k$ be all the intervals of $\supp{\uord}$ such that $A_i$ only contains simple letters and $A_i \cup \{a\}$ is not an interval, for any simple letter $a \notin A_i$. Then, $\uord \equiv_\jtg \vord$ if and only if $\cont{\uord} = \cont{\vord}$ and $\uord[A_i] \equiv_\hypo \vord[A_i]$ for all $1 \leq i \leq k$. 
\end{proposition}
\begin{proof} 
    Consider the congruence $\equiv$ given by $\uord \equiv \vord$ if and only if $\cont{\uord} = \cont{\vord}$ and $\uord[A_i] \equiv_\hypo \vord[A_i]$ for all $1 \leq i \leq k$ where $A_1,\dots A_k$ are as given in the statement of the proposition. 
    We first show that $\uord \equiv \vord$ implies $\uord \equiv_\jtg \vord$. 
    Let $x_1, \dots, x_l$ denote the non-simple letters of $\uord$ and $\vord$.
    Ranging $1 \leq j \leq l$, we can repeatedly use $\equiv_\rst$ to shift the leftmost $x_j$ to the beginning of both words, and then repeatedly use $\equiv_\lst$ to obtain words with the prefix $x_j^{|\uord|_{x_j}}$. Inductively, this allows us to show that
\begin{align*}
    \uord \equiv_\jst x_l^{|\uord|_{x_l}} &\cdots x_1^{|\uord|_{x_1}}\ol{\uord}, \quad \text{and} \\
    \vord \equiv_\jst x_l^{|\uord|_{x_l}} &\cdots x_1^{|\uord|_{x_1}}\ol{\vord}.
\end{align*}
Now, by applying the $\sylvh$ relations, we can order the simple variables so that the variables in the interval $A_i$ appear to the left of the simple variables in $A_{i+1}$ for each $1 \leq i < k$. Thus, 
\begin{align*}
    \uord \equiv_\jst x_l^{|\uord|_{x_l}} &\cdots x_1^{|\uord|_{x_1}}\uord[A_1]\cdots \uord[A_k], \quad \text{and} \\
    \vord \equiv_\jst x_l^{|\uord|_{x_l}} &\cdots x_1^{|\uord|_{x_1}}\vord[A_1]\cdots \vord[A_k].
\end{align*}
Now as $\cR_{\hypo} = \cR_{\sylv} \cup \cR_{\sylvh} \subset \cR_{\jtg}$, then, by the definition of $\equiv$, we have that $\uord[A_i] \equiv_\jtg \vord[A_i]$. Thus, $\uord \equiv_{\jtg} \vord$. 

On the other hand, first note that all the relations in $\cR_{\jtg}$ are content preserving, thus, if $\uord \equiv_\jtg \vord$, then $\cont{\uord} = \cont{\vord}$.
Let $\rord,\sord \in \N^*$. If $\rord \equiv_\jst \sord$, then $\ol{\rord} = \ol{\sord}$ by Proposition~\ref{prop:jst_characterisation} and if $\rord \equiv_\hypo \sord$ then, as $\hypo$ is compatible with restriction to alphabet intervals \cite[Theorem 5.6]{novelli_hypoplactic}, $\rord[A_i] \equiv_\hypo \sord[A_i]$ for all $1 \leq i \leq k$.

Thus, since the join-taiga congruence is the congruence join of $\equiv_\jst$ and $\equiv_\hypo$, $\uord \equiv_\jtg \vord$ implies $\uord[A_i] \equiv_\hypo \vord[A_i]$ for all $1 \leq i \leq k$. Hence, $\uord \equiv_\jst \vord$ if and only if $\uord \equiv \vord$.
\end{proof}


\subsection{Compatibility properties} \label{subsection:compatibility_properties}

Recall that the Baxter and hypoplactic congruences are defined, respectively, as the meet and join of the sylvester and \#-sylvester congruences \cite{giraudo_baxter,hnt_stalactic}. Some of the  compatibility properties (see Subsection~\ref{subsection:words}) of the Baxter and hypoplactic monoids have already been studied. We state the known results here and prove the remaining ones:

\begin{proposition} \label{prop:baxt_compatible}
    The Baxter monoid is compatible with standardisation, packing, restriction to alphabet intervals and the Schützenberger involution, but not with reverse-standardisation, restriction to alphabet subsets and word reversal.
\end{proposition}

\begin{proof}
    Compatibility with standardisation was shown in \cite[Proposition~3.2]{giraudo_baxter}, from which compatibility with packing follows. Compatibility with restriction to alphabet intervals was shown in \cite[Proposition~3.3]{giraudo_baxter}, and compatibility with the Schützenberger involution was shown in \cite[Proposition~3.4]{giraudo_baxter}.

    On the other hand, notice that $2121$ and $2211$ are $\baxt$-congruent, but $4231$ and $4321$ are not, hence $\baxt$ is not compatible with reverse-standardisation. Furthermore, $2132$ and $2312$ are $\baxt$-congruent, but $13$ and $31$ are not, hence $\baxt$ is not compatible with restriction to alphabet subsets. Similarly, notice that $2131$ and $2311$ are $\baxt$-congruent, but $1312$ and $1132$ are not, hence $\baxt$ is not compatible with word reversal.
\end{proof}

\begin{proposition} \label{prop:hypo_compatible}
    The hypoplactic monoid is compatible with standardisation, packing, restriction to alphabet intervals and the Schützenberger involution, but not with reverse-standardisation, restriction to alphabet subsets and word reversal.
\end{proposition}

\begin{proof}
    Compatibility with standardisation follows from \cite[Lemma~4.13 and Theorem~4.18]{novelli_hypoplactic}, from which compatibility with packing follows. Compatibility with restriction to alphabet intervals was shown in \cite[Theorem~5.6]{novelli_hypoplactic}, and compatibility with the Schützenberger involution was shown in \cite[Theorem~5.4]{novelli_hypoplactic}.

    On the other hand, as with the Baxter monoid, notice that $2121$ and $2211$ are $\hypo$-congruent, but $4231$ and $4321$ are not, hence $\hypo$ is not compatible with reverse-standardisation. Furthermore, notice that $213$ and $231$ are $\hypo$-congruent, but $13$ and $31$ are not, hence $\hypo$ is not compatible with restriction to alphabet subsets. Similarly, notice that $221$ and $212$ are $\hypo$-congruent, but $122$ and $212$ are not, hence $\hypo$ is not compatible with word reversal.
\end{proof}

We now study the compatibility properties of the monoids introduced in this section, and compare them with the Baxter and hypoplactic monoids:

\begin{proposition} \label{prop:mst_compatible}
    The meet-stalactic monoid is compatible with packing, restriction to alphabet subsets, word reversal and the Schützenberger involution, but not with standardisation and reverse-standardisation.
\end{proposition}

\begin{proof}
    It is immediate that $\mst$ is compatible with packing and restriction to alphabet subsets since there is no restriction placed on the letters and words occurring in its defining relations. Furthermore, $\mst$ is compatible with both word reversal and the Schützenberger involution, since if two words are $\lst$-congruent (resp. $\rst$-congruent), then their images under word reversal or the Schützenberger involution are $\rst$-congruent (resp. $\lst$-congruent), once again due to the lack of restrictions on the defining relations. 

    On the other hand, notice that while the words $1211$ and $1121$ are $\mst$-congruent, their standardised words $1423$ and $1243$ are not, hence $\mst$ is not compatible with standardisation. Similarly, their reverse-standardised words $3421$ and $3241$ are also not $\mst$-congruent, hence $\mst$ is not compatible with reverse-standardisation.
\end{proof}

An immediate consequence is that $\mst$ is compatible with restriction to alphabet intervals. Analogously, one can prove that:

\begin{proposition} \label{prop:jst_compatible}
    The join-stalactic monoid is compatible with packing, restriction to alphabet subsets, word reversal and the Schützenberger involution, but not with standardisation and reverse-standardisation.
\end{proposition}

\begin{proposition} \label{prop:mtg_compatible}
    The meet-taiga monoid is compatible with packing, restriction to alphabet intervals, word reversal and the Schützenberger involution, but not with standardisation, reverse-standardisation, and restriction to alphabet subsets.
\end{proposition}

\begin{proof}
    It is immediate that $\mtg$ is compatible with packing since packing can be viewed as an isomorphism that preserves the order on the support of a word, hence the defining relations of $\mtg$ still hold under it. Similarly, $\mtg$ is compatible with restriction to alphabet intervals, since its defining relations are of the form
    \[
        b \uord ad \vord c \equiv_\mst b \uord da \vord c,
    \]
    where $a \leq d$, $b,c \in [a,d]$ and $\uord,\vord \in \N^*$, hence restricting the alphabet to an interval that includes $a$ and $d$ still allows equality under $\equiv_\mtg$ (excluding $a$ or $d$ trivially gives equality). To show that $\mtg$ is compatible with word reversal and the Schützenberger involution, it suffices to see that the defining relations still hold under these operations, which is clear from their definition.
    
    On the other hand, notice that $1122$ and $1212$ are $\mtg$-congruent, however, their standardised words $1234$ and $1324$ are not, hence $\mtg$ is not compatible with standardisation. Similarly, notice that $2211$ and $2121$ are $\mtg$-congruent, however, their reverse-standardised words $4321$ and $4231$ are not, hence $\mtg$ is not compatible with reverse-standardisation. Furthermore, notice that $2132$ and $2312$ are $\mtg$-congruent, but $13$ and $31$ are not, hence $\mtg$ is not compatible with restriction to alphabet subsets.
\end{proof}
Analogously, one can prove that:

\begin{proposition} \label{prop:jtg_compatible}
    The join-taiga monoid is compatible with packing, restriction to alphabet intervals, word reversal and the Schützenberger involution, but not with standardisation, reverse-standardisation, and restriction to alphabet subsets.
\end{proposition}

In Table~\ref{table:compatibilities}, we give a summary of the results on compatibilities of the hypoplactic, Baxter, meet and join-stalactic and meet and join-taiga monoids.
{\renewcommand{\arraystretch}{1.2}\begin{table}[ht]
    \caption{Compatibility of plactic-like monoids, where \textbf{Std} stands for `standardisation'; \textbf{rStd} for `standardisation'; \textbf{Pack} for `packing'; \textbf{RAS} for `restriction to alphabet subsets'; \textbf{RAI} for `restriction to alphabet intervals'; \textbf{WR} for `word reversal'; \textbf{SI} for `Schützenberger involution'.}
    \label{table:compatibilities}
    \begin{tabular}{c|cccccccc}
        \toprule
        \textbf{Monoid} & \textbf{Std} & \textbf{rStd} & \textbf{Pack} & \textbf{RAS} & \textbf{RAI} & \textbf{WR} & \textbf{SI} & \textbf{Result} \\
        \midrule
        $\hypo$ & \ding{51} & \ding{55} & \ding{51} & \ding{55} & \ding{51} & \ding{55} & \ding{51} & Proposition~\ref{prop:hypo_compatible}\\ 
        $\baxt$ & \ding{51} & \ding{55} & \ding{51} & \ding{55} & \ding{51} & \ding{55} & \ding{51} & Proposition~\ref{prop:baxt_compatible}\\ 
        $\mst$ & \ding{55} & \ding{55} & \ding{51} & \ding{51} & \ding{51} & \ding{51} & \ding{51} & Proposition~\ref{prop:mst_compatible}\\ 
        $\jst$ & \ding{55} & \ding{55} & \ding{51} & \ding{51} & \ding{51} & \ding{51} & \ding{51} & Proposition~\ref{prop:jst_compatible}\\ 
        $\mtg$ & \ding{55} & \ding{55} & \ding{51} & \ding{55} & \ding{51} & \ding{51} & \ding{51} & Proposition~\ref{prop:mtg_compatible}\\ 
        $\jtg$ & \ding{55} & \ding{55} & \ding{51} & \ding{55} & \ding{51} & \ding{51} & \ding{51} & Proposition~\ref{prop:jtg_compatible}\\ 
        \bottomrule
    \end{tabular}
    \end{table}}


\section{Robinson--Schensted-like correspondences} \label{section:Robinson_Schensted_like_correspondences}

In this section, we introduce analogues of the Robinson--Schensted correspondence for the meet-stalactic and meet-taiga monoids. We then show how to extract representatives of the congruence classes from their $\psymb$-symbols and give iterative versions of the insertion algorithms.


\subsection{Definitions and correctness of the correspondences} \label{subsection:Definitions_and_correctness_of_the_correspondences}

We begin by defining the $\psymb$ and $\qsymb$-symbols and stating their properties, which then allow us to prove the correctness of the correspondences.


\subsubsection{The meet-stalactic correspondence} \label{subsubsection:The_meet_stalactic_correspondence}

\begin{definition} \label{definition:mstpandqsymbols}
    For a word $\word \in \N^*$, the \emph{meet-stalactic $\psymb$-symbol} of $\word$ is the pair $\pmst{\word}:=(\plst{\word}, \prst{\word})$ of stalactic tableaux, while the \emph{meet-stalactic $\qsymb$-symbol} of $\word$ is the pair $\qmst{\word}:=(\qlst{\word}, \qrst{\word})$ of (respectively) increasing and decreasing patience-sorting tableaux.
\end{definition}

\begin{example} \label{example:meet_stalactic_p_and_q_symbols}
The meet-stalactic $\psymb$-symbol of $212511354$ is
\begin{equation*}
    \left(\tikz[tableau]\matrix{
		2 \& 1 \& 5 \& 3 \& 4\\
		2 \& 1 \& 5 \&   \&  \\
          \& 1 \&   \&   \&  \\ 
	};, 
    \tikz[tableau]\matrix{
		2 \& 1 \& 3 \& 5 \& 4\\
		2 \& 1 \&   \& 5 \&  \\
          \& 1 \&   \&   \&  \\ 
	};\right),
\end{equation*}
and its meet-stalactic $\qsymb$-symbol is
\begin{equation*}
    \left(\tikz[tableau]\matrix{
		1 \& 2 \& 4 \& 7 \& 9\\
		3 \& 5 \& 8 \&   \&  \\
          \& 6 \&   \&   \&  \\ 
	};,
    \tikz[tableau]\matrix{
		3 \& 6 \& 7 \& 8 \& 9\\
		1 \& 5 \&   \& 4 \&  \\
          \& 2 \&   \&   \&  \\ 
	};\right).
\end{equation*}
\end{example}

\begin{proposition} \label{prop:pmst_pair_plst_prst}
    Let $\uord,\vord \in \N^*$. Then $\uord \equiv_\mst \vord$ if and only if $\pmst{\uord} = \pmst{\vord}$.
\end{proposition}

\begin{proof}
    By definition, $\uord \equiv_\mst \vord$ if and only if $\uord \equiv_\lst \vord$ and $\uord \equiv_\rst \vord$. This is equivalent to $\uord$ and $\vord$ having the same left-stalactic and right-stalactic $\psymb$-symbol, that is, $\plst{\uord} = \plst{\vord}$ and $\prst{\uord} = \prst{\vord}$. In turn, by the definition of the \hyperref[definition:mstpandqsymbols]{meet-stalactic $\psymb$-symbol}, this is equivalent to $\pmst{\uord} = \pmst{\vord}$.
\end{proof}

\begin{definition} \label{definition:TwinStalacticTableauxCondition}
    Let $T_L,T_R$ be stalactic tableaux. We say $(T_L,T_R)$ are a \emph{pair of twin stalactic tableaux} if $\cont{T_R} = \cont{T_L}$ and for each simple column labelled $c$ and any other column labelled $d$, if $\rho_{T_L}(c) < \rho_{T_L}(d)$, then $\rho_{T_R}(c) < \rho_{T_R}(d)$.
\end{definition}

The condition on the order of columns is equivalent to the following: if $\rho_{T_R}(d) < \rho_{T_R}(c)$, then $\rho_{T_L}(d) < \rho_{T_L}(c)$. Notice that, in particular, the order of simple columns, from left-to-right, is the same in pairs of twin stalactic tableaux. The $\psymb$-symbol given in Example~\ref{example:meet_stalactic_p_and_q_symbols} is a pair of twin stalactic tableaux. In fact, we have the following:

\begin{proposition} \label{prop:pmst_pair_twin}
For any $\word \in \N^*$, the meet-stalactic $\psymb$-symbol $(\plst{\word},\prst{\word})$ of $\word$ is a pair of twin stalactic tableaux.
\end{proposition}
\begin{proof}
    Notice that, since $\plst{\word}$ and $\prst{\word}$ are computed from the same word $\word$, then $\cont{\plst{\word}} = \cont{\prst{\word}}$. Let $a,x \in \supp{\word}$ be such that $a$ is simple. Suppose that $\rho_{\plst{\word}}(a) < \rho_{\plst{\word}}(x)$, which implies that the only occurrence of $a$ in $\word$ is to the left of the leftmost and, therefore, all occurrences of $x$ in $\word$. In particular, $a$ only occurs to the left of the rightmost occurrence of $x$ in $\word$, which, by Algorithm~\hyperref[alg:RightStalacticLeftInsertion]{\textsf{rStLI}}, implies $\rho_{\prst{\word}}(a) < \rho_{\prst{\word}}(x)$. Thus, $(\plst{\word},\prst{\word})$ is a pair of twin stalactic tableaux.
\end{proof}

\begin{definition} \label{definition:TwinPatienceSortingTableauxCondition}
    Let $(S_L,S_R)$ be a pair of (respectively) increasing and decreasing patience-sorting tableaux. We say $(S_L,S_R)$ are a \emph{pair of twin patience-sorting tableaux} if
    \begin{enumerate}[(i)]
        \item The columns of $S_L$ are in bijection with the columns of $S_R$, with the content of the columns being preserved by the bijection;
        \item for each pair of corresponding columns $(c_L,c_R)$ and $(d_L,d_R)$ of $(S_L,S_R)$ such that $c_L$ and $c_R$ are simple, if $c_L$ appears to the left of $d_L$ in $S_L$ then $c_R$ appears to the left of $d_R$ in $S_R$.
    \end{enumerate}
\end{definition}

The condition on the order of columns is equivalent to the following: if $c_R$ appears to the right of $d_R$ in $S_R$ then $c_L$ appears to the right of $d_L$ in $S_L$. The $\qsymb$-symbol given in Example~\ref{example:meet_stalactic_p_and_q_symbols} is a pair of twin patience-sorting tableaux. In fact, we have the following:

\begin{proposition} \label{prop:qmst_pair_twin}
For any $\word \in \N^*$, the meet-stalactic $\qsymb$-symbol $(\qlst{\word},\qrst{\word})$ of $\word$ is a pair of twin patience-sorting tableaux.
\end{proposition}
\begin{proof}
    For each $x \in \supp{\word}$, by Lemma~\ref{lemma:pxst_qxst_same_shape}, there exists a column in $\qlst{\word}$ and a column in $\qrst{\word}$, each containing all the positions of where $x$ appears in $\word$, when read left-to-right. These columns must therefore have the same content, and hence the first property holds. The second property then follows from Proposition~\ref{prop:pmst_pair_twin}, Lemma~\ref{lemma:pxst_qxst_same_shape}.
\end{proof}

\begin{theorem}
    The map $\word \mapsto (\pmst{\word}, \qmst{\word})$ is a bijection between the elements of $\N^*$ and the set formed by the pairs $\bigl((T_L,T_R),(S_L,S_R)\bigr)$ where
    \begin{enumerate}[(i)]
        \item $(T_L,T_R)$ is a pair of twin stalactic tableaux.
        \item $(S_L,S_R)$ is a pair of twin patience-sorting tableau.
        \item $(T_L,T_R)$ and $(S_L,S_R)$ have the same shape.
        \item The $i$-th column of $S_L$ is in bijection with the $\rho_{T_R}(\rho_{T_L}^{-1}(i))$-th column of $S_R$.
    \end{enumerate}
\end{theorem}
\begin{proof}
    By Propositions~\ref{prop:pmst_pair_twin} and \ref{prop:qmst_pair_twin}, we have that the output of the map satisfies the first two conditions. Moreover, we can see that they also satisfy the third condition as, by Lemma~\ref{lemma:pxst_qxst_same_shape}, $T_L$ and $S_L$ have the same shape and $T_R$ and $S_R$ have the same shape. Again by Lemma~\ref{lemma:pxst_qxst_same_shape}, we have that the $i$-th column of $S_L$ (resp. $S_R$) gives the positions of the letter $\rho_{T_L}^{-1}(i)$ (resp. $\rho_{T_R}^{-1}(i)$) in $\word$, and hence the $i$-th column of $S_L$ is in bijection with the $\rho_{T_R}(\rho_{T_L}^{-1}(i))$-th column of $S_R$.
    Thus, the map is well-defined.

    The map is injective due to the stalactic Robinson--Schensted correspondences (Theorems~\ref{theorem:robinson_left_stalactic} and \ref{theorem:robinson_right_stalactic}), as from either of the pairs $(\plst{\word},\qlst{\word})$ or $(\prst{\word},\qrst{\word})$ we can reconstruct $\word$.

    For surjectivity, let $\bigl((T_L,T_R),(S_L,S_R)\bigr)$ satisfy conditions \textit{(i)--(iv)}. Using the stalactic Robinson--Schensted correspondences, we have that there exist unique $\uord,\vord \in \N^*$ such that $\bigl(\plst{\uord},\qlst{\uord}\bigr) = (T_L,S_L)$ and $\bigl(\prst{\vord},\qrst{\vord}\bigr) = (T_R,S_R)$. Clearly, by condition \textit{(i)}, we have $\cont{\uord} = \cont{\vord}$. Suppose $a = \overrightarrow{\rho_\uord}^{-1}(i)$ for some $1 \leq i \leq |\supp{\uord}|$. Then, by Lemma~\ref{lemma:pxst_qxst_same_shape}, the $i$-th column of $S_L$ is the unique column of $S_L$ which gives the positions of $a$ in $\uord$. Moreover, by condition \textit{(iv)}, the $i$-th column of $S_L$ is also the unique column of $S_L$ which gives the positions of $\overleftarrow{\rho_\vord}^{-1}(\rho_{T_R}(\rho_{T_L}^{-1}(i))) = \rho_{T_L}^{-1}(i) = a$ in $\vord$. As such, $\overrightarrow{\rho_\vord}$ is completely determined by $(T_L,S_L)$, which implies $\plst{\vord} = T_L$, and, by Lemma~\ref{lemma:pxst_qxst_same_shape}, $\qlst{\vord} = S_L$. Thus, $\uord = \vord$ and $\bigl(\pmst{\uord},\qmst{\uord}\bigr) = \bigl((T_L,T_R),(S_L,S_R)\bigr)$.
\end{proof}


\subsubsection{The meet-taiga correspondence} \label{subsubsection:The_meet_taiga_correspondence}

\begin{definition} \label{definition:mtgpandqsymbols}
    For a word $\word \in \N^*$, the \emph{meet-taiga $\psymb$-symbol} of $\word$ is the pair $\pmtg{\word}:=(\pltg{\word}, \prtg{\word})$ of BSTMs, while the \emph{meet-taiga $\qsymb$-symbol} of $\word$ is the pair $\qmtg{\word}:=(\qltg{\word}, \qrtg{\word})$ of (respectively) increasing and decreasing binary trees.
\end{definition}

\begin{example} \label{example:meet_taiga_p_and_q_symbols}
The meet-taiga $\psymb$-symbol of $451423412$ is
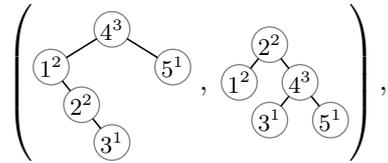
\begin{equation*}
    \left(\begin{tikzpicture}[tinybst,baseline=-8mm]
        \node {$4^3$}
        child { node {$1^2$} 
            child[missing]
            child { node {$2^2$}
                child[missing]
                child { node {$3^1$} }
            }
        }
        child[missing]
        child { node {$5^1$} };
    \end{tikzpicture}\;,\;
    \begin{tikzpicture}[tinybst, baseline=-6mm]
        \node {$2^2$}
        child { node {$1^2$} }
        child { node {$4^3$}
            child { node {$3^1$} }
            child { node {$5^1$} }
        };
    \end{tikzpicture}\right),
\end{equation*}
and its meet-taiga $\qsymb$-symbol is
\begin{equation*}
    \left(\begin{tikzpicture}[microbst,baseline=-8mm]
        \node {$1,4,7$}
        child { node {$3,8$} 
            child[missing]
            child { node {$5,9$}
                child[missing]
                child { node {$6$} }
            }
        }
        child[missing]
        child { node {$2$} };
    \end{tikzpicture}\;,\;
    \begin{tikzpicture}[microbst, baseline=-6mm]
        \node {$5,9$}
        child { node {$3,8$} }
        child { node {$1,4,7$}
            child { node {$6$} }
            child { node {$2$} }
        };
    \end{tikzpicture}\right).
\end{equation*}
\end{example}

\begin{proposition} \label{MtgClassesPsymbolCorr}
    Let $\uord,\vord \in \N^*$. Then $\uord \equiv_\mtg \vord$ if and only if $\pmtg{\uord} = \pmtg{\vord}$.
\end{proposition}
\begin{proof}
    By definition, $\uord \equiv_\mtg \vord$ if and only if $\uord \equiv_\ltg \vord$ and $\uord \equiv_\rtg \vord$. This is equivalent to $\uord$ and $\vord$ having the same left-taiga and right-taiga $\psymb$-symbol, that is, $\pltg{\uord} = \pltg{\vord}$ and $\prtg{\uord} = \prtg{\vord}$. In turn, by the definition of the \hyperref[definition:mtgpandqsymbols]{meet-taiga $\psymb$-symbol}, this is equivalent to $\pmtg{\uord} = \pmtg{\vord}$.
\end{proof}

\begin{definition} \label{definition:TwinBTMCondition}
    Let $T_L,T_R$ be BTMs, with the same number of nodes. We say $(T_L,T_R)$ is a \emph{pair of twin binary trees with multiplicities} (pair of twin BTMs), if for all $i$,
    \begin{enumerate}[(i)]
        \item the $i$-th node of $T_L$ and the $i$-th node of $T_R$ have the same multiplicities.
        \item if the $i$-th node of $T_L$ has multiplicity 1 and has a left (resp. right) child then the $i$-th node of $T_R$ does not have a left (resp. right) child.
    \end{enumerate}
\end{definition}

Notice that condition (ii) is equivalent to saying that if $r_i$ has multiplicity 1 then if $r_i$ has a left (resp. right) child then $l_i$ does not have a left (resp. right) child.

\begin{definition} \label{definition:TwinBSTMCondition}
    Let $T_L,T_R$ be BSTMs. We say $(T_L,T_R)$ is a \emph{pair of twin binary search trees with multiplicities} (pair of twin BSTMs) if $\cont{T_R} = \cont{T_L}$ and the shape of $(T_L,T_R)$ is a pair of twin BTMs.
\end{definition}

The $\psymb$-symbol given in Example~\ref{example:meet_taiga_p_and_q_symbols} is a pair of twin BSTMs. In fact, we have the following:

\begin{proposition} \label{prop:pmtg_pair_twin}
For any $\uord \in \N^*$, the $\psymb_\mtg$-symbol $(\pltg{\uord},\prtg{\uord})$ of $\uord$ is a pair of twin BSTMs.
\end{proposition}
\begin{proof}
    As $\pltg{\uord}$ and $\prtg{\uord}$ are both computed by the same word $\uord$, $\cont{\pltg{\uord}} = \cont{\prtg{\uord}}$. Moreover, suppose $a$ is a simple variable in $\uord$ and that $a$ has a left child in $\pltg{\uord}$. Let $x \in \supp{\uord}$ be the rightmost descendant of the left child of $a$ in $\pltg{\uord}$, then $z \leq x < a$ for all $z < a$.
    
    As $x$ is a descendant of $a$, we have that the leftmost appearance of $x$ appears to the right of $a$ in $\uord$ and as $a$ is simple, the rightmost appearance of $x$ appears to the right of $a$ in $\uord$. Therefore, $a$ is a right descendant of $x$ in $\prtg{\uord}$ as there does not exist a $z$ such that $x < z < a$. 
    
    Thus, $a$ does not have a left child in $\prtg{\uord}$, as if $z$ is a descendant of $x$ and $z < a$, then $z < x$ and hence $z$ is a left descendant of $x$ and therefore not a left descendant of $a$. The proof for right children is similar.
\end{proof}

\begin{definition} \label{definition:TwinBinaryTreesOverSetsCondition}
    Let $(S_L,S_R)$ be a pair of (respectively) increasing and decreasing binary trees over sets, with the same number of nodes. We say $(S_L,S_R)$ are a \emph{pair of twin binary trees over sets} (pair of twin BTSs) if, for all $i$,
    \begin{enumerate}[(i)]
        \item the $i$-th node of $S_L$ has the same label as the $i$-th node of $S_R$.
        \item if the $i$-th node of $S_L$ is labelled by a set of cardinality 1 and has a left (resp. right) child, then the $i$-th node of $S_R$ does not have a left (resp. right) child.
    \end{enumerate}
\end{definition}

The $\qsymb$-symbol given in Example~\ref{example:meet_taiga_p_and_q_symbols} is a pair of twin BTSs. In fact, we have the following:

\begin{proposition} \label{prop:qmtg_pair_twin}
For any $\word \in \N^*$, the meet-taiga $\qsymb$-symbol $(\qltg{\word},\qrtg{\word})$ of $\word$ is a pair of twin BTSs.
\end{proposition}
\begin{proof}
    Let $\supp{\word} = \{x_1 < \dots < x_k\}$. For each $x_i \in \supp{\word}$, by Lemma~\ref{lemma:pxtg_qxtg_same_shape}, 
    the $i$-th node of $\qltg{\word}$ and the $i$-th node of $\qrtg{\word}$ each contain all the positions of where $x_i$ appears in $\word$, when read left-to-right. Thus, the $i$-th nodes of $S_L$ and $S_R$ share the same label. The second property then follows from Proposition~\ref{prop:pmtg_pair_twin} and Lemma~\ref{lemma:pxtg_qxtg_same_shape}.
\end{proof}

\begin{theorem}
    The map $\word \mapsto (\pmtg{\word}, \qmtg{\word})$ is a bijection between the elements of $\N^*$ and the set formed by the pairs $\bigl((T_L,T_R),(S_L,S_R)\bigr)$ where
    \begin{enumerate}[(i)]
        \item $(T_L,T_R)$ is a pair of twin BSTMs.
        \item $(S_L,S_R)$ is a pair of twin BTSs such that the union of the sets labelling $S_L$, and therefore $S_R$, is the interval $[m]$, where $m$ is the sum of the multiplicities of $T_L$ (or $T_R$).
        \item $(T_L,T_R)$ and $(S_L,S_R)$ have the same underlying pair of binary trees shape.
        \item the multiplicity of the $i$-th node of $T_L$ (resp. $T_R$) is the cardinality of the set labelling the $i$-th node of $S_L$ (resp. $S_R$).
    \end{enumerate}
\end{theorem}
\begin{proof}
    Let $\supp{\word} = \{x_1<\dots<x_k\}$. By Propositions~\ref{prop:pmtg_pair_twin} and \ref{prop:qmtg_pair_twin}, we have that the output of the map satisfies the first two conditions. By Lemma~\ref{lemma:pxtg_qxtg_same_shape}, the third and fourth conditions are satisfied, as $T_L$ (resp. $T_R$) and $S_L$ (resp. $S_R$) have the same shape, and furthermore, the $i$-th node of $S_L$ (resp. $S_R$) gives the positions of the letter $x_i$ in $\word$, and hence the multiplicity of the $i$-th node of $T_L$ (resp. $T_R$) is the cardinality of the set labelling the $i$-th node of $S_L$  (resp. $S_R$).
    Thus, the map is well-defined.

    The map is injective due to the taiga Robinson--Schensted correspondences 
    (Theorems~\ref{theorem:robinson_left_taiga} and \ref{theorem:robinson_right_taiga}), 
    as from either of the pairs $(\pltg{\word},\qltg{\word})$ or $(\prtg{\word},\qrtg{\word})$ we can reconstruct $\word$.

    For surjectivity, let $\bigl((T_L,T_R),(S_L,S_R)\bigr)$ satisfy conditions \textit{(i)--(iv)}. Using the taiga Robinson--Schensted correspondences, there exist unique $\uord,\vord \in \N^*$ such that $\bigl(\pltg{\uord},\qltg{\uord}\bigr) = (T_L,S_L)$ and $\bigl(\prtg{\vord},\qrtg{\vord}\bigr) = (T_R,S_R)$. 
    Clearly, by condition \textit{(i)}, we have $\cont{\uord} = \cont{\vord}$. 
    Let $\supp{\uord} = \supp{\vord} = \{x_1 < \dots < x_k\}$. Then, by Lemma~\ref{lemma:pxst_qxst_same_shape}, 
    for all $i \in [k]$, the label of the $i$-th node of $S_L$ gives the positions of $x_i$ in $\uord$. By condition \textit{(ii)}, the $i$-th node of $S_L$ has the same label as the $i$-th node of $S_R$, thus the label of the $i$-th node of $S_R$ gives the positions of $x_i$ in $\uord$. As such, $\prtg{\uord} = T_R$ and $\qrtg{\uord} = S_R$. Therefore, $\uord = \vord$ and $\bigl(\pmtg{\uord},\qmtg{\uord}\bigr) = \bigl((T_L,T_R),(S_L,S_R)\bigr)$.
\end{proof}


\subsection{Extraction algorithms} \label{subsection:Extraction_algorithms}

In the previous subsection, we have shown how to obtain a word from its meet-stalactic or meet-taiga $\psymb$ and $\qsymb$-symbols. We now show how to obtain words only from the $\psymb$-symbols, without the need to compute the $\qsymb$-symbols. We give a deterministic algorithm for the meet-stalactic case and a non-deterministic one for the meet-taiga case.

\subsubsection{The meet-stalactic extraction algorithm} \label{subsubsection:The_meet_stalactic_extraction_algorithm}

Algorithm~\hyperref[ExtractMeetStalactic]{\textup{\textsf{EmSt}}} takes a pair of twin stalactic tableaux $(T_L,T_R)$ and outputs a word in the $\equiv_\mst$-class corresponding to $(T_L,T_R)$.

\begin{algorithm}[H] \label{ExtractMeetStalactic}
        \DontPrintSemicolon
        \KwIn{A pair of twin stalactic tableaux $(T_L,T_R)$.}
        \KwOut{A word in the $\equiv_\mst$-class corresponding to $(T_L,T_R)$.}
        \BlankLine
        let $\uord := \varepsilon$; \\
        \While{$T_L \neq \perp$}{
        let $a$ be the label of the leftmost column of $T_L$; \\
        let $m_a$ be the length of the column labelled $a$; \\
        remove the leftmost column of $T_L$; \\
        \eIf{$m_a >1$}{
        set $\uord := \uord a^{m_a-1}$;}{
        \While{$a \in \supp{T_R}$}{
        let $b$ be the label of the leftmost column of $T_R$; \\
        set $\uord := \uord b$; \\
        remove the leftmost column of $T_R$;}}
        }
        let $\word$ be the reading of the top row of $T_R$ from left-to-right;  \\
        \Return the word $\uord \word$.
        \caption{\textit{Extract Meet-Stalactic} (\textsf{EmSt}).}
\end{algorithm}

\begin{proposition} \label{ExtractTwinToMst}
    For any pair of twin stalactic tableaux $(T_L, T_R)$ as input, Algorithm~\hyperref[ExtractMeetStalactic]{\textup{\textsf{EmSt}}} computes a word belonging to the $\equiv_\mst$-equivalence class encoded by $(T_L, T_R )$.
\end{proposition}

\begin{proof}
    The key to this proof is to show that the order of first and last occurrences of letters in the output of \hyperref[ExtractMeetStalactic]{\textsf{EmSt}} is given, respectively, by the row readings of $T_L$ and $T_R$.

    Notice that, when computing \hyperref[ExtractMeetStalactic]{\textsf{EmSt}}, if one encounters a simple column, then it follows from the definition of \hyperref[definition:TwinStalacticTableauxCondition]{pairs of twin stalactic tableaux} that the columns to be removed from $T_R$, in the simple column subroutine, have already been removed from $T_L$. As such, $\overrightarrow{\rho_{\uord \word}} = \rho_{T_L}$. 

    Recall that $T_L$ has a simple column labelled by $a$ if and only if $T_R$ does as well, hence they have the same simple columns. If $T_L$ and $T_R$ have no simple columns, then $\supp{\word} = \supp{\uord}$, hence $\overleftarrow{\rho_{\uord \word}} = \overleftarrow{\rho_{\word}}$. Furthermore, $\word$ is the row reading of $T_R$, hence $\overleftarrow{\rho_{\uord \word}} = \rho_{T_R}$.

    Suppose $T_L$ has $n$ simple columns, that is, 
    \[\mathrm{col}(T_L) = \vord_1 a_1 \vord_2 a_2 \cdots a_n \vord_{n+1},\]
    where $\vord_i = \prod_{j=1}^{\alpha_i} b_{i,j}^{|T_L|_{b_{i,j}}}$, with $a_1, \dots, a_n, b_{i,j} \in \N$ all distinct, $\alpha_i \in \N$. Informally, $a_1, \dots, a_n$ correspond to the simple columns of $T_L$ and $\vord_i$ corresponds to the block of non-simple columns between $a_{i-1}$ and $a_i$ (with the exception of $\vord_1$ and $\vord_{n+1}$, which are before $a_1$ and after $a_n$, respectively). For each $\vord_i$, $\alpha_i$ is its number of columns and $b_{i,j}$ is the label of the cells of its $j$-th column (counting from left-to-right). 

    By the definition of \hyperref[definition:TwinStalacticTableauxCondition]{pairs of twin stalactic tableaux}, the top row, read from left-to-right, of $T_R$ is of the form $\rord_1 a_1 \rord_2 a_2 \cdots a_n \rord_{n+1}$, where $\supp{\rord_i} \subseteq \supp{\vord_1 \cdots \vord_i}$. Informally, the $i$-th non-simple column block of $T_R$ has columns from the blocks of non-simple columns of $T_L$ of index less than or equal to $i$ (but not necessarily from every single one of those blocks). Thus, the columns of each block of non-simple columns of $T_L$ are distributed throughout the blocks of non-simple columns of $T_R$ of higher or equal index.
    
    Notice that, by the algorithm, the word $\word$ is equal to $\rord_{n+1}$. Furthermore, the word $\uord$ is of the form
    \[
        \prod_{i=1}^{n+1} \left(\prod_{j=1}^{\alpha_i} b_{i,j}^{|T_L|_{b_{i,j}}-1}\right) \sord_i 
    \]
    where $\sord_i = \rord_i a_i$ for $1 \leq i \leq n$ and $\sord_{n+1} = \varepsilon$. In other words, $\uord$ is obtained by reading, from left-to-right, each non-simple column block of $T_L$ while excluding the last cell of each column, then reading the top row of the corresponding non-simple column block of $T_R$ and then reading the following simple column. As each $\supp{\vord_i} \subseteq \supp{\rord_i \cdots \rord_{n+1}}$, we have that $\overleftarrow{\rho_{\uord \word}} = \rho_{T_R}$.

    The result follows from the fact that $\overrightarrow{\rho_{\uord \word}} = \rho_{\plst{\uord \word}}$ and $\overleftarrow{\rho_{\uord \word}} = \rho_{\prst{\uord \word}}$. 
    As such, we have that $\plst{\uord \word} = T_L$ and $\prst{\uord \word} = T_R$, and hence $\pmst{\uord \word} = (T_L,T_R)$.
\end{proof}

As a consequence of Propositions~\ref{prop:pmst_pair_plst_prst}, \ref{prop:pmst_pair_twin} and \ref{ExtractTwinToMst}, we have the following:

\begin{corollary} \label{mstClasstoTwin}
    For any $n,m \geq 0$, there is a bijection between the set of $\mst$-equivalence classes of words of length $m$ over $[n]^*$ and pairs of twin stalactic tableaux with $m$ blocks labelled by letters from $[n]$.
\end{corollary}


\subsubsection{The meet-taiga extraction algorithm} \label{subsubsection:The_meet_taiga_extraction_algorithm}

Algorithm \hyperref[ExtractMeetTaiga]{\textsf{EmTg}} takes a pair of taiga trees $(T_L,T_R)$ and outputs a word in the $\equiv_\mtg$-class corresponding to $(T_L,T_R)$.

\begin{algorithm}[H] \label{ExtractMeetTaiga}
        \DontPrintSemicolon
        \KwIn{A pair of twin BSTMs $(T_L,T_R)$.}
        \KwOut{A word in the $\equiv_\mtg$-class corresponding to $(T_L,T_R)$.}
        \BlankLine
        let $\uord := \varepsilon$; \\
        set $(U_L,U_R) = (T_L,T_R)$; \\
        \While{$(U_L,U_R) \neq (\perp,\perp)$}{
        \eIf{all leaves in $U_R$ share letters with nodes in $U_L$}
        {choose a letter $a$ that labels the root node of a tree in $U_L$ such that $a$ has multiplicity greater than $1$ or $a$ labels a leaf in $U_R$; \\
        let $m_a$ be the multiplicity of $a$; \\
        \eIf{$m_a > 1$}{
        set $\uord := \uord a^{m_a-1}$;}
        {set $\uord := \uord a$; \\
        remove the node labelled $a$ from $U_R$;
        }
        remove the node labelled $a$ from $U_L$;}
        {let $a$ be the letter labelling a leaf in $U_R$ such that $a$ does not label any node in $U_L$; \\
        set $\uord := \uord a$; \\
        remove the node labelled $a$ from $U_R$;}
        }
        \Return the word $\uord$.
        \caption{\textit{Extract Meet-Taiga} (\textsf{EmTg}).}
\end{algorithm}

Notice that, in contrast with \hyperref[ExtractMeetStalactic]{\textsf{EmSt}}, \hyperref[ExtractMeetTaiga]{\textsf{EmTg}} is a non-deterministic algorithm, since there may be several choices for $a$ in steps 5 and 14.

\begin{proposition} \label{ExtractTwinToMtg}
    For any pair of twin BSTMs $(T_L, T_R)$ as input, Algorithm \hyperref[ExtractMeetTaiga]{\textup{\textsf{EmTg}}} computes a word belonging to the $\equiv_\mtg$-equivalence class encoded by $(T_L, T_R )$.
\end{proposition}
\begin{proof}
    First, we show that \hyperref[ExtractMeetTaiga]{\textsf{EmTg}} always terminates. It is clear that, in each step of the algorithm, a node is removed from $U_L$ or $U_R$. Thus, it suffices to show that while running \hyperref[ExtractMeetTaiga]{\textsf{EmTg}}, if all leaves in $U_R$ share letters with nodes in $U_L$, and all root nodes of trees in $U_L$ are simple, then at least one root node in $U_L$ shares a letter with a leaf in $U_R$.
    
    For contradiction, suppose that all the root nodes of the trees in $U_L$ are simple, and no root node in $U_L$ shares a letter with a leaf in $U_R$. Let $\{s_1 < \cdots < s_k\}$ be the letters of the root nodes of the trees in $U_L$. It follows from the definition of \hyperref[definition:TwinBSTMCondition]{pairs of twin BSTMs} that the node labelled $s_i$ cannot have left (resp. right) subtrees in both $U_L$ and $U_R$. Therefore, the node labelled $s_1$ cannot have a left child in $U_R$, since the letter labelling this child must be less than $s_1$ and, in $U_L$, all nodes with letters less than $s_1$ must be in the left subtree of the root node labelled $s_1$. By symmetric reasoning, one can show that $s_k$ does not have a right child in $U_R$.
    
    Let $s_i$ be the least letter such that the node labelled $s_i$ has a left subtree in $U_R$, and let $a$ be the greatest letter labelling a node of said subtree. Therefore, the node labelled $a$ does not have a right child, hence, by hypothesis, it is not the node labelled $s_{i-1}$. Thus we have $s_{i-1} < a < s_i$ and, as such, in $U_L$, the node labelled $a$ must be a right descendant of $s_{i-1}$ or a left descendant of $s_i$. By the definition of \hyperref[definition:TwinBTMCondition]{pairs of twin BTMs}, this contradicts our hypothesis. As such, we can conclude that at least one root node in $U_L$ shares a letter with a leaf in $U_R$.
    
    Now, we show that the meet-taiga $\psymb$-symbol of the output $\uord$ of \hyperref[ExtractMeetTaiga]{\textsf{EmTg}} coincides with its input. Clearly, $T_L$, $T_R$, $\pltg{\uord}$ and $\prtg{\uord}$ all share the same content. Furthermore, the root node of $T_L$ has the same label as the root node of $\pltg{\uord}$, which is the first letter of $\uord$, and therefore read from the root node of $T_L$ by the algorithm. Similarly, the root node of $T_R$ has the same label as the root node of $\prtg{\uord}$, which is the last letter of $\uord$, and therefore read from the last leaf of $U_R$ by the algorithm, that is, the root node of $T_R$.
    
    Suppose, in order to obtain a contradiction, that $T_L$ and $\pltg{\uord}$ are different, and that they are equal up to depth $k \geq 1$. As such, they share their binary tree shape (without multiplicities) up to depth $k+1$, since $T_L$ and $\pltg{\uord}$ have the same inorder reading. Hence, there must be a node of depth $k$ common to both $T_L$ and $\pltg{\uord}$ such that at least one of its children is labelled differently in each tree. Assume, without loss of generality, that said node is labelled $i$ in $T_L$ and labelled $j$ in $\pltg{\uord}$. 
    Then, by the definition of a BSTM, there must be a descendant of node labelled $i$ (resp. $j$) in $T_L$ (resp. $\pltg{\uord}$) labelled by $j$ (resp. $i$). On one hand, we can deduce that $\overrightarrow{\rho_\uord}(i) < \overrightarrow{\rho_\uord}(j)$, since \hyperref[ExtractMeetTaiga]{\textsf{EmTg}} will always read first any ancestor of the node labelled $j$. On the other hand, we can deduce that $\overrightarrow{\rho_\uord}(i) > \overrightarrow{\rho_\uord}(j)$, since the insertion of $i$ as the label of a descendant of the node labelled $j$ can only be done after reading $j$. We obtain a contradiction, hence $T_L = \pltg{\uord}$.
    
    The case of $T_R$ and $\prtg{\uord}$ is similarly proven. By the same reasoning as before, assuming $T_R$ and $\prtg{\uord}$ are different, then there exists a common node to both trees, with at least one differently labelled child, $i$ in $T_R$ and $j$ in $\prtg{\uord}$. On one hand, we can conclude that $\overleftarrow{\rho_\uord}(i) > \overleftarrow{\rho_\uord}(j)$, since \hyperref[ExtractMeetTaiga]{\textsf{EmTg}} will output a letter if and only if it labels a node in $U_R$, and it will only delete the node labelled $i$ when it is a leaf in $U_R$, so it will delete the node labelled $j$ first. On the other hand, we can conclude that $\overleftarrow{\rho_\uord}(i) < \overleftarrow{\rho_\uord}(j)$, since by \hyperref[alg:RightStalacticLeftInsertion]{\textsf{rStLI}}, we read the word $\uord$ from right-to-left, and insert $i$ as a descendant of the node labelled $j$ only after reading $j$. We obtain a contradiction, hence $T_R = \prtg{\uord}$.
\end{proof}

As a consequence of Propositions~\ref{MtgClassesPsymbolCorr}, \ref{prop:pmtg_pair_twin} and \ref{ExtractTwinToMtg}, we have the following:
\begin{corollary}
    For any $n,m \geq 0$, there is a bijection between the set of $\mtg$-equivalence classes of words of length $m$ over $[n]^*$ and pairs of twin BSTMs, labelled by letters from $[n]$ and whose sums of multiplicities is $m$.    
\end{corollary}


\subsection{Iterative insertion algorithms} \label{subsection:Iterative_insertion_algorithms}

We now introduce iterative versions of our insertion algorithms which allow us to compute a pair of twin stalactic tableaux from a word, while reading it only in one direction. Thus, we can easily compute the concatenation of two words under the meet-stalactic and meet-taiga congruences. Furthermore, we can compute the product of two pairs of twin stalactic tableaux (resp. pairs of twin BSTMs) $(T_L,T_R)$ and $(T'_L,T'_R)$ by applying \hyperref[ExtractMeetStalactic]{\textsf{EmSt}} (resp. \hyperref[ExtractMeetTaiga]{\textsf{EmTg}}) to the second pair and inserting the letters of the resulting word, from left-to-right, into the first pair.

\subsubsection{The meet-stalactic iterative insertion algorithm} \label{subsubsection:The_meet_stalactic_iterative_insertion_algorithm}

Algorithm \hyperref[alg:RightStalacticRightInsertion]{\textsf{rStRI}} allows one to insert a letter into the top row of a stalactic tableau:

\begin{algorithm}[H] \label{alg:RightStalacticRightInsertion}
        \DontPrintSemicolon
        \KwIn{A stalactic tableau $T$, a letter $a \in \N$.}
        \KwOut{A stalactic tableau $T \leftarrow a$.}
        \BlankLine
        \eIf{$T$ has a column whose cells are labelled by $a$}{remove the column labelled by $a$ from $T$, add a new $a$-labelled cell to the top of the column, and reattach the column to the right of the tableau;}{attach an $a$-labelled cell to the right of $T$;}
        \Return the resulting tableau $T \leftarrow a$.
        \caption{\textit{Right-Stalactic Right-Insertion} (\textsf{rStRI}).}
\end{algorithm}
We can define a \textit{Left-Stalactic Left-Insertion} (\hyperref[alg:RightStalacticRightInsertion]{\textsf{lStLI}}) algorithm in an analogous way. 

As before in Subsubsection~\ref{subsubsection:composition_diagrams_stalactic_tableaux_and_patience_sorting_tableaux}, using Algorithm \hyperref[alg:RightStalacticRightInsertion]{\textsf{rStRI}} (resp. \hyperref[alg:RightStalacticRightInsertion]{\textsf{lStLI}}), one can compute a unique stalactic tableau from a word $\word \in \N^*$: Starting from the empty tableau, read $\word$ from left-to-right (resp. right-to-left) and insert its letters one-by-one into the tableau. The resulting tableau is denoted by $\psymb^\rightarrow_\rst(\word)$ (resp. $\psymb^\leftarrow_\lst(\word)$).

\begin{lemma} \label{lemma:alternate_insertion_algorithm_stalactic}
    Let $\word \in \N^*$. Then, $\prst{\word} = \psymb^\rightarrow_\rst(\word)$ and $\plst{\word} = \psymb^\leftarrow_\lst(\word)$.
\end{lemma}

\begin{proof}
    We prove the first equality, as the proof for the second equality is analogous. The proof is done by induction on the length of $\word$. If $\word$ is the empty word, then the result is trivially satisfied. 
    
    Suppose $\word = \uord a$, for some $a \in \N$ and $\uord \in \N^*$. Recall that $\prst{\word}$ is obtained by reading $\word$ from right-to-left and inserting its letters using Algorithm \hyperref[alg:RightStalacticLeftInsertion]{\textsf{rStLI}}, and as such, its rightmost column is labelled by $a$ and of height $|\word|_a = |\prst{\uord}|_a + 1$, and the remaining columns form the tableau $(\prst{\uord})_{\neq a}$. By induction, we have that $\prst{\uord} = \psymb^\rightarrow_\rst(\uord)$. As such, $\prst{\word}$ is obtained in the same way as $\psymb^\rightarrow_\rst(\word)$.
\end{proof}

As such, we can compute the $\psymb$-symbol of a meet-stalactic class by reading a word in only one direction and applying Algorithms \hyperref[alg:RightStalacticLeftInsertion]{\textsf{rStLI}} and \hyperref[alg:RightStalacticRightInsertion]{\textsf{rStRI}} (or Algorithms \hyperref[alg:RightStalacticLeftInsertion]{\textsf{lStRI}} and \hyperref[alg:RightStalacticRightInsertion]{\textsf{lStLI}}) at the same time:

\begin{corollary} \label{corollary:mst_root_insertion}
    For any $\word \in \N^*$, we have that 
    \[
        \pmst{\word} = (\psymb^\leftarrow_\lst(\word), \prst{\word}) = (\plst{\word}, \psymb^\rightarrow_\rst(\word)).
    \]
\end{corollary}

We define the \emph{left-to-right iterative meet-stalactic $\psymb$-symbol} of a word $\word \in \N^*$ as the pair $\pmst{\word} = (\plst{\word}, \psymb^\rightarrow_\rst(\word))$, obtained by reading $\word$ from left-to-right and iteratively inserting its letters into $(\perp,\perp)$, using Algorithms \hyperref[alg:RightStalacticLeftInsertion]{\textsf{lStRI}} and \hyperref[alg:RightStalacticRightInsertion]{\textsf{lStLI}} for $\plst{\word}$ and $\psymb^\rightarrow_\rst(\word)$, respectively. The \emph{left-to-right iterative meet-stalactic $\qsymb$-symbol} of $\word$ is the pair $\qmst{\word} = (\qlst{\word}, \qsymb^\rightarrow_\rst(\word))$ of the same shape as $\pmst{\word}$, where each cell is labelled by the position in $\word$ of the letter in its corresponding cell in $\pmst{\word}$. In other words, each cell in $\qmst{\word}$ is labelled by the step in which its corresponding cell in $\pmst{\word}$ was created, when applying the left-to-right iterative insertion algorithm.

The correctness of the iterative insertion algorithm follows from Corollary~\ref{corollary:mst_root_insertion}. Thus, we obtain an insertion algorithm in line with the usual Robinson--Schensted correspondence algorithms. Furthermore, notice that we can also define a right-to-left version of the iterative insertion algorithm. From these, we also obtain iterative versions of insertion algorithms for the left and right-stalactic $\psymb$ and $\qsymb$-symbols.

\subsubsection{The meet-taiga iterative insertion algorithm} \label{subsubsection:The_meet_taiga_iterative_insertion_algorithm}

Algorithm \hyperref[alg:TaigaRootInsertion]{\textsf{TgRI}} allows one to obtain a new BSTM from another one, with the root node labelled by the letter of our choice:

\begin{algorithm}[ht] \label{alg:TaigaRootInsertion}
        \DontPrintSemicolon
        \KwIn{A BSTM $T$, a letter $a \in \N$.}
        \KwOut{A BSTM $a \downarrow T$.}
        \BlankLine
        let $T_{<a}$ (resp. $T_{>a}$) be the tree of all nodes of $T$ with letter $<a$ (resp. $>a$), such that a node $x$ is a descendent of a node $y$ in $T_{<a}$ (resp. $T_{>a}$ only if $x$ is a descendant of $y$ in $T$; \\
        let $a \downarrow T$ be the tree with root node labelled $a$ with multiplicity $|T|_a + 1$, with left subtree $T_{<a}$ and right subtree $T_{>a}$; \\
        \Return the resulting tree $a \downarrow T$.
        \caption{\textit{Taiga Root Insertion} (\textsf{TgRI}).}
\end{algorithm}

As before in Subsubsection~\ref{subsubsection:binary_trees_with_multiplicities_and_binary_search_trees_with_multiplicities}, using Algorithm \hyperref[alg:TaigaRootInsertion]{\textsf{TgRI}}, one can compute a unique BSTM from a word $\word \in \N^*$: Starting from the empty tree, read $\word$ from left-to-right (resp. right-to-left) and insert its letters one-by-one into the tree. The resulting BSTM is denoted by $\psymb^\rightarrow_\rtg(\word)$ (resp. $\psymb^\leftarrow_\ltg(\word)$). The proof of \cite[Lemma~4.18]{giraudo_baxter} can be easily adapted to show the following:

\begin{lemma} \label{lemma:alternate_insertion_algorithm_taiga}
    Let $\word \in \N^*$. Then $\psymb^\rightarrow_\rtg(\word) = \prtg{\word}$ and $\psymb^\leftarrow_\ltg(\word) = \pltg{\word}$.
\end{lemma}

As such, we can compute the $\psymb$-symbol of a meet-taiga class by reading a word in only one direction and applying Algorithms \hyperref[alg:TaigaLeafInsertion]{\textsf{TgLI}} and \hyperref[alg:TaigaRootInsertion]{\textsf{TgRI}} at the same time:

\begin{corollary} \label{corollary:mtg_root_insertion}
    For any $\word \in \N^*$, we have that 
    \[
        \pmtg{\word} = (\psymb^\leftarrow_\ltg(\word), \prtg{\word}) = (\pltg{\word}, \psymb^\rightarrow_\rtg(\word)).
    \]
\end{corollary}

We define the \emph{left-to-right iterative meet-taiga $\psymb$-symbol} of a word $\word \in \N^*$ as the pair $\pmtg{\word} = (\pltg{\word}, \psymb^\rightarrow_\rtg(\word))$, obtained by reading $\word$ from left-to-right and iteratively inserting its letters into $(\perp,\perp)$, using Algorithms \hyperref[alg:TaigaLeafInsertion]{\textsf{TgLI}} and \hyperref[alg:TaigaRootInsertion]{\textsf{TgRI}} for $\pltg{\word}$ and $\psymb^\rightarrow_\rtg(\word)$, respectively. The \emph{left-to-right iterative meet-taiga $\qsymb$-symbol} of $\word$ is the pair $\qmtg{\word} = (\qltg{\word}, \qsymb^\rightarrow_\rtg(\word))$ of the same shape as $\pmtg{\word}$, where each node is labelled by the positions in $\word$ of the letter in its corresponding node in $\pmtg{\word}$. In other words, each node in $\qmtg{\word}$ is labelled by the steps in which its corresponding node in $\pmtg{\word}$ was created or had its multiplicity incremented, when applying the left-to-right iterative insertion algorithm.

The correctness of the iterative insertion algorithm follows from Corollary~\ref{corollary:mtg_root_insertion}. Thus, as in the previous Subsubsection, we obtain an insertion algorithm in line with the usual Robinson--Schensted correspondence algorithms. We can also define a right-to-left version of the iterative insertion algorithm. From these, we also obtain iterative versions of insertion algorithms for the left and right-taiga $\psymb$ and $\qsymb$-symbols.


\section{Counting in the stalactic and taiga monoids} \label{section:Counting_in_the_stalactic_and_taiga_monoids}

Throughout this section, for any permutation $\sigma$ of $[n]$, we denote by $\sigma_i$ the image of $i$ by $\sigma$, to simplify the notation.


\subsection{Sizes of classes of words under stalactic and taiga congruences} \label{subsection:Size_of_equivalence_classes_in_stalactic_and_taiga_monoids}

We now obtain formulas for the sizes of classes of words under the stalactic and taiga congruences. The meet-stalactic and meet-taiga cases are treated separately from the other cases, due to the complexity of the arguments we use.


\subsubsection{The left, right, and join cases} \label{subsubsection:The_left_right_and_join_cases}

We first give non-recursive formulas for the left, right and join cases.

\begin{proposition}
Let $\word \in \N^*$. Let $\supp{\word} = \{x_1,\dots,x_k\}$ be such that $\overrightarrow{\rho_{\word}}(x_i) < \overrightarrow{\rho_{\word}}(x_{i+1})$ for all $1 \leq i < k$. Then, there are
\[ \frac{|\word|!}{\prod_{i=1}^k \left((|\word|_{x_i}-1)! \sum_{j=i}^k |\word|_{x_j}\right)}\]
words over $\N$ in the $\equiv_\lst$-class of $\word$.
\end{proposition}
\begin{proof}
    By the left-stalactic Robinson--Schensted correspondence (Theorem~\ref{theorem:robinson_left_stalactic}), the size of the $\equiv_\lst$-class of $\word$ is exactly the number of increasing patience-sorting tableaux of the same shape as $\plst{\word}$ with entries in $[|\word|]$. 
    
    We can obtain the unique column reading of an increasing patience-sorting tableau by choosing $W_1,\dots,W_k \subseteq [|\word|]$ such that $|W_i| = |\word|_{x_i}$ and, for each $1 \leq i < j \leq k$, we have $\min(W_i) < \min(W_j)$ and $W_i \cap W_j = \emptyset$. Hence, we can see that $\min(W_i)$ is the least letter not contained in $\bigcup_{j=1}^{i-1} W_j$ but the remaining letters in $W_i$ can be any letters from $[|\word|]$ not in $\bigcup_{j=1}^{i-1} W_j$. Thus, there are 
    \[ \prod_{i=1}^{k}\begin{pmatrix}
        |\word| - 1 - \sum_{j=1}^{i-1} |\word|_{x_{j}} \\
        |\word|_{x_i}-1
    \end{pmatrix}\]
    possible column readings of increasing patience-sorting tableaux, and therefore, words over $\N$ in the same $\equiv_\lst$-class of $\word$. Furthermore, by expanding the product, we can see that 
    \[ \prod_{i=1}^{k}\begin{pmatrix}
        |\word| - 1 - \sum_{j=1}^{i-1} |\word|_{x_j} \\
        |\word|_{x_i}-1
    \end{pmatrix} = \frac{(|\word|-1)!}{\left(\prod_{i=1}^k(|\word|_{x_i}-1)!\right)\prod_{t=2}^k(|\word| - \sum_{j=1}^{t-1}|\word|_{x_j})}.\]
    Finally, by noting that $|\word| - \sum_{j=1}^{t-1}|\word|_{x_j} = \sum_{j=t}^k|\word|_{x_j}$ and that $|\word| = \sum_{j=1}^k|\word|_{x_j}$, we have the result.
\end{proof}

\begin{proposition}
Let $\word \in \N^*$. Let $\supp{\word} = \{x_1,\dots,x_k\}$ and suppose $\word$ has $m$ simple variables. Then, there are
\[ \frac{|\word|!}{m! \cdot \prod_{i=1}^k|\word|_{x_i}!}\]
words over $\N$ in the same $\equiv_\jst$-class of $\word$.    
\end{proposition}
\begin{proof}
    By Proposition~\ref{prop:jst_characterisation}, $\uord \equiv_{\jst} \vord$ if and only if $\cont{\uord} = \cont{\vord}$ and $\ol{\uord} = \ol{\vord}$. Thus, we aim to count the number of rearrangements of $\word$ which preserve the order of the $m$ simple variables, that is, rearrangements which are equal to $\ol{\word}$ when restricted to the simple letters. There are 
    \[\frac{|\word|!}{\prod_{i=1}^k|\word|_{x_i}!}\]
    unique rearrangements of $\word$ and $m!$ rearrangements of $\ol{\word}$. Thus, there are
    \[ \frac{|\word|!}{m! \cdot \prod_{i=1}^k|\word|_{x_i}!}\]
    words over $\N$ in the same $\equiv_\jst$-class of $\word$.
\end{proof}

We prove by different means the following known proposition:
\begin{proposition}[Proposition 5, \cite{priez_binary_trees}]
Let $\word \in \N^*$. Let $\supp{\word} = \{x_1 < \dots < x_k\}$ and for all $1 \leq i \leq k$, let $m_i$ be the sum of the multiplicities of the node labelled $x_i$ and its descendants in $\pltg{\word}$. Then, there are
\[ \frac{|\word|!}{\prod_{i=1}^k(|\word|_{x_i}-1)! \cdot m_i}\]
words over $\N$ in the same $\equiv_\ltg$-class of $\word$.
\end{proposition}
\begin{proof}
    We prove this by induction on the size of the $\supp{\word}$. If $|\supp{\word}| = 1$, then there is a unique word in the $\equiv_\ltg$-class of $\word$, so the formula holds. For induction, suppose the formula holds when $|\supp{\word}| < k$. Now, let $|\supp{\word}| = k$ and $\word = x_l\word'$ for some $l \in [k]$ and $\word' \in \supp{\word}^*$. We define the sets $X = \{ x_1\dots,x_{l-1}\}$ and $Y = \{ x_{l+1},\dots x_k\}$. Now, we may apply the inductive hypothesis to $\word[X]$ and $\word[Y]$, from which we obtain that there are
    \[ \frac{|\word[X]|!}{\prod_{i=1}^{l-1}(|\word|_{x_i}-1)!m_i} \quad \text{and} \quad \frac{|\word[Y]|!}{\prod_{i=l+1}^k(|\word|_{x_i}-1)!m_i}\]
    words over $\N$ in the same $\equiv_\ltg$-class as $\word[X]$ and $\word[Y]$ respectively.
    We aim to show that there are
    \[ \begin{pmatrix}
        |\word|-1 \\
        |\word|_{x_l}-1
    \end{pmatrix}
        \begin{pmatrix}
        |\word[X]| + |\word[Y]| \\
        |\word[X]|
    \end{pmatrix}
    \frac{\word[X]|!|\word[Y]|!}{\prod_{i\neq l}(|\word|_{x_i}-1)!m_i}\]
    words over $\N$ in the same $\equiv_\ltg$-class as $\word$. 
    
    The first factor gives the number of choices for the positions of $x_l$ in a word in $[\word]_\ltg$ as, by the characterisation of equality in $\ltg$, its first letter has to be $x_l$, but there are no restrictions on the remaining occurrences of $x_l$ in the word.  
    The second factor then gives the number of choices for the positions of letters $x_1, \dots, x_{l-1}$ in the word, without distinction between the letters, after choosing the positions of $x_l$. Notice that these choices determine the positions of letters $x_{l+1}, \dots, x_k$ in the word, without distinction between the letters. The final factor gives us the number of possible rearrangements of $\word[X]$ and $\word[Y]$, as given by the induction hypothesis. 
    
    Remark that a word chosen in this way is equal to $\word$ in $\ltg$ since any letter from $X$ can commute with any letter from $Y$ as the word begins with $x_l$. Thus, the above formula counts the number of words of $\N$ in the same $\equiv_\ltg$-class as $\word$. Finally, by noting that $|\word| = m_l$ and that
    \[\begin{pmatrix}
        |\word[X]| + |\word[Y]| \\
        |\word[X]|
    \end{pmatrix}|\word[X]|!|\word[Y]|! = (|\word|-|\word|_{x_l})!,\]
    the result follows.
\end{proof}

The following proposition shows that we are able to obtain the size of a $\equiv_\jtg$-class, by computing the size of multiple $\equiv_\hypo$-classes. By a result of Cain and Malheiro {\cite[Theorem~3]{cm_hypo_crystal}}, these are known to be efficiently computable.

\begin{proposition}
Let $\word \in \N^*$. Let $\supp{\word} = \{x_1 < \dots < x_k\}$ and $A_1,\dots A_l$ be all the intervals of $\supp{\word}$ such that $A_j$ only contains simple letters and $A_j \cup \{a\}$ is not an interval, for any simple letter $a \notin A_j$. Let $h(\word[A_j])$ be the size of the hypoplactic class of $\word[A_j]$.
Then, there are
\[ \frac{|\word|!}{\prod_{i=1}^{k}|\word|_{x_i}!} \prod_{j=1}^{l} \frac{h(\word[A_j])}{|A_j|!}\]
words over $\N$ in the same $\equiv_\jtg$-class of $\word$.
\end{proposition}
\begin{proof}
    Let $\vord \in [\word]_{\jtg}$. By Proposition~\ref{prop:jtg_characterisation}, $\word \equiv_{\jtg} \vord$ if and only if $\cont{\word} = \cont{\vord}$ and $\word[A_j] \equiv_\hypo \vord[A_j]$ for all $1 \leq j \leq k$.
    We now aim to show that there are 
    \[ \left(\prod_{j=1}^{l}\begin{pmatrix}
        |\word| - \sum_{r=1}^{j-1} |A_{r}| \\
        |A_{j}|
    \end{pmatrix}h(\word[A_j])\right)\frac{(|\word| - \sum_{s=1}^l|A_s|)!}{\prod_{i=1}^{k}|\word|_{x_i}!}\]
    words over $\N$ in the same $\equiv_\jtg$-class as $\word$. 
    
    The first factor gives the product of the number of possible positions of the letters from $A_i$ in any rearrangement of $\word$, without distinction between the letters, by the number of rearrangements of $\word[A_j]$ which are $\equiv_\hypo$-congruent to $\word[A_j]$. The second factor is the number of positions of the remaining letters in any rearrangement of $\word$, as there is no restriction on their positions. Note that if $x \in A_j$ for some $j$, then $|\word|_{x} = 1$ and therefore $|\word|_{x}! = 1$. As such, we do not need to exclude $|\word|_{x}!$ from the product $\prod_{i=1}^{k}|\word|_{x_i}!$ in the final factor. 
    
    Finally, by noting that
     \[ \prod_{j=1}^{l}\begin{pmatrix}
        |\word| - \sum_{r=1}^{j-1} |A_r| \\
        |A_{j}|
    \end{pmatrix} = \frac{|\word|!}{(|\word| - \sum_{s=1}^l|A_s|)! \cdot \prod_{j=1}^l|A_j|!},\]
    the result follows.
\end{proof}


\subsubsection{The meet cases} \label{subsubsection:The_meet_cases}

We now give recursive formulas for the meet cases. In order to compute the size of $\equiv_\mst$ and $\equiv_\mtg$-classes, it is easier to visualize a word of said classes as a linear extension of a partially ordered set (poset for short). Let $P = (I,\leq_P)$ be a poset. Then, a linearly ordered set $Q = (I, \leq_Q)$ is a \emph{linear extension} of $P$ if $i \leq_P j$ implies $i \leq_Q j$, for all $i,j \in I$. 

In the following, we define a class of posets: for $k \in \N$, $X = (X_1,\dots, X_k) \in \N^k$, and a permutation $\tau$ of $[k]$, let $P[X;\tau]$ be the poset built by taking $k$ chains, where the $i$-th chain has $X_i$ elements, and requiring the least (resp. greatest) element of the $i$-th (resp. $\tau_i$-th) chain to be less than the least (resp. greatest) element of the $(i+1)$-th (resp. $\tau_{i+1}$-th) chain. We represent these posets by graphs and refer to their elements as `nodes'. The order relations are represented by arrows, in the sense that an arrow starts in one node and ends in another if the element corresponding to the former node is less than the element corresponding to the latter node.

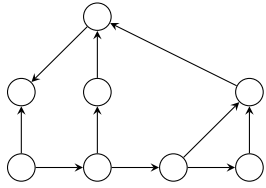
\begin{figure}[h]
    \centering
    \begin{tikzpicture}
        \node[circle,draw] (A) at (0,0){};
        \node[circle,draw] (B) at (0,1){};
        \node[circle,draw] (C) at (1,0){};
        \node[circle,draw] (D) at (1,1){};
        \node[circle,draw] (E) at (1,2){};
        \node[circle,draw] (F) at (2,0){};
        \node[circle,draw] (G) at (3,0){};
        \node[circle,draw] (H) at (3,1){};
        \draw[-stealth] (A) to (B);
        \draw[-stealth] (A) to (C);
        \draw[-stealth] (C) to (D);
        \draw[-stealth] (C) to (F);
        \draw[-stealth] (D) to (E);
        \draw[-stealth] (E) to (B);
        \draw[-stealth] (F) to (G);
        \draw[-stealth] (F) to (H);
        \draw[-stealth] (G) to (H);
        \draw[-stealth] (H) to (E);
        \end{tikzpicture}
    \caption{The poset $P[(2,3,1,2);3421]$ has 4 different linear extensions.}
    \label{fig:example_poset_linear_extensions}
\end{figure}

\begin{proposition} \label{prop:mst_poset_bijection}
    Let $\word \in \N^*$. Then, each word in $[\word]_\mst$ is in bijection with a linear extension of $P[X;\tau]$ where $k = |\supp{\word}|$, $X_i = |\word|_{{\overrightarrow{\rho_\word}}^{-1}(i)}$ and $\tau_i = {\overrightarrow{\rho_\word}}({\overleftarrow{\rho_\word}}^{-1}(i))$, for $1 \leq i \leq k$.
\end{proposition}
\begin{proof}
    Let $\uord \in [\word]_\mst$. From $\uord$ we define a chain $P'$, with $|\uord|$ element, by letting the $i$-th least element in $P'$ be the $j$-th node in the $l$-th chain of $P[X;\tau]$ if the $i$-th letter of $\uord$ is the $j$-th appearance of ${\overrightarrow{\rho_\uord}}^{-1}(l)$.
    
    We want to show that $P'$ is a linear extension of $P[X;\tau]$. First, note that every node of $P[X;\tau]$ appears in $P'$ as $\overrightarrow{\rho_\word} = \overrightarrow{\rho_\uord}$ and $\cont{\uord} = \cont{\word}$, therefore $X_l = |\word|_{{\overrightarrow{\rho_\word}}^{-1}(l)} = |\uord|_{{\overrightarrow{\rho_\uord}}^{-1}(l)}$. 
    
    Next, we show that if a node is less than another node in $P[X;\tau]$ then the same happens in $P'$. 
    By the definition of $P'$, it is clear that, in $P'$, the $j$-th node in the $l$-th chain of $P[X;\tau]$ is less than the $(j+1)$-th node in the $l$-th chain, for all $1 \leq l \leq k$ and $1 \leq j < X_l$. 
    
    For each $1 \leq l < k$, in $P'$, the least element in the $l$-th chain of $P[X;\tau]$ is less than the least element in the $(l+1)$-th chain as 
    \[l = {\overrightarrow{\rho_\uord}}({\overrightarrow{\rho_\uord}}^{-1}(l)) < {\overrightarrow{\rho_\uord}}({\overrightarrow{\rho_\uord}}^{-1}(l+1)) = l+1,\] and hence the first occurrence of ${\overrightarrow{\rho_\uord}}^{-1}(l)$ appears to the left of the first occurrence of ${\overrightarrow{\rho_\uord}}^{-1}(l+1)$ in $\uord$.
    
    By its definition, in $P'$, the greatest element of the $\tau_l$-th chain of $P[X;\tau]$ is less than the greatest element of the $\tau_{l+1}$-th chain if the last occurrence of ${\overrightarrow{\rho_\uord}}^{-1}(\tau_l)$ appears to the left of the last occurrence of ${\overrightarrow{\rho_\uord}}^{-1}(\tau_{l+1})$ in $\uord$, that is, if ${\overleftarrow{\rho_\uord}}{\overrightarrow{\rho_\uord}}^{-1}(\tau_l) < {\overleftarrow{\rho_\uord}}{\overrightarrow{\rho_\uord}}^{-1}(\tau_{l+1})$.
    
    As $\overrightarrow{\rho_\word} = \overrightarrow{\rho_\uord}$, $\overleftarrow{\rho_\word} = \overleftarrow{\rho_\uord}$ and $\tau_l = {\overrightarrow{\rho_\word}}({\overleftarrow{\rho_\word}}^{-1}(l)$, we have that ${\overleftarrow{\rho_\uord}}({\overrightarrow{\rho_\uord}}^{-1}(\tau_l)) = l$ for each $1 \leq l \leq k$. 
    Therefore, ${\overleftarrow{\rho_\uord}}{\overrightarrow{\rho_\uord}}^{-1}(\tau_l) < {\overleftarrow{\rho_\uord}}{\overrightarrow{\rho_\uord}}^{-1}(\tau_{l+1})$. 
    
    Now, let $P'$ be a linear extension of $P[X;\tau]$ and define $\uord \in \N^*$ by letting the $i$-th letter of $\uord$ be ${\overrightarrow{\rho_\word}}^{-1}(l)$ if the $i$-th least element of $P'$ is in the $l$-th chain of $P[X;\tau]$. By the definition of $\uord$ and $P[X;\tau]$, $\cont{\uord} = \cont{\word}$ and, by considering the least elements of each chain in $P[X;\tau]$, we have that $\overrightarrow{\rho_\word} = \overrightarrow{\rho_\uord}$.
    To see that $\overleftarrow{\rho_\word} = \overleftarrow{\rho_\uord}$, note that, for any $a,b \in \supp{\uord}$, we have that $\overleftarrow{\rho_\uord}(a) < \overleftarrow{\rho_\uord}(b)$ if and only if the greatest node in the $\overrightarrow{\rho_\word}(a)$-th chain of $P[X,\tau]$ is less than the greatest node in the $\overrightarrow{\rho_\word}(b)$-th chain. This happens if and only if there exist $1 \leq s < t \leq k$ such that $\overrightarrow{\rho_\word}(a) = \tau_s$ and $\overrightarrow{\rho_\word}(b) = \tau_t$, by the definition of $P[X,\tau]$. Furthermore, as $\tau_s = {\overrightarrow{\rho_\word}}({\overleftarrow{\rho_\word}}^{-1}(s))$, we have $a = {\overleftarrow{\rho_\word}}^{-1}(s)$ and, similarly, $b = {\overleftarrow{\rho_\word}}^{-1}(t)$. Therefore,
    \[
        \overleftarrow{\rho_\word}(a) = s < t = \overleftarrow{\rho_\word}(b),
    \]
    and hence $\uord \in [\word]_\mst$.
\end{proof}

To improve the readability of the following theorem, we introduce the following function: let $k,n,l \in \N$, such that $k \geq 2$, and $1 \leq n < l \leq k$. Let $M = (m_1,\dots,m_{l-n}) \in \N_0^{l-n}$ and $\lVert M \rVert_1 := \sum_{i=1}^{l-n} m_i$. We define the function $ f_{k,n,l,M} \colon (\N_0^k,\N_0^{k-1}) \to (\N_0^{k-1},\N_0^{k-2})$, where for $V = (v_1,\dots,v_k) \in \N_0^k$, $B = (b_1,\dots, b_{k-1}) \in \N_0^{k-1}$, the first coordinate of $f_{k,n,l,M}(V,B)$ is given by 
\[ (v_1,\dots,v_{n-1},v_{n+1},\dots,v_{l-1},v_{l}+v_{n}-\lVert M \rVert_1,v_{l+1},\dots,v_k),\]
and the second coordinate is given by 
\[ (b_2 + m_2,\dots,b_{l-1} + m_{l-1},b_{l},\dots,b_k) \]
when $n = 1$, and by 
\[ (b_1, \dots, b_{n-2}, b_{n-1} + 1 + b_n + m_1, b_{n+1} + m_2,\dots,b_{l-1} + m_{l-n}, b_{l},\dots,b_k)\]
when $n > 1$.

Let $V \in \N_0^k$, $B \in \N_0^{k-1}$, and $\sigma$ be a permutation of $[k]$ such that if $\sigma_i > \sigma_j$, then $v_{\sigma_j} \neq 0$, for any $1 \leq i < j \leq k$. Define $G[V;B;\sigma]$ to be the poset built by taking $k$ chains, where each chain has length $v_i + 1$, and requiring the least element in the $i$-th chain to be less than the least element in the $(i+1)$-th chain, and to have a chain of $b_i$ nodes between them. Moreover, we require the greatest element in the $\sigma_i$-th chain to be less than the greatest element in the $\sigma_{i+1}$-th chain.

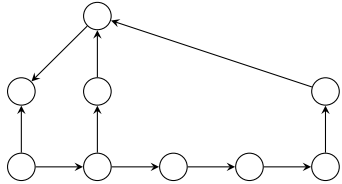
\begin{figure}[ht]
    \centering
    \begin{tikzpicture}
        \node[circle,draw] (A) at (0,0){};
        \node[circle,draw] (B) at (0,1){};
        \node[circle,draw] (C) at (1,0){};
        \node[circle,draw] (D) at (1,1){};
        \node[circle,draw] (E) at (1,2){};
        \node[circle,draw] (F) at (2,0){};
        \node[circle,draw] (G) at (3,0){};
        \node[circle,draw] (H) at (4,0){};
        \node[circle,draw] (I) at (4,1){};
        \draw[-stealth] (A) to (B);
        \draw[-stealth] (A) to (C);
        \draw[-stealth] (C) to (D);
        \draw[-stealth] (C) to (F);
        \draw[-stealth] (D) to (E);
        \draw[-stealth] (E) to (B);
        \draw[-stealth] (F) to (G);
        \draw[-stealth] (G) to (H);
        \draw[-stealth] (I) to (E);
        \draw[-stealth] (H) to (I);
        \end{tikzpicture}
    \caption{The poset $G[(1,2,1);(0,2);321]$.}
    \label{fig:example_poset_linear_extensions3}
\end{figure}

The following theorem follows a similar reasoning to the one given by Pan in \cite{ranpan}. In the following, we define $\cL(P)$ to be the number of linear extensions of the poset $P$ and, for ease of notation, $\begin{pmatrix}
    -1 \\
    0
\end{pmatrix} := 1$. By inserting a node in an edge, we mean that we remove said edge and add two extra edges, one starting in the initial node and ending in the inserted node, and the other starting in the inserted node and ending in the terminal node.

\begin{theorem} \label{theorem:recursive_formula_linear_extensions_mst_poset}
Let $k \geq 2$, $V = (v_1, \dots, v_k) \in \N_0^k$, $B = (b_1, \dots, b_{k-1}) \in \N_0^{k-1}$ and $\sigma$ be a permutation of $[k]$. Then, when $\sigma_1 < \sigma_2$, $\cL(G[V;B;\sigma])$ is equal to
\begin{equation*}
    \sum_{\substack{M\in \N_0^{\sigma_2-\sigma_1} \\ 0 \leq \lVert M \rVert_1 \leq v_{\sigma_1}}} \cL_M \cdot
    \begin{pmatrix}
         v_{\sigma_2} - 1 + v_{\sigma_1} - \lVert M \rVert_1 \\ v_{\sigma_1} - \lVert M \rVert_1
    \end{pmatrix} 
    \prod_{i=1}^{\sigma_2-\sigma_1}
    \begin{pmatrix}
        b_{\sigma_1 - 1 + i} + m_i \\
        m_i
    \end{pmatrix}, 
\end{equation*}
where $\cL_M = 1$ if $k=2$ and $\cL_M = \cL(G[f_{k,\sigma_1,\sigma_2,M}(V,B);\std{\sigma_2\sigma_3\cdots \sigma_k}])$ otherwise, and, when $\sigma_2 < \sigma_1$, is equal to
\begin{equation*}
        \sum_{\substack{M \in \N_0^{\sigma_1-\sigma_2} \\ 0 \leq \lVert M \rVert_1 \leq v_{\sigma_2}-1}} \cL'_M \cdot
        \begin{pmatrix}
            v_{\sigma_1} + v_{\sigma_2} - 1 - \lVert M \rVert_1 \\ 
            v_{\sigma_2} - 1 - \lVert M \rVert_1
        \end{pmatrix}  
            \prod_{i=1}^{\sigma_1-\sigma_2}
        \begin{pmatrix}
            b_{\sigma_2 - 1 + i} + m_i \\
            m_i
        \end{pmatrix}, 
\end{equation*}
where $\cL_M' = 1$ if $k = 2$ and $\cL'_M = \cL(G[f_{k,\sigma_2,\sigma_1,M}(V,B);\std{\sigma_1\sigma_3\cdots \sigma_k}])$ otherwise.
\end{theorem}

\begin{proof}
    In this proof, we show that we are able to take any poset of the form $G[V;B;\sigma]$ and create a new set of (sometimes smaller) posets in which the sum of the number of linear extensions of the new posets is equal to the number of linear extensions of the original poset.

    We begin with the $\sigma_1 < \sigma_2$ case. To create a new poset, we take the nodes in the $\sigma_1$-th chain, excluding the least node, and insert each of these $v_{\sigma_1}$ nodes into either an edge between the least node of the $\sigma_1$-th chain and the least node of the $\sigma_2$-th chain, or an edge in the $\sigma_2$-th chain. In the case where $v_{\sigma_1} = 0$, we just remove the edge from the only node in the $\sigma_1$-th chain to the greatest node of the $\sigma_2$-th chain.
    
    Let $M = (m_1,\dots, m_{\sigma_2 - \sigma_1})$ be such that $0 \leq \lVert M \rVert_1 \leq v_{\sigma_1}$ where $m_i$ is the number of nodes inserted into the edges between the least nodes of the $(\sigma_1 + i - 1)$-th chain and the $(\sigma_1 + i)$-th chain.  Then, $\lVert M \rVert_1$ nodes are inserted into the bottom edges and $v_{\sigma_1}-\lVert M \rVert_1$ nodes are inserted into the $\sigma_2$-th chain. By the definition of $f_{k,\sigma_1,\sigma_2,M}(V,B)$, the poset obtained from this process is given by
    \[
        G[f_{k,\sigma_1,\sigma_2,M}(V,B);\std{\sigma_2\sigma_3\cdots \sigma_k}],
    \]
    with the exception of the case where $\sigma_1 = 1$, where we obtain a poset of this form with a chain of least nodes before the first chain. As these nodes are the least in the poset and already linearly ordered, we can remove them without changing the number of linear extensions.
    As such, in all cases, we remove the $\sigma_1$-th chain and relabel the remaining chains to start at 1 and end at $k-1$. 
     
    For each $M$, there are 
    $\left(\begin{smallmatrix}
        v_{\sigma_2} -1 + v_{\sigma_1} - \lVert M \rVert_1 \\ v_{\sigma_1} - \lVert M \rVert_1
    \end{smallmatrix}\right)$
    ways to insert the $v_{\sigma_1} -\lVert M \rVert_1$ nodes into the $\sigma_2$-th chain and, for each $1 \leq i \leq \sigma_2 -\sigma_1$, there are 
    $\left(\begin{smallmatrix}
        b_{\sigma_1 - 1 + i} + m_i \\
        m_i
    \end{smallmatrix}\right)$ ways to insert the $m_i$ nodes into the edges between the least nodes of the $(\sigma_1 - 1 + i)$-th chain and $(\sigma_1 + i)$-th chain. In the case where $v_{\sigma_1} = 0$, we only have one possible choice of $M$, so the number of linear extensions of the new poset is the same as that of $G[V;B;\sigma]$.
    
    For the $\sigma_1 > \sigma_2$ case, first notice that $v_{\sigma_2} > 0$ by the definition of $G[V;B;\sigma]$. To create a new poset, we instead take the nodes in the $\sigma_2$-th chain, excluding the least node, move the greatest node of the $\sigma_2$-th chain to the top of the $\sigma_1$-th chain, and then insert each of the remaining $v_{\sigma_2} - 1$ nodes into either an edge between the least nodes of the $\sigma_2$-th chain and the $\sigma_1$-th chain or an edge in the $\sigma_1$-th chain (which has now been extended by a node). 
    
    Let $M = (m_1,\dots, m_{\sigma_1 - \sigma_2})$ be such that $0 \leq \lVert M \rVert_1 \leq v_{\sigma_2}-1$ where $m_i$ is the number of nodes inserted into an edge between the least nodes of the $(\sigma_2 - 1 + i)$-th chain and the $(\sigma_2 + i)$-th chain.  Then, $\lVert M \rVert_1$ nodes are inserted into the bottom edges and $v_{\sigma_2} - 1 -\lVert M \rVert_1$ nodes are inserted into the $\sigma_1$-th chain. Notice that we do not count the greatest node of the $\sigma_2$-th chain moving to the top of the $\sigma_1$-th chain. By the definition of $f_{k,\sigma_2,\sigma_1,M}(V,B)$, the poset obtained from this process is given by
    \[G[f_{k,\sigma_2,\sigma_1,M}(V,B);\std{\sigma_1\sigma_3\cdots \sigma_k}]\]
    with the exception of the case where $\sigma_2 = 1$, whereas in the previous case, we obtain a poset of this form with a chain of least nodes before the first chain, which we can remove without changing the number of linear extensions.
    As such, in all cases, we remove the $\sigma_2$-th chain and relabel the remaining chains to start at 1 and end at $k-1$. 

    Similarly to above, for each $M$, there are 
    $\left(\begin{smallmatrix}
        v_{\sigma_1} + v_{\sigma_2} - 1 - \lVert M \rVert_1 \\ v_{\sigma_1}
    \end{smallmatrix}\right)$ 
    ways to insert the $v_{\sigma_2} -1 - \lVert M \rVert_1$ nodes into the $\sigma_1$-th chain and for $1 \leq i \leq \sigma_1 - \sigma_2$, there are $\left(\begin{smallmatrix}
        b_{\sigma_2 - 1 + i} + m_i \\
        m_i
    \end{smallmatrix}\right)$ 
    ways to insert the $m_i$ nodes into the edges between the least nodes of the $(\sigma_2 - 1 + i)$-th chain and $(\sigma_2 + i)$-th chain.

    The result follows from the fact that each choice of inserting nodes results in a different poset, uniquely characterised by the ordering of specific elements, namely the total ordering of the nodes in the $\sigma_1$-th and $\sigma_2$-th chains and the nodes between the least nodes of these two chains. As such, no double-counting occurs when summing the number of linear extensions of the posets obtained using this method.
\end{proof}

We are able to count the number of linear extensions of a poset of the above form by recursively applying the above theorem until each poset is a chain. The time complexity of such a process is as follows:

\begin{proposition} \label{prop:recursive_formula_complexity}
    The algorithm given in Theorem~\ref{theorem:recursive_formula_linear_extensions_mst_poset} to compute the number of linear extensions of a poset has time complexity $\mathcal{O}(n^{2k-2}k!)$ where $n = \lVert V \rVert_1 + \lVert B \rVert_1 + k$.
\end{proposition}

\begin{proof}
    We prove this result by induction on the size of $k \geq 2$. First of all, notice that each binomial coefficient in the formulas can be computed in $\mathcal{O}(n)$ operations. Furthermore, per each factor of the sum, we multiply at most $k$ binomial coefficients. Note that $\lVert M \rVert_1 \leq n - k$ in either formula, and $|\sigma_1 - \sigma_2| \leq k-1$. Therefore, there are at most $\left(\begin{smallmatrix}
        n-k+k-1 \\ n-k
    \end{smallmatrix}\right) = \left(\begin{smallmatrix}
        n-1 \\ n-k
    \end{smallmatrix}\right)$ choices for $M$, and thus the sums have $\mathcal{O}(n)$ factors, in the worst-case scenario.

    In the $k=2$ case, we have that $\mathcal{L}_M = \mathcal{L}'_M = 1$. As such, with the previously given information, we have that the algorithm performs
    \[
        \mathcal{O}(n) \cdot \mathcal{O}(2) \cdot \mathcal{O}(n) = \mathcal{O}(n^2)
    \]
    operations.

    As the induction hypothesis, assume that the formula holds up to the case $k=m-1$, for some $m > 2$. As such in the $k=m$ case, we have that $\mathcal{L}_M$ and $\mathcal{L}'_M$ can be computed by performing $\mathcal{O}(n^{2m-4}(m-1)!)$ operations. Thus, the algorithm performs
    \[
        \mathcal{O}(n) \cdot \mathcal{O}(n^{2m-4}(m-1)!) \cdot \mathcal{O}(m) \cdot \mathcal{O}(n) = \mathcal{O}(n^{2m-2} m!)
    \]
    operations. As such, the result follows.
\end{proof}

We can use Proposition~\ref{prop:mst_poset_bijection} and Theorem~\ref{theorem:recursive_formula_linear_extensions_mst_poset} to compute the size of $\equiv_\mst$-classes as for any $X = (X_1,\dots,X_k) \in \N^k$ and any permutation $\tau$ of $[k]$ we have that 
\[P[X;\tau] = G[(X_1-1,\dots,X_k-1);(0,\dots,0);\tau]. \]
The time complexity to perform such a computation is $\mathcal{O}(n^{2k-2}k!)$, by Proposition~\ref{prop:recursive_formula_complexity}, where $n$ is the length of the words in the $\equiv_\mst$-class.

We now look at the case of $\equiv_\mtg$-classes. First, for any binary tree $T$, we define $\Delta(T)$ (resp. $\nabla(T)$) to be the poset $(N, \leq)$, where $N$ is the set of nodes of $T$ and, for all $i,j \in N$,
$i \leq j$ if the $i$-th node is an ancestor (resp. descendant) of the $j$-th node in $T$.

For $k \in \N$, $X = (X_1,\dots, X_k) \in \N^k$, and $(T_L,T_R)$ a pair of twin BTMs, define $Q[X;T_L;T_R]$ to be the poset built by taking $k$ chains of nodes, where the $i$-th chain has $X_i$ elements, and requiring the least (resp. greatest) element of the $i$-th chain to be less than the least (resp. greatest) element of the $j$-th chain if $i \leq j$ in $\Delta(T_L)$ (resp. $\nabla(T_R)$). 

\begin{figure}[h]
    \centering
    \begin{tikzpicture}
        \node[circle,draw] (A) at (0,0){};
        \node[circle,draw] (B) at (0,1){};
        \node[circle,draw] (C) at (1,0){};
        \node[circle,draw] (D) at (1,1){};
        \node[circle,draw] (E) at (1,2){};
        \node[circle,draw] (F) at (2,0){};
        \node[circle,draw] (G) at (3,0){};
        \node[circle,draw] (H) at (3,1){};
        \draw[-stealth] (A) to (B);
        \draw[-stealth] (C) to (D);
        \draw[-stealth] (D) to (E);
        \draw[-stealth] (G) to (H);
        \draw[-stealth,bend right=45] (A) to (F);
        \draw[-stealth] (F) to (C);
        \draw[-stealth] (F) to (G);
        \draw[-stealth] (F) to (H);
        \draw[-stealth] (H) to (E);
        \draw[-stealth] (B) to (E);
        \end{tikzpicture}
    \caption{The poset $Q[(2,3,1,2);T_L;T_R]$, where $(T_L,T_R) = \pmtg{13242142}$, has 36 different linear extensions. 
    }
    \label{fig:example_poset_linear_extensions2}
\end{figure}
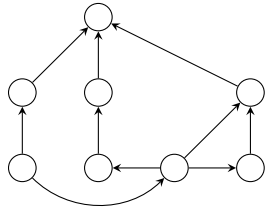

\begin{proposition} \label{prop:mtg_poset_bijection}
    Let $\word \in \N^*$. Then, each word in $[\word]_\mtg$ is in bijection with a linear extension of $Q[X;T_L;T_R]$ where $(T_L,T_R) = \pmtg{\word}$, $\supp{\word} = \{x_1 < \cdots < x_k\}$, and $X \in \N^k$ is such that $X_i = |\word|_{x_i}$.
\end{proposition}
\begin{proof}
    Let $\uord \in [\word]_\mtg$. From $\uord$ we define a chain $Q'$, with $|\uord|$ elements by letting the $i$-th least node of $Q'$ be the $j$-th node in the $l$-th chain of $Q[X;T_L;T_R]$ if the $i$-th letter of $\uord$ is the $j$-th appearance of $x_l$. 
    
    We want to show that $Q'$ is a linear extension of $Q[X;T_L;T_R]$. First, note that every node of $Q[X;T_L;T_R]$ appears in $Q'$ as $\cont{\uord} = \cont{\word}$ and $X_i = |\word|_{x_i}$. Next, we show that if a node is less than another in $Q[X;T_L;T_R]$ then the same happens in $Q'$. By the definition of $Q'$, it is clear that, in $Q'$, the $j$-th node in the $l$-th chain of $Q[X;T_L;T_R]$ is less than the $j+1$-th node in the $l$-th chain, for all $1 \leq l \leq k$ and $1 \leq j < X_l$. 
    
    Suppose that the least node in the $i$-th chain of $Q[X;T_L;T_R]$ is less than the least node in the $j$-th chain, that is, if $i < j$ in $\Delta(T_L)$. 
    We can see that this also holds in $Q'$ as $\pltg{\uord} = \pltg{\word} = T_L$. We can similarly show that if the greatest node in the $i$-th chain of $Q[X;T_L;T_R]$ is less than the greatest node in the $j$-th chain, the same happens in $Q'$. Therefore, $Q'$ is a linear extension of $Q[X;T_L;T_R]$.
    
    Now, let $Q'$ be a linear extension of $Q[X;T_L;T_R]$ and define $\uord \in \N^*$ by letting the $i$-th letter of $\uord$ be $x_l$ if the $i$-th least element of $Q'$ is in the $l$-th chain of $Q[X;T_L;T_R]$. Then, $\uord \in [\word]_\mtg$, as, by its definition and that of $Q[X;T_L;T_R]$, $\cont{\uord} = \cont{\word}$ and, from the least and greatest elements of each chain in $Q[X;T_L;T_R]$, we obtain the shape of $\pmtg{\uord}$, which coincides with that of $(T_L,T_R)$.
\end{proof}

Proposition~\ref{prop:mtg_poset_bijection} allows us to view the problem of computing the size of $\equiv_\mtg$-classes as a problem of counting linear extensions of posets. With the following theorem, we can count these linear extensions by splitting the problem into multiple $\equiv_\mst$ cases and counting linear extensions of these instead. 

We first require the following notation: Let $(T_L,T_R)$ be a pair of twin BSTMs. Then, for a poset $P$ with the same underlying set as $\nabla(T_R)$, let $\nabla(T_R,P)$ be the poset $\nabla(T_R)$ with the extra condition that if $i \leq j$ in $P$ then $i \leq j$ in $\nabla(T_R,P)$. Moreover, for any chain $P' = (I, \leq)$, let $\lambda_{P'} \colon I \to \{1,\dots,|I|\}$ be the function that sends each $x \in I$ to $m$ if $x$ is the $m$-th least element in $P'$.

\begin{theorem} \label{theorem:recursive_formula_linear_extensions_mtg_poset}
    Let $(T_L,T_R)$ be a pair of twin BSTMs with $k$ nodes and $X \in \N^k$ such that $X_i$ is the multiplicity of the $i$-th node of the trees. Then,
    \[ \cL(Q[X;T_L;T_R]) = \sum_{L} \sum_{R} \cL(P[Y_{L};\lambda_L \lambda_R^{-1}])\]
    where the first sum is over all linear extensions $L$ of $\Delta(T_L)$ and the second sum is over all linear extensions $R$ of $\nabla(T_R,p_L)$ where $a < x$ in $p_L$ if and only if $a$ is simple in $(T_L,T_R)$ and $a < x$ in $L$, and $(Y_{L})_i = X_{\lambda_L^{-1}(i)}$.
\end{theorem}
\begin{proof}
    Throughout this proof, let $Q = Q[X;T_L;T_R]$. 

    First, we show that given linear extensions of $L$ of $\Delta(T_L)$, $R$ of $\nabla(T_R,p_L)$ and $P'$ of $P[Y_{L};\lambda_L \lambda_R^{-1}]$, we are able to construct a unique linear extension $Q'$ of $Q$. We then show that given a linear extension of $Q$ we are able to construct linear extensions $L$ of $\Delta(T_L)$, $R$ of $\nabla(T_R,p_L)$ and $P'$ of $P[Y_{L};\lambda_L \lambda_R^{-1}]$, and we show that this triple is uniquely dependent on $Q'$.
    
    Let $L$ be a linear extension of $\Delta(T_L)$, $R$ be a linear extension of $\nabla(T_R,p_L)$, and $P = P[Y_{L};\lambda_L\lambda_R^{-1}]$. 
    Now, suppose $P'$ is a linear extension of $P$ and define a chain $Q'$ by letting the $i$-th least element of $Q'$ be the $j$-th node in the $l$-th chain of $Q$ if the $i$-th least element of $P'$ is the $j$-th node in the $\lambda_L(l)$-th chain of $P$. We now show that $Q'$ satisfies the conditions of being a linear extension of $Q$.
    
    First, note that the $l$-th chain of $Q$ and the $\lambda_L(l)$-th chain of $P$ are of the same length as $X_l = {(Y_L)}_{\lambda_L(l)}$.
    Thus, $Q$ and $Q'$ have the same underlying set. We now aim to show that if a node is less than another in $Q$ then the same happens in $Q'$. By the definition of $Q'$, in $Q'$, the $j$-th node in the $l$-th chain of $Q$ is less than $(j+1)$-th node in the $l$-th chain.
    
    Suppose that the least node in the $i$-th chain of $Q$ is less than the least node in the $j$-th chain, that is, if $i < j$ in $\Delta(T_L)$.
    By the definition of $Q'$, this happens in $Q'$ if, in $P'$, the least node in the $\lambda_L(i)$-th chain of $P$ is less than the least node in the $\lambda_L(j)$-th chain, that is, if $\lambda_L(i) < \lambda_L(j)$, which is true as $i < j$ in $\Delta(T_L)$ and $L$ is a linear extension of $T_L$.
    
    Next, suppose that the greatest node in the $i$-th chain of $Q$ is less than the greatest node in the $j$-th chain, that is, if $i < j$ in $\nabla(T_R)$. Again, by the definition of $Q'$, this happens in $Q'$ if, in $P'$, the greatest node in the $\lambda_L(i)$-th chain of $P$ is less than the greatest node in the $\lambda_L(j)$-th chain of $P$, that is, by the definition of $P$, if there exist $s < t$ such that $\lambda_L(i) = \lambda_L\lambda_R^{-1}(s)$ and $\lambda_L(j) = \lambda_L\lambda_R^{-1}(t)$. So it suffices to show that $\lambda_R(i) < \lambda_R(j)$, which is true as $i < j$ in $\nabla(T_R)$ and $R$ is a linear extension of $\nabla(T_R,p_L)$, and hence a linear extension of $\nabla(T_R)$.
    Therefore, $Q'$ is a linear extension of $Q$, and as such, this method gives a well-defined function.

    We now show that this function is injective. Given linear extensions $P_1' \neq P_2'$ of $P$, it follows from the definition of $Q'$ that linear extensions $Q_1'$ and $Q_2'$ of $Q$, obtained respectively from $P_1'$ and $P_2'$ by the method given above, must be different. Furthermore, notice that $L$ (resp. $R$) determines the order, in $Q'$, of the least (resp. greatest) nodes of the chains of $Q$. Thus, different choices of $L$ or $R$ result in different linear extensions of $Q$. Hence, this function is injective.
    
    Now, let $Q'$ be a linear extension of $Q$. Let $L$ (resp. $R$) be the chain on $[k]$ where $i < j$ if, in $Q'$, the least (resp. greatest) node in the $i$-th chain of $Q$ is less than the least (resp. greatest) node in the $j$-th chain.
    Clearly, $L$ and $R$ are linear extensions of $\Delta(T_L)$ and $\nabla(T_R)$, respectively. Furthermore, for any chain of $Q$ with only one node, if the node is less than the least node in the $j$-th chain of $Q$, then it is less than the greatest node in the $j$-th chain. Thus, $R$ is a linear extension of $\nabla(T_R,p_L)$ as well. 
    
    Let $P = P[Y_{L};\lambda_L\lambda_R^{-1}]$. Define a chain $P'$ by letting the $i$-th least element of $P'$ be the $j$-th node in the $l$-th chain of $P$ if the $i$-th least element of $Q'$ is the $j$-th node in the $\lambda_L^{-1}(l)$-th chain of $Q$. We now show that $P'$ satisfies the conditions of being a linear extension of $P$. 
    
    First, note that the $l$-th chain of $P$ and the $\lambda^{-1}_L(l)$-th chain of $Q$ are of the same length as $(Y_L)_{l} = X_{\lambda_L^{-1}(l)}$.
    Thus, $P$ and $P'$ have the same underlying set. Furthermore, by the definition of $P'$, in $P'$, the $j$-th node in the $l$-th chain of $P$ is less than the $(j+1)$-th node in the $l$-th chain.
    
    Suppose $i < j$. We now show that, in $P'$, the least node in the $i$-th chain of $P$ is less than the least node in the $j$-th chain. By the definition of $P'$, this happens in $P'$ if, in $Q'$, the least node in the $\lambda_L^{-1}(i)$-th chain of $Q$ is less than the least node in the $\lambda_L^{-1}(j)$-th chain, that is, if $i = \lambda_L\lambda_L^{-1}(i) < \lambda_L\lambda_L^{-1}(j) = j$, which holds. 
    
    Finally, we show that, in $P'$, the greatest element in the $\lambda_L\lambda_R^{-1}(i)$-th chain of $P$ is less than the greatest element in the $\lambda_L\lambda_R^{-1}(j)$-th chain.
    By the definition of $P'$, this happens in $P'$ if, in $Q'$, the greatest node in the $\lambda_L^{-1}\lambda_L\lambda_R^{-1}(i)$-th chain of $Q$ is less than the greatest node in the $\lambda_L^{-1}\lambda_L\lambda_R^{-1}(j)$-th chain, that is, if 
    \[\lambda_R\lambda_L^{-1}\lambda_L\lambda_R^{-1}(i) = i < j = \lambda_R\lambda_L^{-1}\lambda_L\lambda_R^{-1}(j),\]
    which holds. Therefore, $P'$ is a linear extension of $P$, and as such, this method gives a well-defined function.

    We now show that this function is injective. Suppose that, given linear extensions $Q_1' \neq Q_2'$ of $Q$, we obtain, by the method above, the same linear extensions $L$ of $\Delta(T_L)$ and $R$ of $\nabla(T_R,p_L)$. As such, $P_1$ and $P_2$ as given by the method above are equal. Let $P_1'$ and $P_2'$ be linear extensions of $P_1$, obtained by the method above from $Q_1'$ and $Q_2'$, respectively. By the definition of $P'$, it follows that $P_1' \neq P_2'$. Thus, different choices of $Q$ which give the same $L$ and $R$ result in different linear extensions of $P$. Hence, this function is injective.

    As such, we can conclude that the cardinality of the set of linear extensions of $Q$ is the sum of cardinalities of the sets of linear extensions of posets of the form $P[Y_L;\lambda_L\lambda_R^{-1}]$, summing over all linear extensions $L$ of $\Delta(T_L)$ and $R$ of $\nabla(T_R,p_L)$.
\end{proof}

The number of linear extensions of $\Delta(T_L)$ and $\nabla(T_R,p_L)$ depends on the number of nodes of $T_L$ and $T_R$, but not on the sum of multiplicities of the nodes. By \cite{atkinson_linear_extensions_trees}, computing each of these numbers takes $\mathcal{O}(k^2)$ operations. As such, computing $\cL(Q[X;T_L;T_R])$ takes $\mathcal{O}(n^{2k-2}k!k^4)$ operations, where $n$ is the sum of the multiplicities of the nodes. 

As such, from Proposition~\ref{prop:mtg_poset_bijection}, we can conclude that using the algorithm given in Theorem~\ref{theorem:recursive_formula_linear_extensions_mtg_poset} allows us to compute the size of $\equiv_\mtg$-classes in $\mathcal{O}(n^{2k-2}k!k^4)$ operations, where $n$ is the length of the words and $k$ is the size of their support. Notice that, for a fixed support, as the length of a word increases, the behaviour of the time-complexity of counting the linear extensions of the $\equiv_\mtg$-class of the word is the same as that of counting the linear extensions of $\equiv_\mst$-class of the same word.


\subsection{`Hook-length'-like formulas for pairs of twin stalactic tableaux and binary search trees with multiplicities} \label{subsection:Pairs_of_twin_stalactic_tableaux_and_binary_search_trees_with_multiplicities}

For the rest of this section, we focus on `hook-length'-like formulas for the meet-stalactic and meet-taiga $\psymb$-symbols. Given a stalactic tableau, we determine how many distinct stalactic tableaux form a pair of twin stalactic tableaux with it. We then consider the case of BSTMs and find bounds for the corresponding question.

\begin{proposition}
    Let $T_L$ be a stalactic tableau. Let $a_1,\dots,a_k$ be the labels of the simple columns in $T_L$, ordered from left-to-right. For $1 \leq i \leq k$, let $m_i$ be the number of columns to the right of and including the column labelled $a_i$. Then, there are 
    \[\frac{|\supp{T_L}|!}{m_1\cdots m_k}\]
    distinct stalactic tableaux $T_R$ such that $(T_L,T_R)$ is a pair of twin stalactic tableaux.
\end{proposition}
\begin{proof}
    Recall that stalactic tableaux are uniquely determined by their top row and content. Thus, there are $|\supp{T_L}|!$ distinct stalactic tableaux with the same content as $T_L$. By the definition of \hyperref[definition:TwinStalacticTableauxCondition]{pairs of twin stalactic tableaux}, we need to count the number of stalactic tableaux $T_R$ such that for every simple column labelled $c$ and column labelled $d$, if $\rho_{T_L}(c) < \rho_{T_L}(d)$, then $\rho_{T_R}(c) < \rho_{T_R}(d)$. 
    
    Let $\word$ be the reading of the top row of $T_L$. Then, there exist words $\word_i$ over $\supp{T_L}$ such that $\word = \word_1a_1\word_2\cdots \word_k a_k\word_{k+1}$.
    A rearrangement $\uord$ of $\word$ is the reading of the top row of a stalactic tableau $T_R$, where $(T_L,T_R)$ is a pair of twins stalactic tableaux, if and only if $\uord[X_i]$ begins with $a_i$ for each $1 \leq i \leq k$, where $X_i = \supp{a_i \word_{i+1} \cdots \word_k a_k \word_{k+1}}$. 

    We now recursively define any such rearrangement $\uord$. View $\uord$ as an empty $|\word|$-tuple. From $1 \leq i \leq k$, reading $\word_i$ from left-to-right, choose any unoccupied positions in $\uord$ to be the positions of the letters of $\word_i$, then choose the least remaining unoccupied position to be the position of $a_i$. Finally, add the letters of $\word_{k+1}$ to the remaining unoccupied positions of $\uord$.
        
    Note that there are $|\supp{T_L}|+1-i$ choices for the position of the $i$-th letter of $\word$ in $\uord$, unless it is simple, then there is only one option. Hence, as each rearrangement $\uord$ corresponds to a distinct stalactic tableau $T_R$, there are
    \[\frac{|\supp{T_L}|!}{m_1\cdots m_k}\]
    distinct stalactic tableaux $T_R$ such that $(T_L,T_R)$ is a pair of twin stalactic tableaux, where $m_i = |X_i|$.
\end{proof}

Notice that, given a stalactic tableau $T_R$, one can obtain a similar formula for the number of distinct stalactic tableaux $T_L$ such that $(T_L,T_R)$ is a pair of twin stalactic tableaux, by symmetrical reasoning and considering the number of columns to the left of and including the simple columns instead.

\begin{proposition}
    Let $T_L$ be a BSTM with $n$ nodes. Suppose there are $k$ simple nodes in $T_L$ that are not leaves. Let $R(T_L)$ denote the number of distinct BSTMs $T_R$ such that $(T_L,T_R)$ is a pair of twin BSTMs. Then, 
    \[ C_{n-k} \leq R(T_L) \leq C_n \]
    where $C_m = \frac{(2m)!}{(m+1)!m!}$ is the $m$-th Catalan number.
\end{proposition}
\begin{proof}
    It is a well-known fact that there are exactly $C_n$ binary trees with $n$ nodes, thus $R(T_L) \leq C_n$. We aim to show that $C_{n-k} \leq R(T_L)$. First note that, as there are $k$ nodes which are simple non-leaf nodes, there are $n-k$ nodes which are either leaves or non-simple.

    Let $\word$ be such that $\pltg{\word} = T_L$. Then, let $\word_1$ (resp. $\word_2$) be the word obtained from $\word$ by removing (resp. keeping) all letters except the first occurrence of each letter which does not label a leaf in $T_L$.
    Note that for any rearrangement $\vord$ of $\word_2$, $\pltg{\word_1 \vord} = T_L$ as $\supp{\word_1}$ contains every letter that does not label a leaf in $T_L$. On the other hand, for each pair of rearrangements $\vord$, $\vord'$ of $\word_2$ with $\prtg{\vord} \neq \prtg{\vord'}$, we have that $\prtg{\word_1\vord} \neq \prtg{\word_1\vord'}$, since $\rtg$ is left-cancellative. As such, $R(T_L) \geq C_{n-k}$ as there are $C_{n-k}$ BSTMs with the same labels as $\prtg{\word_2}$, since $|\supp{\word_2}| = n-k$.
\end{proof}


\section{Syntacticity of plactic-like congruence} \label{section:Syntacticity_of_plactic_like_monoids}

Recall that, for each $\N$-labelled combinatorial object which identifies classes of a plactic-like monoid, its \textit{shape} is the unlabelled object, that is, unlabelled by letters of $\N$. For each plactic-like monoid $\mathsf{M}$ with associated $\psymb$-symbol $\psymb_\mathsf{M}$, let $\mathrm{Sh}_\mathsf{M}$ be the function which maps a word $\word \in \N^*$ to the shape of its combinatorial object $\psymb_\mathsf{M}{(\word)}$. We now study the syntacticity of plactic-like congruences with regard to their corresponding shape functions.

\begin{remark}
    We consider the shape functions of BSTMs and pairs of twin BSTMs that give us combinatorial objects still labelled by multiplicities. 
\end{remark}
The following is a generalisation of the usual definition of syntactic congruences, extended to functions from a free monoid to any set: Let $S$ be any set and let $f \colon \N^* \to S$. We say a congruence $\equiv$ of $\N^*$ is the (two-sided) \emph{syntactic congruence} of $f$ if 
\[ \uord \equiv \vord \Leftrightarrow (\forall \rord,\sord \in \N^*, f(\rord \uord \sord) = f(\rord \vord \sord)).\]
This is the coarsest congruence of $\N^*$ which is compatible with $f$. Similarly, a left congruence $\equiv$ is the \emph{left syntactic congruence} of $f$ if
\[ \uord \equiv \vord \Leftrightarrow (\forall \rord \in \N^*, f(\rord \uord) = f(\rord \vord)).\]
This is the coarsest left congruence of $\N^*$ which is compatible with $f$. One defines a \emph{right syntactic congruence} in a parallel way. Remark that, for congruences, left or right syntacticity implies syntacticity.

Lascoux and Schützenberger stated without proof \cite[Theoréme~2.15]{LS1978} that the plactic congruence is the syntactic congruence of the shape function of (semi-standard) Young tableaux. Recently, Abram and Reutenauer have shown a generalisation of this result, namely that the plactic congruence is also the left and right syntactic congruence of said shape function \cite[Theorem~13.3]{abram_reutenauer_stylic_monoid_2021}. On the other hand, Novelli has shown that the hypoplactic congruence is both the left and right syntactic congruence of the shape function of quasi-ribbon tableaux \cite[Subsection~5.4]{novelli_hypoplactic}. We now look at the sylvester, Baxter, stalactic and taiga cases. Notice that, since we are already working with two-sided congruences, we only need to prove the necessary condition of the definition of (left or right) syntacticity.

Recall the right strict insertion algorithm given in \cite[Definition~7]{hivert_sylvester}, which computes a unique right strict binary search tree $\psylv{\word}$ from a word $\word \in \N^*$.

\begin{proposition} \label{SylvesterLeftRightSyntactic}
    The sylvester (resp. \#-sylvester) congruence is both the left and right syntactic congruence of $\mathrm{Sh}_\sylv$ (resp. $\mathrm{Sh}_\sylvh$).
\end{proposition}

\begin{proof}
    We prove the result for the sylvester monoid only, as the case of the \#-sylvester monoid is analogous. Let $\uord,\vord \in \N^*$ be such that $\uord \not\equiv_\sylv \vord$ and $\shsylv{\uord} = \shsylv{\vord}$. Then, since there is at most one way to label a binary tree with the letters of a word and obtain a right strict binary search tree, we must have $\cont{\uord} \neq \cont{\vord}$. 
    
    Let $a \in \N$ be the least letter such that $|\uord|_a \neq |\vord|_a$. Assume, without loss of generality, that $|\uord|_a > |\vord|_a$. Then, $\psylv{\uord a}$ and $\psylv{\vord a}$ will both have root node labelled $a$. Notice that the number of nodes of the left subtree of the root nodes in $\psylv{\uord a}$ (resp. $\psylv{\vord a}$) is given by $\sum_{c \in [a]} |\uord|_c$ (resp. $\sum_{c \in [a]} |\vord|_c$). By hypothesis, if $c < a$, then $|\uord|_c = |\vord|_c$. As such, the left subtree of the root node in $\psylv{\uord a}$ will have more nodes than the left subtree of the root node in $\psylv{\vord a}$, and thus $\shsylv{\uord a} \neq \shsylv{\vord a}$. Therefore, $\equiv_\sylv$ is the right syntactic congruence of $\mathrm{Sh}_{\sylv}$.

    Now, let $b \in \N$ be the greatest letter such that $|\uord|_b \neq |\vord|_b$. Assuming, in order to obtain a contradiction, that $\psylv{b\uord}$ and $\psylv{b\vord}$ have the same shape then we can conclude that $\psylv{b\uord}$ and $\psylv{b\vord}$ both have a leaf node labelled $b$, which has the same position in the inorder traversals of the trees. Recall that the inorder reading of a right strict binary search tree gives the (unique) weakly increasing reading of the tree, and that the first occurrence of $b$ in the inorder reading corresponds to the first node labelled $b$ in the inorder traversal, which is always the node labelled $b$ of greatest depth. Since $\psylv{b\uord}$ and $\psylv{b\vord}$ have different content, their inorder reading is different, and in particular, the shortest suffix containing all the letters $b$ is of greater length in one reading than in the other. As such, the leaf node labelled $b$ is in a different position in the inorder traversals of the trees, which contradicts our hypothesis, and hence implies that $\shsylv{b\uord} \neq \shsylv{b\vord}$. Therefore, $\equiv_\sylv$ is the left syntactic congruence of $\mathrm{Sh}_{\sylv}$.
\end{proof}

Notice that $\mathrm{Sh}_{\sylv}$ is not the same function as $\mathrm{Sh}_{\sylvh}$: they are both functions from $\N^*$ to the set of binary trees, however, for $\word \in \N^*$, $\mathrm{Sh}_{\sylv}(\word)$ (resp. $\mathrm{Sh}_{\sylvh}(\word)$) is defined as the shape of the right (resp. left) strict binary search tree obtained from $\word$ by the right (resp. left) strict insertion algorithm. For example, 
\[
    \mathrm{Sh}_{\sylvh}(12) = \begin{tikzpicture}[tinybst,baseline=-3mm]
            \node {}
            child[missing]
            child { node {}}
        ;
        \end{tikzpicture}
    \quad \text{and} \quad
    \mathrm{Sh}_{\sylv}(12) = \begin{tikzpicture}[tinybst,baseline=-3mm]
            \node {}
            child { node {}}
            child[missing]
        ;
        \end{tikzpicture}
\]

\begin{proposition} \label{BaxterLeftRightSyntactic}
    The Baxter congruence is both the left and right syntactic congruence of $\mathrm{Sh}_\baxt$.
\end{proposition}

\begin{proof}
    Let $\uord,\vord \in \N^*$ be such that $\uord \not\equiv_\baxt \vord$. Then, $\uord \not\equiv_\sylv \vord$ or $\uord \not\equiv_\sylvh \vord$. Assume, without loss of generality, that $\uord \not\equiv_\sylv \vord$. Since $\equiv_\sylv$ is both the left and right syntactic congruence of $\mathrm{Sh}_{\sylv}$, there exist $\rord,\sord \in \N^*$ such that $\shsylv{\rord \uord} \neq \shsylv{\rord \vord}$ and $\shsylv{\uord \sord} \neq \shsylv{\vord \sord}$. Since the shape of pairs of twin binary search trees is determined by the shape of its left and right binary search trees, we have that $\shbaxt{\rord\uord} \neq \shbaxt{\rord\vord}$ and $\shbaxt{\uord\sord} \neq \shbaxt{\vord\sord}$.
\end{proof}

\begin{proposition} \label{StalacticSyntactic}
    The right-stalactic (resp. left-stalactic) congruence is the left (resp. right) syntactic congruence of $\mathrm{Sh}_\rst$ (resp. $\mathrm{Sh}_\lst$), but not its right (resp. left) syntactic congruence.
\end{proposition}

\begin{proof}
    We prove the result for the right-stalactic monoid only, as the case of the left-stalactic monoid is analogous. Let $\uord,\vord \in \N^*$ be such that $\uord \not\equiv_\rst \vord$ and $\shrst{\uord} = \shrst{\vord}$. Notice that, if $\cont{\uord} \neq \cont{\vord}$, then, for $a \in \N$ such that $|\uord|_a > |\vord|_a$, we have $\shrst{a\uord} \neq \shrst{a\vord}$, since these stalactic tableaux are obtained by adding a new $a$-cell to the column of $a$-cells, which has a different height in $\prst{\uord}$ than in $\prst{\vord}$. On the other hand, if $\cont{\uord} = \cont{\vord}$, then $\rho_{\prst{\uord}} \neq \rho_{\prst{\vord}}$. Choosing $b \in \N$ such that $\rho_{\prst{\uord}}(b) \neq \rho_{\prst{\vord}}(b)$, we have that $\shrst{b\uord} \neq \shrst{b\vord}$, since we are increasing the heights of different columns in the shapes of the stalactic tableaux. Therefore, $\equiv_\rst$ is the left syntactic congruence of $\mathrm{Sh}_{\rst}$.

    On the other hand, consider the words $1221$ and $1122$. Notice that
    \begin{center}  
	$\prst{1221} = \tikz[tableau]\matrix{
		2 \& 1 \\
		2 \& 1 \\
	}; \quad \text{and} \quad
    \prst{1122} = \tikz[tableau]\matrix{
		1 \& 2 \\
		1 \& 2 \\
	};$
    \end{center}
    and in particular, $\shrst{1221} = \shrst{1122}$ and $\cont{1221} = \cont{1122}$. Furthermore, notice that, for any $\uord \in \N^*$, we have that $\shrst{1221\uord} = \shrst{1122\uord}$, since $\prst{1221\uord}$ (resp. $\prst{1122\uord}$) is obtained from $\prst{\uord}$ by attaching to the left of it all columns of $\prst{1221}$ (resp. $\prst{1122}$) labelled by letters which do not occur in $\uord$, followed by attaching the remaining columns to the bottom of their respectively labelled columns in $\prst{\uord}$. Therefore, $\equiv_\rst$ is not the right syntactic congruence of $\mathrm{Sh}_{\rst}$.
\end{proof}

\begin{proposition} \label{MeetStalacticSyntactic}
    The meet-stalactic congruence is neither the left nor the right syntactic congruence of $\mathrm{Sh}_\mst$, but it is still the (two-sided) syntactic congruence of $\mathrm{Sh}_\mst$.
\end{proposition}

\begin{proof}
    Consider the words $1122$, $1221$ and $2112$, whose pairs of twin stalactic tableaux are, respectively,
    \begin{center}  
	$\left(\tikz[tableau]\matrix{
		1 \& 2 \\
		1 \& 2 \\
	}; , \tikz[tableau]\matrix{
		1 \& 2 \\
		1 \& 2 \\
	};\right), \quad
    \left(\tikz[tableau]\matrix{
		1 \& 2 \\
		1 \& 2 \\
	}; , \tikz[tableau]\matrix{
		2 \& 1 \\
		2 \& 1 \\
	};\right), \quad \text{and} \quad
    \left(\tikz[tableau]\matrix{
		2 \& 1 \\
		2 \& 1 \\
	}; , \tikz[tableau]\matrix{
		1 \& 2 \\
		1 \& 2 \\
	};\right)$.
    \end{center}
    Notice that these pairs of twin stalactic tableaux all have the same content and shape, and that $\plst{1122} = \plst{1221}$ and $\prst{1122} = \prst{2112}$. By the reasoning given in the previous proof, we have that, for any $\uord \in \N^*$, we have that $\shrst{1122\uord} = \shrst{1221\uord}$ and, by parallel reasoning, $\shlst{\uord 1122} = \shlst{\uord 2112}$. Furthermore, since $\mst$ is a congruence, we also have $\plst{1122 \uord} = \plst{1221 \uord}$ and $\prst{\uord 1122} = \prst{\uord 2112}$. As such, we have that $\shmst{1122\uord} = \shmst{1221\uord}$ and $\shmst{\uord 1122} = \shmst{\uord 2112}$. Therefore, $\equiv_\mst$ is neither the left nor the right syntactic congruence of $\mathrm{Sh}_\mst$.

    Let $\uord,\vord \in \N^*$ be such that $\uord \not\equiv_\mst \vord$. Then, $\uord \not\equiv_\rst \vord$ or $\uord \not\equiv_\lst \vord$. Assume, without loss of generality, that $\uord \not\equiv_\rst \vord$. Since $\equiv_\rst$ is the left syntactic congruence of $\mathrm{Sh}_{\rst}$, there exists $\word \in \N^*$ such that $\shrst{\word\uord} \neq \shrst{\word\vord}$. Since the shape of pairs of twin stalactic tableaux is determined by the shape of its (left and right) twin stalactic tableaux, we have that $\shmst{\word\uord} \neq \shmst{\word\vord}$. Similarly, one can show that if $\uord \not\equiv_\lst \vord$, then there exists $\word \in \N^*$ such that $\shmst{\uord\word} \neq \shmst{\vord\word}$. Therefore, $\equiv_\mst$ is the syntactic congruence of $\mathrm{Sh}_\mst$.
\end{proof}

\begin{proposition} \label{TaigaSyntactic}
    The right-taiga (resp. left-taiga) congruence is both the left and right syntactic congruence of $\mathrm{Sh}_\rtg$ (resp. $\mathrm{Sh}_\ltg$).
\end{proposition}

\begin{proof}
    We prove the result for the right-taiga monoid only, as the case of the left-taiga monoid is analogous. Let $\uord,\vord \in \N^*$ be such that $\uord \not\equiv_\rtg \vord$ and $\shrtg{\uord} = \shrtg{\vord}$. Then, by the observation given after \cite[Definition~8]{priez_binary_trees}, we must have $\supp{\uord} \neq \supp{\vord}$.

    Let $a \in \N$ be any letter such that $|\uord|_a \neq 0$ and $|\vord|_a = 0$. Then, $\prtg{\uord a}$ and $\prtg{\vord a}$ will have different shapes, since $\prtg{\vord a}$ will have one more node than $\prtg{\uord a}$. For the same reason, $\prtg{a\uord}$ and $\prtg{a\vord}$ will have different shapes. Therefore, $\equiv_\rtg$ is both the right and left syntactic congruence of $\mathrm{Sh}_{\rtg}$.
\end{proof}

As with the case of $\mathrm{Sh}_{\sylv}$ and $\mathrm{Sh}_{\sylvh}$, notice that $\mathrm{Sh}_{\rtg}$ is not the same function as $\mathrm{Sh}_{\ltg}$: they are both functions from $\N^*$ to the set of BTMs, however, for $\word \in \N^*$, $\mathrm{Sh}_{\rtg}(\word)$ (resp. $\mathrm{Sh}_{\ltg}(\word)$) is defined as the shape of the BSTM obtained from $\word$ by reading it from right-to-left (resp. left-to-right) and inserting letters using Algorithm \hyperref[alg:TaigaLeafInsertion]{\textsf{TgLI}}. 

\begin{proposition} \label{MeetTaigaSyntactic}
    The meet-taiga congruence is both the left and right syntactic congruence of $\mathrm{Sh}_\mtg$.
\end{proposition}

\begin{proof}
    The proof is analogous to the proof of Proposition~\ref{BaxterLeftRightSyntactic}.
\end{proof}

In Table \ref{table:syntacticity}, we give a summary of the results on syntacticity of plactic-like congruences.

\begin{table}[h]
    \caption{Syntacticity of plactic-like congruences with regards to shape functions.}
    \label{table:syntacticity}
    \begin{tabular}{c|cccc}
        \toprule
        \textbf{Congruence} & \textbf{Syntactity} & \textbf{Left Syn.} & \textbf{Right Syn.} & \textbf{Result}\\
        \midrule
        $\plac$ & \ding{51} & \ding{51} & \ding{51} & \cite[Theorem~13.3]{abram_reutenauer_stylic_monoid_2021}\\
        $\hypo$ & \ding{51} & \ding{51} & \ding{51} & \cite[Subsection~5.4]{novelli_hypoplactic}\\
        $\sylv$ & \ding{51} & \ding{51} & \ding{51} & Proposition~\ref{SylvesterLeftRightSyntactic}\\
        $\sylvh$ & \ding{51} & \ding{51} & \ding{51} & Proposition~\ref{SylvesterLeftRightSyntactic}\\
        $\baxt$ & \ding{51} & \ding{51} & \ding{51} & Proposition~\ref{BaxterLeftRightSyntactic}\\
        $\lst$ & \ding{51} & \ding{55} & \ding{51} & Proposition~\ref{StalacticSyntactic}\\
        $\rst$ & \ding{51} & \ding{51} & \ding{55} & Proposition~\ref{StalacticSyntactic}\\
        $\mst$ & \ding{51} & \ding{55} & \ding{55} & Proposition~\ref{MeetStalacticSyntactic}\\
        $\ltg$ & \ding{51} & \ding{51} & \ding{51} & Proposition~\ref{TaigaSyntactic}\\
        $\rtg$ & \ding{51} & \ding{51} & \ding{51} & Proposition~\ref{TaigaSyntactic}\\
        $\mtg$ & \ding{51} & \ding{51} & \ding{51} & Proposition~\ref{MeetTaigaSyntactic}\\
        \bottomrule
    \end{tabular}
    \end{table}
    

\section{Equational theories of the meet and join-stalactic and meet and join-taiga monoids} \label{section:Equational_theories_of_meet_and_join_stalactic_and_taiga_monoids}

We now characterise the identities satisfied by the monoids we defined. For a general background on universal algebra, see \cite{Bergman_universal_algebra,bs_universal_algebra}. The following background is given in the context of monoids.

An \emph{identity} over an alphabet of variables $\X$ is a formal equality $\uord \approx \vord$, where $\uord$ and $\vord$ are words over $\X$, said to be \emph{non-trivial} if $\uord \neq \vord$. A variable $x$ \emph{occurs} in $\uord \approx \vord$ if $x$ occurs in $\uord$ or $\vord$, and $\uord \approx \vord$ is said to be \emph{balanced} if $\uord$ and $\vord$ have the same content. Since two words with the same content must have the same length, we say the \emph{length} of a balanced identity is the length of its left or right-hand side. Two identities are equivalent if one can be obtained from the other by renaming variables or swapping both sides of the identities.

A monoid $M$ \emph{satisfies} the identity $\uord \approx \vord$ if for every morphism $\psi\colon \X^* \to M$, referred to as an \emph{evaluation}, we have $\psi(\uord) = \psi(\vord)$. The \emph{identity checking problem} of a monoid $M$ is the combinatorial decision problem $\textsc{Check-Id}{(M)}$ of deciding whether an identity is satisfied or not by $M$. The input of this problem is simply the identity, hence its time complexity is measured in terms of the size of the identity, that is, the sum of the lengths of each side of the formal equality. 

Given a class of monoids $\mathbb{K}$, the set of all identities simultaneously satisfied by all monoids in $\mathbb{K}$ is called its \emph{equational theory}. On the other hand, given a set of identities $\Sigma$, the class of all monoids that satisfy all identities in $\Sigma$ is called its \emph{variety}. By Birkhoff's $HSP$-theorem, a class of monoids is a variety if and only if is closed under taking homomorphic images, submonoids and direct products. A \emph{subvariety} is a subclass of a variety which is itself a variety. We say a variety is \emph{generated} by a monoid $M$ if it is the least variety containing $M$, and denote it by $\mathbb{V}(M)$. 

A congruence $\equiv$ on a monoid $M$ is \emph{fully invariant} if $a \mathrel{\equiv} b$ implies $f(a) \mathrel{\equiv} f(b)$, for every $a,b \in M$ and every endomorphism $f$ of $M$. Equational theories can be viewed as fully invariant congruences on $\X^*$ (see, for example, \cite[II\S{14}]{bs_universal_algebra}). The \emph{free object} of rank $n$ of a variety $\mathbb{V}$ is the quotient of the free monoid over an $n$-letter alphabet by the equational theory of $\mathbb{V}$. The equational theory of the varietal join (resp. meet) of varieties $\mathbb{V}_1$ and $\mathbb{V}_2$ is the meet (resp. join) of the equational theories of $\mathbb{V}_1$ and $\mathbb{V}_2$.

An identity $\uord \approx \vord$ is a \emph{consequence} of a set of identities $\Sigma$ if it is in the equational theory generated by $\Sigma$. An \emph{equational basis} $\cB$ of a variety $\mathbb{V}$ is a subset of its equational theory such that each identity of the equational theory is a consequence of $\cB$. An equational theory is \emph{finitely based} if it admits a finite equational basis. The \emph{axiomatic rank} of $\mathbb{V}$ is the least natural number such that $\mathbb{V}$ admits a basis where the number of distinct variables occurring in each identity of the basis does not exceed said number. Notice that if a variety is finitely based, then it has finite axiomatic rank.

The identities satisfied by $\rst$ have been independently studied by Cain \textit{et al.} \cite{cain_johnson_kambites_malheiro_representations_2022} and Han and Zhang \cite{han2021preprint}:

\begin{theorem}[{\cite[Corollary~4.6]{cain_johnson_kambites_malheiro_representations_2022},\cite[Lemma~2.1~and~Theorem~2.3]{han2021preprint}}] \label{theorem:rst_identity_characterisation}
The identities satisfied by $\rst$ are precisely the balanced identities of the form $\uord \approx \vord$, where $\overleftarrow{\rho_\uord} = \overleftarrow{\rho_\vord}$.
\end{theorem}
By symmetric reasoning, we have that the identities satisfied by $\lst$ are precisely the balanced identities of the form $\uord \approx \vord$, where $\overrightarrow{\rho_\uord} = \overrightarrow{\rho_\vord}$.

\begin{theorem}[{\cite[Corollary~4.6]{cain_johnson_kambites_malheiro_representations_2022},\cite[Theorem~2.3]{han2021preprint}}] \label{theorem:rst_identity_basis}
The variety generated by $\rst$ admits a finite basis, consisting of the identity
    \[ xyx \approx yxx. \]
\end{theorem}
By symmetric reasoning, we have that the variety generated by $\lst$ admits a finite equational basis, consisting of the identity
\[ xyx \approx xxy. \]
From these results, and the fact that identities satisfied by either $\vlst$ or $\vrst$ must be balanced, it follows that both varieties have axiomatic rank $2$.

The following hasn't been stated in the literature, but it is an immediate consequence of the previous result:

\begin{proposition} 
    \label{prop:stalactic_free_object}
    The left-stalactic (resp. right-stalactic) monoid of rank $n$ is the free object of rank $n$ of $\vlst$ (resp. $\vrst$), for countable $n$.
\end{proposition}

\begin{proof}
    We prove the case of $\lst$, since the case of $\rst$ is analogous. The left-stalactic congruence can also be defined as the congruence on $\N^*$ generated by the relations $(a \word a, a a \word)$, for all $a \in \N$ and $\word \in \N^*$ \cite[Section~3.7]{hnt_stalactic}. On the other hand, the equational theory of $\lst$ can be viewed as a congruence on $\X^*$ generated by the relations $(xyx,xxy)$, for all $x,y \in \X$. Since it is fully invariant, it is also generated by the same relations as the left-stalactic congruence, hence they are the same. As such, $\lst$ is isomorphic to the free object of countably infinite rank of $\vlst$. Since $\lst$ is compatible with restriction to alphabet intervals, the result also holds for finite rank $n$.
\end{proof}

As a consequence, the left-stalactic and right-stalactic congruences can be viewed as equational theories of $\vlst$ and $\vrst$. As such, the meet-stalactic congruence, given as the meet of $\equiv_\lst$ and $\equiv_\rst$, can also be viewed as the equational theory of the varietal join $\vrst \vee \vlst$. Hence, the variety generated by $\mst$, denoted by $\vmst$, is equal to $\vrst \vee \vlst$, and we have the following:

\begin{corollary}
\label{corollary:meet_stalactic_free_object}
The meet-stalactic monoid of rank $n$ is the free object of rank $n$ of $\vmst$, for countable $n$.
\end{corollary}

Therefore, the identities satisfied by the meet-stalactic monoid are exactly those identities that are simultaneously satisfied by $\rst$ and $\lst$:

\begin{corollary}
\label{corollary:mst_identities}
For $\uord,\vord \in \X^*$, the identity $\uord \approx \vord$ is satisfied by $\mst$ if and only if it is balanced, $\overrightarrow{\rho_\uord} = \overrightarrow{\rho_\vord}$ and $\overleftarrow{\rho_{\uord}} = \overleftarrow{\rho_{\vord}}$.
\end{corollary}

In other words, the identities satisfied by $\mst$ are those identities where the left and right-hand sides have the same content, and, when reading them from left-to-right, the order of the first and last occurrences of the variables is the same on both sides. Clearly, this condition is verifiable in polynomial time:

\begin{corollary}
\label{corollary:check_id_mst}
The decision problem $\textsc{Check-Id}(\mst)$ belongs to the complexity class $\mathsf{P}$.
\end{corollary}

From the presentation of $\mst$ given in Proposition~\ref{prop:MeetStalacticPresentation}, we obtain the following:

\begin{corollary}
\label{corollary:mst_finite_basis}
$\vmst$ admits a finite equational basis $\mathcal{B}_{\mst}$, consisting of the identities
    \begin{align}
        xzxyty \approx xzyxty, \label{id22}\\
        xzxytx \approx xzyxtx. \label{id31}
    \end{align}
\end{corollary}

\begin{proof}
    Since $\equiv_\mst$ is a fully invariant congruence, we can see that the relations given in Proposition~\ref{prop:MeetStalacticPresentation} can be deduced from the identities \eqref{id22} and \eqref{id31}. The result follows.
\end{proof}

The identity \eqref{id31} can be deduced from other non-equivalent identities satisfied by $\mst$, such as, for example, the identity $xyxzx \approx xxyzx$. Such is not the case with the identity \eqref{id22}, or any equivalent one, which must be in any equational basis for $\vmst$ where only up to four different variables occur in each identity. To show this, we need the following lemma:

\begin{lemma}
\label{lemma:mst_shortest}
The shortest non-trivial identity, with n variables, satisfied by $\mst$, is of length $n+2$.
\end{lemma}

\begin{proof}
    Let $\uord \approx \vord$ be a non-trivial identity, with $|\supp{\uord \approx \vord}|=n$, satisfied by $\mst$. Since it is a non-trivial identity, there exist $x,y \in \X$ and $\word,\uord',\vord' \in \X^*$ such that $\uord = \word x \uord'$ and $\vord=\word y \vord'$. Notice that, by Corollary~\ref{corollary:mst_identities}, $\uord \approx \vord$ is a balanced identity, therefore $x$ occurs in $\vord'$ and $y$ occurs in $\uord'$. Furthermore, since $\overrightarrow{\rho_\uord} = \overrightarrow{\rho_\vord}$, then $x$ occurs in $\word$ before the first occurrence of $y$, or $y$ occurs in $\word$ before the first occurrence of $x$. Similarly, since $\overleftarrow{\rho_{\uord}} = \overleftarrow{\rho_{\vord}}$, then $x$ occurs in $\uord'$ after the last occurrence of $y$, or $y$ occurs in $\vord'$ after the last occurrence of $x$. Thus, we can conclude that the length of $\uord \approx \vord$ is at least $n+2$.
    
    On the other hand, by Corollary~\ref{corollary:mst_identities}, for variables $x,y,a_1, \dots, a_{n-2} \in \X$, the identity
    \[ x y x a_1 \cdots a_{n-2} x \approx x x y a_1 \cdots a_{n-2} x, \]
    of length $n+2$, is satisfied by $\mst$. The result follows.
\end{proof}

\begin{proposition}
\label{proposition:id22_not_consequence}
    The identity \eqref{id22} is not a consequence of the set of non-trivial identities, satisfied by $\mst$, over an alphabet with four variables, excluding \eqref{id22} itself and equivalent identities. Therefore, any equational basis for $\vmst$ with only identities over an alphabet with four variables must contain the identity \eqref{id22}, or an equivalent identity.
\end{proposition}

\begin{proof}
    Let $\mathcal{S}$ be the set of all non-trivial identities, satisfied by $\mst$, over an alphabet with four variables. By definition, \eqref{id22} is a consequence of $\mathcal{S}$. As such, there exists a non-trivial identity $\uord \approx \vord$ in $\mathcal{S}$, and a substitution $\psi$, such that
    \[
    xzxyty = \word_1 \psi(\uord) \word_2,
	\]
	where $\word_1, \word_2$ are words over the four-letter alphabet, and $\psi(\uord) \neq \psi(\vord)$. We can assume, without loss of generality, that $\psi$ does not map any variable to the empty word. By Lemma~\ref{lemma:mst_shortest}, no proper factor of $xzxyty$ is the left-hand side of a non-trivial identity satisfied by $\mst$, hence $\word_1$ and $\word_2$ are the empty word, that is, $xzxyty = \psi(\uord)$.
	
	Since $\cont{xzxyty}=\bigl(\begin{smallmatrix} x & y & z & t \\ 2 & 2 & 1 & 1 \end{smallmatrix}\bigr)$ and, by Lemma~\ref{lemma:mst_shortest}, there is a lower bound on the length of identities satisfied by $\mst$, we can conclude that, up to renaming of variables, $x$ and $y$ occur exactly twice, and $z$ and $t$ can each occur at most once in $\uord \approx \vord$. Notice that, since all factors of length $2$ of $xzxyty$ are distinct, then $\psi(x)$ and $\psi(y)$ must be single variables, in particular, up to renaming, $x$ and $y$. Furthermore, since $zt$ is not a factor of $xzxyty$, then $\psi(z)$ and $\psi(t)$ must also be single variables, in particular, up to renaming, $z$ and $t$. As such, the length of $\uord$ is the same as that of $\psi(\uord)$, which implies that exactly four distinct variables occur in $\uord \approx \vord$. But, since we are considering only substitutions which do not map variables to the empty word, this implies that $\psi$ is only a renaming of variables. Notice that $xzxyty$ is the left-hand side of a non-trivial identity satisfied by $\mst$ if and only if the right-hand side is $xzyxty$, by Corollary~\ref{corollary:mst_identities}. Hence $\uord \approx \vord$ is equivalent to \eqref{id22}.
\end{proof}

As a consequence, we have the following:

\begin{corollary}
\label{corollary:mst_axiomatic_rank}
The axiomatic rank of $\mst$ is $4$.
\end{corollary}

Recall that the meet-taiga monoid $\mtg$ is given by the meet of the left and right-taiga congruences. Therefore, in parallel with the meet-stalactic case, we can show that the variety generated by $\mtg$ is the varietal join of the varieties generated, respectively, by $\ltg$ and $\rtg$. By \cite[Corollary~5.8]{cain_johnson_kambites_malheiro_representations_2022} and \cite[Theorem~2.3]{han2021preprint}, $\ltg$ (resp. $\rtg$) generates the same variety as $\lst$ (resp. $\rst$). Hence, we have that:

\begin{corollary}
\label{corollary:mst_mtg_same_variety}
    $\mtg$ generates the same variety as $\mst$.
\end{corollary}

Since the equational theory of $\mst$ is a strict subset of the equational theory of $\baxt$, given in \cite[Theorem~4.3]{cain_malheiro_ribeiro_sylvester_baxter_2023} and \cite[Lemma~3.3]{han2021preprint}, we have the following:

\begin{corollary}
\label{corollary:vmst_strictly_contained_vbaxt}
$\vmst$ is strictly contained in $\vbaxt$.
\end{corollary}

On the other hand, by \cite[Corollary~4.6]{cain_johnson_kambites_malheiro_representations_2022}, $\vrst$ is the varietal join of the variety $\mathbb{COM}$ of all commutative monoids and the variety $\mathbb{RRB}$ of all right regular bands (restricted to monoids). By a parallel argument, $\vlst$ is the varietal join of $\mathbb{COM}$ and the variety $\mathbb{LRB}$ of all left regular bands (restricted to monoids). Therefore, since $\vmst = \vrst \vee \vlst$, and the varietal join of $\mathbb{LRB}$ and $\mathbb{RRB}$ is the variety $\mathbb{RB}$ of all regular bands (restricted to monoids), we can conclude that:

\begin{corollary}
\label{corollary:vmst_varietal_join}
    $\vmst$ is the varietal join of $\mathbb{COM}$ and $\mathbb{RB}$.
\end{corollary}

Now, we consider the case of the join-stalactic monoid. We write $\vjst$ to denote the variety generated by $\jst$. As before in the meet-stalactic case, the join-stalactic congruence can be viewed as the equational theory of the varietal meet $\vlst \wedge \vrst$, hence $\vlst \wedge \vrst$ is equal to $\vjst$, and we have the following:

\begin{corollary}
\label{corollary:join_stalactic_free_object}
    The join-stalactic monoid of rank $n$ is the free object of rank $n$ of $\vjst$, for countable $n$.
\end{corollary}

This stands in contrast with the hypoplactic monoid: the hypoplactic congruence is the join of the sylvester and \#-sylvester congruences, however, the hypoplactic monoid does not generate the varietal meet of the varieties generated by the sylvester and \#-sylvester monoids, respectively. 

From Proposition~\ref{prop:jst_characterisation}, we also obtain the following:

\begin{corollary}
\label{corollary:jst_equational_theory}
The equational theory of $\jst$ is the set of balanced identities $\uord \approx \vord$ over the alphabet of variables $\X$ such that $\ol{\uord} = \ol{\vord}$.
\end{corollary}
    
It follows that checking if an identity holds in $\jst$ can be done in polynomial time:

\begin{corollary}
\label{corollary:check_id_jst}
The decision problem $\textsc{Check-Id}(\jst)$ belongs to the complexity class $\mathsf{P}$.
\end{corollary}

Recall that $\equiv_\lst$ and $\equiv_\rst$ are generated, respectively, by the relations $(a \word a, a a \word)$ and $(a \word a, \word a a)$, for all $a \in \N$ and $\word \in \N^*$. As such, $\equiv_\jst$ is generated by these relations, from which we obtain the following:

\begin{corollary}
\label{corollary:jst_finite_basis}
$\vjst$ admits a finite equational basis $\mathcal{B}_{\jst}$, consisting of the identities
    \begin{align}
        xxy \approx xyx \approx yxx. 
    \end{align}
\end{corollary}

Since the only identities satisfied by $\jst$ where only one variable occurs are trivial, we can conclude the following:

\begin{corollary}
\label{corollary:jst_axiomatic_rank}
The axiomatic rank of $\jst$ is $2$.
\end{corollary}

While $\jst$ satisfies the identities $xxy \approx xyx \approx yxx$, it is not commutative. As such, it is natural to ask which other monoids in $\vjst$ are non-commutative. 

\begin{proposition}
\label{prop:vjst_no_proper_over-commutative_varieties}
$\vjst$ is the unique cover of $\mathbb{COM}$ in the lattice of all varieties of monoids. 
\end{proposition} 

\begin{proof}
Recall that any over-commutative variety satisfies only balanced identities. Let $\uord \approx \vord$ be a balanced identity not satisfied by $\jst$. Then, $\ol{\uord} \neq \ol{\vord}$, that is, there exist $x,y \in \X$ such that $|\uord|_x = |\uord|_y = 1$ and, without loss of generality, $x$ occurs before $y$ in $\uord$ but after it in $\vord$. Consider the substitution $\phi$ that keeps $x$ and $y$, and maps all other variables to the empty word. Then, $xy = \phi(\uord) \approx \phi(\vord) = yx$ can be deduced from $\uord \approx \vord$. Thus, any monoid in $\vjst$ that does not generate it must be commutative. As such, $\vjst$ admits no proper over-commutative subvarieties (other than $\mathbb{COM}$), hence it is the unique cover of $\mathbb{COM}$ in the lattice of all varieties of monoids, by \cite[Remark~3.10]{gusev_lee_vernikov_lattice_variety_monoids}.
\end{proof}

Thus, any non-commutative monoid in $\vjst$ generates the variety. From this, it is immediate that $\vjst$ is generated by $\jtg$, a non-commutative homomorphic image of $\jst$:

\begin{corollary}
\label{corollary:jst_jtg_same_variety}
    $\jtg$ generates the same variety as $\jst$.
\end{corollary}

On the other hand, since the equational theory of $\hypo$, given in \cite[Theorem~4.1]{cain_malheiro_ribeiro_hypoplactic_2022}, is a strict subset of the equational theory of $\jst$, we have that:

\begin{corollary}
\label{corollary:vjst_strictly_contained_vhypo}
$\vjst$ is strictly contained in $\vhypo$.
\end{corollary}

\section*{Acknowledgements}
The authors would like to thank Alan Cain and António Malheiro for their suggestions, in particular to study the results in Section~\ref{section:Syntacticity_of_plactic_like_monoids}, and many helpful comments.

\bibliographystyle{plain}
\bibliography{Meet_and_join_plactic_like_monoids.bib}
\end{document}